%% file: case2d.tex
\documentclass[english,toc=bib,toc=idx]{scrartcl}

\input{preamble}
\input{math}
\usepackage{fullpage}

\newcommand{\ricardo}[1]{\todo[color=green!40]{#1}}
\newcommand{\ricardoline}[1]{\todo[inline,color=green!40]{#1}}
\newcommand{\najib}[1]{\todo[color=blue!40]{#1}}

\newcommand{\Mod}{\mathrm{Mod}}

\newcommand{\conf}{\mathrm{Conf}}
\newcommand{\Ind}{\mathrm{Ind}}
\newcommand{\C}{\mathbb C}
\newcommand{\BV}{\mathsf{BV}}
\newcommand{\BVGraphs}{\mathsf{BVGraphs}}
\newcommand{\BVc}{\mathsf{BV}^c}
\newcommand{\FFM}{\mathsf{FFM}}
\newcommand{\HH}{\mathbf{M}}
\newcommand{\Mo}{\mathsf{Mo}}
\newcommand{\PaRB}{\mathrm{PaRB}}
\newcommand{\Tot}{\mathrm{Tot}}
\newcommand{\Hom}{\mathrm{Hom}}
\newcommand{\G}{\mathtt{G}}

\newcommand{\sS}{\mathbf{S}}
\newcommand{\BVDelta}{\Delta_{BV}}
\newcommand{\GRT}{\mathrm{GRT}}

\tikzset{every loop/.style={}}

\author{
	Ricardo Campos\thanks{IMAG, Univ. Montpellier, CNRS, Montpellier, France. Email: \href{ricardo.campos@umontpellier.fr}{\nolinkurl{ricardo.campos@umontpellier.fr}}}
	\and
	Najib Idrissi\thanks{Université de Paris, IMJ-PRG, CNRS, F-75013 Paris, France. \href{mailto:najib.idrissi-kaitouni@imj-prg.fr}{\nolinkurl{najib.idrissi-kaitouni@imj-prg.fr}}}
	\and
	Thomas Willwacher\thanks{Department of Mathematics, ETH Zurich, R\"amistrasse 101, 8092 Zurich, Switzerland. Email: \href{thomas.willwacher@math.ethz.ch}{\nolinkurl{thomas.willwacher@math.ethz.ch}}}
}

\newcommand{\mytitle}{Configuration Spaces of Surfaces}
\title{\mytitle}

\begin{document}
	\maketitle
	
	\begin{abstract}
      We compute small rational models for configuration spaces of points on oriented surfaces, as right modules over the framed little disks operad.
      We do this by splitting these surfaces in unions of several handles.
      We first describe rational models for the configuration spaces of these handles as algebras in the category of right modules over the framed little disks operad.
      We then express the configuration spaces of the surface as an ``iterated Hochschild complex'' of these algebras with coefficients in the module given by configurations in a sphere with holes.
      
      Physically, our results may be interpreted as saying that the partition function of the Poisson-$\sigma$-model on closed surfaces has no quantum corrections, i.e., no terms coming from Feynman diagrams of positive loop order.
	\end{abstract}
	
	\tableofcontents
	
	\addsec{Introduction}

    In this paper we study the rational homotopy theory of the configuration spaces $\Conf_{r}(X)$ of $r$ points on surfaces $X$. These spaces are classical objects of topology with a long history.
    Their rational homotopy type has been well understood for along time.
    For example, in the ``local'' case $X=\R^2$, one knows due to work of Arnold \cite{Arnold1969} that the cohomology algebra of the configuration space has the presentation
    \[
    H^*(\Conf_{r}(\R^2)) \cong S(\omega_{ij})_{1 \le i \neq j \le r}
    / (\omega_{ji} - \omega_{ij}, \omega_{ij} \omega_{jk} + \omega_{jk} \omega_{ki} + \omega_{ki} \omega_{ij}),
    \]
    where $\omega_{ij}$ are generators of degree $1$, and $S(-)$ denotes the free graded commutative algebra.
    Furthermore, $\Conf_{r}(\R^2)$ is also formal (i.e.\ its cohomology encodes its rational homotopy type), and at least over $\C$, there is a one-line proof: The map from cohomology to differential forms
    \[
    H^*(\Conf_{r}(\R^2)) \to \Omega(\Conf_{r}(\R^2))
    \]
    that sends $\omega_{ij}$ to $d\log(z_i-z_j)$ is a quasi-isomorphism, where $z_i\in \C\cong \R^2$ it the position of the $i$-th point in a configuration.
    For higher genus (closed) surfaces $X=\Sigma_g$, the cohomology and the rational homotopy type of the configuration spaces are also well understood by work of Bezrukavnikov, Fulton--MacPherson, and others \cite{Bezrukavnikov1994,FultonMacPherson1994}. In this case, the configuration spaces are generally not rationally formal, but small models can be found that we shall recall below.
    
    Going beyond the (rational) homotopy types of the individual configuration spaces, there are however strong algebraic structures that tie them together, for various $r$.
    For example, the configuration spaces $\Conf_r(\R^2)$ can be compactified and assembled into a topological operad $\FM_2$, the Fulton--MacPherson--Axelrod--Singer operad, which is weakly equivalent to the little disks operad.
    One can then ask what the rational homotopy type of the operad $\FM_2$ is. In other words, one wants to enhance Arnold's result so as to also capture the operadic structure. This latter question has been solved much more recently, with the result that the formality still holds operadically \cite{Tamarkin2003, Kontsevich1999,LambrechtsVolic2014,FresseBook}. Note, however, that while the formality as spaces has a one-line proof, the operadic upgrade is a highly nontrivial result, and all known proofs use Drinfeld associators or other complicated pieces of data.
    The formality of the operad $\FM_{2}$ (or $\FM_{n}$ in general) has important consequences, as e.g.\ it allows to express the rational homotopy type of the space of higher-dimension long knots $\overline{\operatorname{Emb}}_{c}(\R^{m}, \R^{n})$ purely combinatorially~\cite{FresseTurchinWillwacher2017}.
    
    The main goal of the present paper is to extend these results to surfaces of higher genera.
    To this end, we will deal with the \emph{framed} configuration spaces $\Conf_r^{fr}(X)$ throughout, for our convenience. Concretely, points of $\Conf_r^{fr}(X)$ are ordered subsets of $r$ points in $X$, with a trivialization of the tangent space (a frame) at each of the points. We will assume that these frames are positively oriented and orthonormal with respect to some given orientation and metric on $X$.
    This means that the frames can be specified by providing a unit vector at each of the $r$ points of the configuration.

	Let $\Sigma_g$ be the closed oriented surface of genus $g$.
	One can then define the (compactified) configuration spaces $\FFM_{\Sigma_g}(r)\simeq \Conf^{fr}_r(\Sigma_g)$ of $r$ framed points in $\Sigma_{g}$.
	They form an operadic right module over the Fulton--MacPherson--Axelrod--Singer version of the framed little disks operad $\FFM_2$.
    Our goal, in this paper, is to find an explicit model of $\FFM_{\Sigma_{g}}$ together with its action of $\FFM_{2}$.
	This model could be used e.g.\ to compute spaces of embeddings of surfaces, or factorization homology of $\FFM_{2}$-algebras over surfaces.

	It is well known that the framed little 2-disks operad is also rationally formal \cite{Severa2010,GiansiracusaSalvatore2010}.
    Its cohomology cooperad is the Batalin--Vilkovisky cooperad $\BVc\coloneqq H^*(\FFM_2)$. It is given by the collection of graded commutative algebras
	\begin{equation}\label{equ:BVcpresentation}
      \BVc(r) = S(\omega_{ij})_{1 \le i \neq j \le r}/ (\omega_{ji} - \omega_{ij}, \omega_{ij} \omega_{jk} + \omega_{jk} \omega_{ki} + \omega_{ki} \omega_{ij}) \otimes S(\theta_{i})_{1 \le i \le r},
	\end{equation}
    where the $\omega_{ij}$ and the $\theta_{i}$ are generators of degree $1$.
	
    Rational dg commutative algebra models of the framed configuration spaces $\FFM_{\Sigma_g}(r)$ are well known~\cite{Bezrukavnikov1994}.
    Let $H^{*}(\Sigma_{g}) = S(a^{k}, b^{k})_{1 \le k \le g} / (a^{k} b^{k} - a^{l} b^{l})_{1 \le k \neq l \le g}$ be the cohomology of the surface.
    Let $\nu$ be its volume form ($\nu = a^{k} b^{k}$ for any $k$), and $\Delta = \nu \otimes 1 + 1 \otimes \nu - \sum_{k} (a^{k} \otimes b^{k} + b^{k} \otimes a^{k})$ be its diagonal class.
    A particularly simple model for $\FFM_{\Sigma_{g}}(r)$ is the following:
	\[
	  \Mo_g(r) = \bigl( H^{*}(\Sigma_{g})^{\otimes r} \otimes \BVc(r)) / (a^{k}_{i} \omega_{ij} - a^{k}_{j} \omega_{ij}, b^{k}_{i} \omega_{ij} - b^{k}_{j} \omega_{ij},\nu_{i} \omega_{ij} - \nu_{j} \omega_{ij}), \; d \omega_{ij} = \Delta_{ij}, d \theta_{i} = (2-2g) \nu_{i} \bigr).
	\]
	Furthermore, the collection of dg commutative algebras $\Mo_g$ has a natural right operadic $\BVc$-comodule structure.
	Concretely, the cooperadic coaction is given by the formulas
    \begin{align*}
      \Mo_{g}(r+s-1)
      & \xrightarrow{\circ_{i}^{*}} \Mo_{g}(r) \otimes \BVc(s) \\
      a^{k}_{j}
      & \mapsto a^{k}_{[j]} \otimes 1
      \\
      b^{k}_{j}
      & \mapsto b^{k}_{[j]} \otimes 1
       \\
      \nu_{j}
      & \mapsto \nu_{[j]} \otimes 1
      \\
      \omega_{jk}
      & \mapsto
        \begin{cases}
          \theta_{i} \otimes 1 + 1 \otimes \omega_{(j-s+1)(k-s+1)} & \text{if } i \le j,k \le i+s-1, \\
          \omega_{[j][k]} \otimes 1 & \text{otherwise},
        \end{cases}
      \\
      \theta_{j}
      & \mapsto
        \begin{cases}
          \theta_{i} \otimes 1 + 1 \otimes \theta_{j-s+1} & \text{if } i \le j,k \le i+s-1, \\
          \theta_{[j]} \otimes 1 & \text{otherwise}.
        \end{cases}
    \end{align*}
    where $[j] \in \{1,\dots,s\}$ is either $j$ if $j \le i-1$, $i$ if $i \le j \le i+s-1$, or $j-s+1$ otherwise.
	We refer to~\cite[Prop.~84]{Idrissi2018b} for more details for the case $g = 0$ which generalizes easily to $g > 0$.
	
	Our main result is then that this cooperadic right coaction is indeed the correct one, modeling the right action of $\FFM_2$ on $\FFM_{\Sigma_g}$.
	
	\begin{thmA}\label{thm:main}
		The pair $(\BVc,\Mo_g)$ consisting of the Batalin--Vilkovisky cooperad $\BVc\coloneqq H^*(\FFM_2)$ and its right cooperadic Hopf comodule $\Mo_g$ is a rational model for the pair $(\FFM_2,\FFM_{\Sigma_g})$.
	\end{thmA}
	
	We note that real models for the pair $(\FFM_2,\FFM_{\Sigma_g})$ have already been described by two of the authors in \cite{CamposWillwacher2016}.
	However, the models in loc. cit. depend on a Maurer--Cartan element in a suitable graph complex, which can be interpreted as the partition function of a topological quantum field theory on the surface.
	Our result can be seen as an evaluation of this partition function. In addition we establish our results over the ground field $\Q$ instead of $\R$.
	
	\subsection{Idea of proof}
	The proof of Theorem \ref{thm:main} consists of three steps.
	\begin{itemize}
		\item First we establish the formality of $\FFM_2$ as a cyclic operad, see Theorem \ref{thm:FFMformality}. This is done by slightly extending the known formality proofs for $\FFM_2$ to cover the cyclic structure as well.
		\item We then use the cyclic model for $\FFM_2$ to build models for the configuration spaces of points on spheres with some punctures.
		From these, the configuration spaces of surfaces can be obtained by gluing the punctures together in pairs. We hence obtain models for $\FFM_{\Sigma_g}$, that resemble a higher Hochschild complex, see section \ref{sec:confspacemodels}.
		\item Finally, we relate the models constructed in the previous step further to $\Mo_g$ above. This last step is performed in section \ref{sec:theproof}, using configuration space integral techniques formally resembling those of \cite{CamposWillwacher2016}.
	\end{itemize}
	
	\subsection*{Acknowledgements}

	The authors have been partially supported by the NCCR SwissMAP funded by the Swiss National Science Foundation, and the ERC starting grant 678156 GRAPHCPX.
    N.I.\ thanks Adrien Brochier for helpful discussions.
	
	\section{Notation and recollections}
	
	\subsection{Generalities}
	Unless explicitly said otherwise, we always work over the ground field $\Q$, so vector spaces are $\Q$-vector spaces, and the cohomology $H(X):=H^*(X)$ and the homology $H_*(X)$ of topological spaces $X$ is assumed to be taken with $\Q$-coefficients.
	We usually work with cohomological conventions, which means that the differential in differential graded (dg) vector spaces has degree +1.
	For $V$ a differential graded vector space, we denote the cohomology of $V$ by $H(V)=H^*(V)$.
	Furthermore, we let $V[k]$ be the same dg vector space, except that the degrees have been shifted by $k$ units. Concretely, a degree $d$ element $v\in V$ has degree $d-k$ in $V[k]$.
	
	We generally assume that the reader is familiar with operads, cooperads and the operadic language in general.
	A good introductory textbook is \cite{LodayVallette2012}.
	We also need to use operadic right modules constantly. A detailed treatment can be found in \cite{FresseModules}, see in particular section 5.1.1 for the definition of operadic right module.
	
	Let us note that there are two different labelling conventions for operations in (co)operads. For an integer $r\geq 0$ we can denote by $\op P(r)$ the $r$-ary operations in the operad $\op P$. Or, we can use finite sets $S$ and write $\op P(S)$ for the operations ``with inputs labelled by $S$''. We shall use both conventions as it fits. The reader can think of the ``finite set'' convention as the fundamental one, with the shorthand $\op P(r)=\op P(\{1,\dots,r\})$.
	
	We abbreviate ``differential graded commutative algebra'' to dgca. A cooperad in the symmetric monoidal category of dgcas is called a Hopf cooperad.
	Generally, we will sometimes abuse the name ``Hopf'' to indicate that we are working over the category of dg commutative algebras.
	
	
	\subsection{Dg Hopf models for topological operads and modules}\label{sec:model_definition}
	
    Due to Sullivan's work~\cite{Sullivan1977}, the rational homotopy theory of topological spaces is encoded by commutative differential graded algebras (dgcas).
    Given a space $X$, Sullivan builds a dgca $\Omega^{*}(X)$ of ``piecewise polynomial forms''.
    One of the main theorems of rational homotopy theory states that two simply connected finite-type spaces $X$ and $Y$ are rationally equivalent (i.e.\ connected by a zigzag of maps that induce isomorphisms on rational homotopy groups) if and only if the dgcas $\Omega^{*}(X)$ and $\Omega^{*}(Y)$ are quasi-isomorphic (i.e.\ connected by a zigzag of morphisms that induce isomorphisms on cohomology).
    Sullivan's theory also works for nilpotent path-connected spaces and was recently extended to arbitrary path-connected spaces~\cite{FelixHalperinThomas2015}.
    A model of the space $X$ is then defined to be any dgca quasi-isomorphic to $\Omega^{*}(X)$.

    Let $\op T$ be a topological or simplicial operad.
    Then the collection $\Omega^{*}(\op T)$ is in general not a Hopf cooperad (i.e.\ a cooperad in the category of dgcas), because the Künneth quasi-isomorphisms go in the wrong direction.
	Suppose that $\op C$ is a dg Hopf cooperad together with quasi-isomorphisms of dgcas 
	\[
      \op C(r) \to \Omega(\op T(r))
	\]
	that respect the (co)unit and such that the following diagrams commute:
    \begin{equation}\label{eq:hopf-model}
      \begin{tikzcd}
        \op C(r+s)\ar{r} \ar{dd}& \Omega(\op T(r+s))\ar{d} \\
        & \Omega(\op T(r+1)\times \op T(s)) \\
        \op C(r+1)\otimes \op C(s) \ar{r} & \Omega(\op T(r+1))\otimes \Omega(\op T(s))\ar{u}{\sim}
      \end{tikzcd}
    \end{equation}
    for each (co)operadic (co)composition. Then we call $\op C$ a dg Hopf cooperad model for $\op T$. Furthermore, we call any dg Hopf cooperad quasi-isomorphic to $\op C$ a dg Hopf cooperad model for $\op T$. 
	
	Suppose furthermore that $\op M$ is an operadic right $\op T$-module, and $\op N$ is a cooperadic right $\op C$-comodule.
	Suppose that there are quasi-isomorphisms of dgcas
	\[
	\op N(r) \to \Omega(\op M(r))
	\]
	such that the following diagrams commute
	\[
      \begin{tikzcd}
        \op N(r+s)\ar{r} \ar{dd}& \Omega(\op M(r+s))\ar{d} \\
        & \Omega(\op M(r+1)\times \op T(s)) \\
        \op N(r+1)\otimes \op C(s) \ar{r} & \Omega(\op M(r+1))\otimes \Omega(\op T(s))\ar{u}{\sim}
      \end{tikzcd}
    \] 
    Then we call the pair $(\op C, \op N)$ a dg Hopf cooperad model for the pair $(\op T, \op M)$.
	
    A more proper rational homotopy theory for topological operads has been developed by B.~Fresse \cite{FresseBook,Fresse2018}, and will be extended to operadic modules in the upcoming work by Fresse--Willwacher \cite{FresseWillwacher2019}.
    In the light of the more proper (and more complicated theory) our simplistic definition of dg Hopf models is justified by \cite[Proposition 2.9]{Fresse2018}
    and the analogous result for modules, which imply that models in our cheap sense are indeed models in the sense of loc. cit.
	Briefly, the functor $\Omega^{*}$ admits an operadic upgrade $\Omega^{*}_{\natural}$ which maps topological (or simplicial) operads to Hopf cooperads.
    By abstract nonsense, a morphism of Hopf cooperads $\op C \to \Omega^{*}_{\natural}(\op T)$ is exactly the same thing as collections of dgca morphisms $\op C(k) \to \Omega^{*}(\op T(k))$ which make the diagrams above commute.
    This functor defines a Quillen adjunction between topological/simplicial operads and Hopf cooperads. Moreover if $\op T$ is a cofibrant operad with finite-type components, then the canonical map $\Omega^{*}_{\natural}(\op T)(k) \to \Omega^{*}(\op T(k))$ is a quasi-isomorphism.
    Similarly, the functor $\Omega^{*}$ admits an upgrade (also denoted $\Omega^{*}_{\natural}$) which maps topological/simplicial right modules to dg Hopf comodules, and the rational homotopy theory can be extended analogously \cite{FresseWillwacher2019}.

	In any case, we call the operad or module rationally formal if the cohomology cooperad or comodule is a rational dg Hopf model.
	Furthermore, we will analogously define the notion of dg Hopf model for algebra objects in right $\op T$-modules, and for modules thereof, see the discussion in section \ref{sec:Bmodules} below.


	\subsection{Fulton--MacPherson operad}

    In order to recover operadic structures, we consider the Fulton--MacPherson--Axelrod--Singer compactifications of configuration spaces~\cite{AxelrodSinger1994,FultonMacPherson1994}, rather than configuration spaces themselves.
    We refer to~\cite{Sinha2004} and~\cite[Sections~5.1--5.2]{LambrechtsVolic2014} for detailed treatments.

    Recall that if $X$ is a space, then its $r$th (ordered) configuration space is $\Conf_{r}(X) \coloneqq \{ x \in X^{r} \mid \forall i \neq j, \; x_{i} \neq x_{j} \}$.
    Even if $X$ is compact, the spaces $\Conf_{r}(X)$ are generally not compact for $k \ge 2$.
    In this paper, we will mainly consider the compactification for $X = \R^{2}$, although we will also mention $X = \Sigma_{g}$ in Proposition~\ref{prop:thinmodel} and Section~\ref{sec:recoll-citec-anoth}.

    Let us first consider $X = \R^{2}$.
    The space $\Conf_{r}(\R^{2})$ embeds into $(S^{1})^{r(r-1)} \times [0, +\infty]^{r(r-1)(r-2)}$ as follows :
    \begin{itemize}
    \item For $i \neq j$, the component $\theta_{ij} : \Conf_{r}(\R^{2}) \to S^{1}$ is given by $\theta_{ij}(x) = (x_{i} - x_{j}) / \| x_{i} - x_{j} \|$.
    \item For $i \neq j \neq k \neq i$, the component $\delta_{ijk} : \Conf_{r}(\R^{2}) \to [0,+\infty]$ is given by $\delta_{ijk}(x) = \|x_{i} - x_{k}\| / \|x_{j} - x_{k}\|$.
    \end{itemize}
    The Fulton--MacPherson compactification $\FM_{2}(r)$ is the closure of the image of this embedding.
    It is a compact manifold with corners of dimension $2k - 3$.
    The interior of $\FM_{2}(r)$ is the quotient of $\Conf_{r}(\R^{2})$ by the action of the group of translations and positive rescalings $\R^{2} \rtimes \R_{>0}$.
    Intuitively, a point of the boundary of $\FM_{2}(r)$ can be seen as a virtual configuration in $\R^{2}$ where some of the points are infinitesimally close.
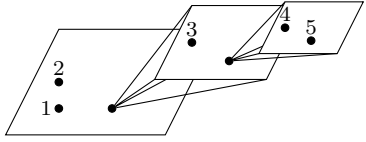
\[
\begin{tikzpicture}[baseline=1cm,scale=.7,int1/.style={draw, circle,fill,inner sep=1pt},label distance=-1mm]
\draw (0,0)--(3,0)--(4,2)--(1,2)--cycle;
\node [int1, label=180:{$\scriptstyle 1$}] (v1) at (1,.5) {};
\node [int1, label={$\scriptstyle 2$}] (v2) at (1,1) {};
\node [int1] (va) at (2,.5) {};
\begin{scope}[scale = .7, xshift=4cm,yshift=1.5cm]
\draw (va) -- (0,0) (va) -- (3,0) (va)-- (4,2) (va) -- (1,2);
\draw[fill=white] (0,0)--(3,0)--(4,2)--(1,2)--cycle;
\node [int1, label={$\scriptstyle 3$}] (v3) at (1,1) {};
\node [int1] (vb) at (2,.5) {};
\begin{scope}[scale = .7, xshift=4cm,yshift=1cm]
\draw (vb) -- (0,0) (vb) -- (3,0) (vb)-- (4,2) (vb) -- (1,2);
\draw[fill=white] (0,0)--(3,0)--(4,2)--(1,2)--cycle;
\node [int1, label={$\scriptstyle 4$}] (v4) at (1,1) {};
\node [int1, label={$\scriptstyle 5$}] (v5) at (2,.5) {};
\end{scope}
\end{scope}
\end{tikzpicture}
\]

    There is an operad structure on the collection $\FM_{2} = \{ \FM_{2}(r) \}_{r \ge 0}$, which can be defined by explicit formulas~\cite[Section 5.2]{LambrechtsVolic2014}.
    This operad is weakly homotopy equivalent to the classical little disks operad.
    In particular, its homology operad is the operad encoding Gerstenhaber algebras. We also note that analogous definitions can be made starting from configurations in $\R^n$ instead of $\R^2$, thus giving rise to operads $\FM_n$, weakly equivalent to the little $n$-disks operads.

    The special orthogonal group $\mathrm{SO}(2)$ acts on $(S^{1})^{r(r-1)} \times [0, +\infty]^{r(r-1)(r-2)}$ by rotating the circles, and this action restricts to $\FM_{2}(r)$.
    This restricted action is compatible (in a sensible way~\cite{SalvatoreWahl2003}) with the operad structure of $\FM_{2}$, which allows us to define the framed Fulton--MacPherson operad as the semidirect product $\FFM_{2} = \FM_{2} \rtimes \mathrm{SO}(2)$, given in arity $r$ by $\FFM_{2}(r) = \FM_{2}(r) \times \mathrm{SO}(2)^{r}$ and an operad structure inherited from the operad structure of $\FM_{2}$, the group structure of $\mathrm{SO}(2)$, and the action of $\mathrm{SO}(2)$ on $\FM_{2}(r)$.
    This operad is weakly equivalent to the classical framed little disks operad.
    In particular, its homology is $\BV$, the operad encoding Batalin--Vilkovisky-algebras.

    Let us consider the genus $g$ surface $X = \Sigma_{g}$.
    We can embed $\Sigma_{g}$ as an algebraic submanifold of some $\R^{N}$.
    We can then embed $\Conf_{r}(\Sigma_{g})$ into $(\Sigma_{g})^{r} \times (S^{N-1})^{r(r-1)} \times [0,+\infty]^{r(r-1)(r-2)}$ as follows:
    \begin{itemize}
    \item the factor $\Conf_{r}(\Sigma_{g}) \to (\Sigma_{g})^{r}$ is the inclusion;
    \item the factors $\theta_{ij} : \Conf_{r}(\Sigma_{g}) \to S^{N-1}$ and $\delta_{ijk} : \Conf_{r}(\Sigma_{g}) \to [0,+\infty]$ are defined similarly to the maps $\theta_{ij}$ and $\delta_{ijk}$ above, using the embedding $\Sigma_{g} \subset \R^{N}$.
    \end{itemize}
    The Fulton--MacPherson compactification $\FM_{\Sigma_{g}}(r)$ is then closure of the image of this embedding.
    It is a manifold with corners of dimension $2r$ and a compact semi-algebraic set.
    Its interior is $\Conf_{r}(\Sigma_{g})$, and its boundary intuitively consists of configuration of points on $\Sigma_{g}$ which can become infinitesimally close.
    The following picture illustrates a boundary point in $\FM_{\Sigma_{2}}(5)$.
\[
\begin{tikzpicture}[baseline=1cm, scale=1,int1/.style={draw, circle,fill,inner sep=1pt}]
\draw (0,0)
.. controls (-.7,0) and (-1.3,-.7) .. (-2,-.7)
.. controls (-4,-.7) and (-4,1.7) .. (-2,1.7)
.. controls (-1.3,1.7) and (-.7,1) .. (0,1)
.. controls  (.7,1) and (1.3,1.7) .. (2,1.7)
.. controls (4,1.7) and (4,-.7) .. (2,-.7)
.. controls (1.3,-.7) and (.7,0)  .. (0,0);
\begin{scope}[xshift=-2cm, yshift=.6cm, scale=1.2]
\draw (-.5,0) .. controls (-.2,-.2) and (.2,-.2) .. (.5,0);
\begin{scope}[yshift=-.07cm]
\draw (-.35,0) .. controls (-.1,.1) and (.1,.1) .. (.35,0);
\end{scope}
\end{scope}
\begin{scope}[xscale=-1, xshift=-2cm, yshift=.6cm, scale=1.2]
\draw (-.5,0) .. controls (-.2,-.2) and (.2,-.2) .. (.5,0);
\begin{scope}[yshift=-.07cm]
\draw (-.35,0) .. controls (-.1,.1) and (.1,.1) .. (.35,0);
\end{scope}
\end{scope}
\node [int1, label={$\scriptstyle 1$}] (v1) at (-3,.5) {};
\node [int1, label={$\scriptstyle 2$}] (v2) at (2,-.1) {};
\node [int1] (va) at (0,.5) {};
\node [int1] (va) at (0,.5) {};
	\begin{scope}[scale = .7, xshift=-1cm,yshift=2.75cm]
	\draw (va) -- (0,0) (va) -- (3,0) (va)-- (4,2) (va) -- (1,2);
	\draw[fill=white] (0,0)--(3,0)--(4,2)--(1,2)--cycle;
	\node [int1, label={$\scriptstyle 3$}] (v3) at (1,1) {};
	\node [int1] (vb) at (2,.5) {};
	\begin{scope}[scale = .7, xshift=2cm,yshift=2cm]
	\draw (vb) -- (0,0) (vb) -- (3,0) (vb)-- (4,2) (vb) -- (1,2);
	\draw[fill=white] (0,0)--(3,0)--(4,2)--(1,2)--cycle;
	\node [int1, label={$\scriptstyle 4$}] (v4) at (1,1) {};
	\node [int1, label={$\scriptstyle 5$}] (v5) at (2,.5) {};
	\end{scope}
	\end{scope}
\end{tikzpicture}
\]

    The surface $\Sigma_{g}$ is not parallelizable for $g \neq 1$, and therefore the collection $\FM_{\Sigma_{g}}$ does not have the structure of a right $\FM_{2}$-module.
    However if we fix a metric and orientation on $\Sigma_g$, then we can consider the oriented orthonormal frame bundle $F_{\Sigma_{g}} \to \Sigma_{g}$, which is a nontrivial $\mathrm{SO}(2)$-bundle for $g \neq 1$.
    We define the framed Fulton--MacPherson compactification of $\FM_{\Sigma_{g}}$ as the pullback:
    \[
      \begin{tikzcd}
        \FFM_{\Sigma_{g}}(r) \ar[r, dashed] \ar[d, dashed] \ar[dr, phantom, "\lrcorner"] & (F_{\Sigma_{g}})^{r} \ar[d] \\
        \FM_{\Sigma_{g}}(r) \ar[r, hook] & (\Sigma_{g})^{r}
      \end{tikzcd}
      .
    \]

    Intuitively, a point in $\FFM_{\Sigma_{g}}(r)$ consists of a (possibly virtual) configuration in $\Sigma_{g}$, together with a trivialization of the tangent bundle of $\Sigma_{g}$ at each point of the configuration.
    Then the collection $\FFM_{\Sigma_{g}}$ has the structure of a right module over the operad $\FFM_{2}$.
    The structure maps can be roughly interpreted as follows.
    Given $x \in \FFM_{\Sigma_{g}}(r)$ and $y \in \FFM_{2}(s)$, the element $x \circ_{i} y \in \FFM_{\Sigma}(r+s-1)$ is obtained using the given trivialization $T_{x_{i}}\Sigma_{g} \cong \R^{2}$ to view the configuration $y$ as a configuration in the tangent space $T_{x_{i}}\Sigma_{g}$, which is then inserted into $x$.
    The trivializations at the $y_{j}$ are obtained from the trivialization at $x_{i}$, rotated by the element of $\mathrm{SO}(2)$ attached to $y_{j}$ in $\FFM_{2}(s)$.

	Note also that on $\FFM_{\Sigma_g}$ one has natural forgetful maps
	\begin{align}\label{equ:piFFMSigma}
	\pi_{i} : \FFM_{\Sigma_g}(r) &\to \FFM_{\Sigma_g}(1)
	&
	\pi_{ij} : \FFM_{\Sigma_g}(r) &\to \FFM_{\Sigma_g}(2)
	\end{align}
	by forgetting all but the points $i$ (respectively $i$ and $j$) from the configuration of points.
	
	\section{Modules}
	\label{sec:Bmodules}
	Let $\op T$ be an operad in some symmetric monoidal category, which for us will usually be topological spaces or dg vector spaces.
	Let $\Mod_{\op T}$ be the category of operadic right $\op T$-modules. Then $\Mod_{\op T}$ is itself equipped with a symmetric monoidal product, given by the formula 
	\[
	(\op M \otimes \op N)(r) = \coprod_{p+q=r} \Ind_{S_p\times S_q}^{S_r} \op M(p) \otimes \op N(q).
	\]
	We can hence talk about algebra objects in $\Mod_{\op T}$, and modules over these algebra objects.
	More concretely, let $A\in \Mod_{\op T}$ be such an algebra object, then we may consider left $A$-modules, right $A$-modules, $A$-bimodules, or more generally $A^{(p,q)}\coloneqq A^{\otimes p}\otimes (A^{op})^{\otimes q}$-modules in the category $\Mod_{\op T}$.
	
	\begin{example}\label{ex:conf_space_modules}
		Let the operad $\op T$ be the (oriented) framed little $n$-disks operad.
		Let $M$ be an $n$-dimensional oriented manifold with boundary $N=\partial M$.
		We fix some collar of $N$ in $M$.
		For $s>0$ we denote by $M^s$ the manifold obtained by adding a cylinder $N\times (0,s)$ to our manifold at the boundary.
		We denote by $\conf(M)$ the space of embeddings of $n$-disks in $M$, considered as a right $\op T$-module.
		Next define an
		algebra object $A\in \Mod_{\op T}$ whose elements are pairs $(s,C)$, where $s\geq 0$ is a number and $C\in \conf((0,s) \times N)$ is a configuration in the cylindrical extension of $N$ of length $s$. The algebra structure is obvious, one just glues two configurations in cylinders of length $s$ and $t$ to a configuration in a cylinder of length $s+t$.
		Finally we consider the right $A$-module $\op M$ whose elements are pairs $(s,C)$ with $s>0$ a number and $C\in \conf(M^s)$ a configuration of disks in the version $M^s$ of $M$ with collar of length $s$ as above.
        The $A$-module structure is given (obviously) by gluing the cylinder to the collar, see Figure~\ref{fig:a-module-struct}.
    \end{example}

    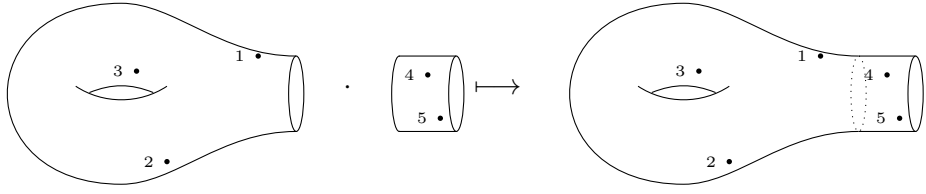
\begin{figure}[htbp]
      \centering
      \tikzset{vtx/.style={circle,fill,inner sep=0,minimum size=2pt}}
      \begin{tikzpicture}[baseline=.5cm]
        \draw (.3,0)
        .. controls (-.7,0) and (-1.3,-.7) .. (-2,-.7)
        .. controls (-4,-.7) and (-4,1.7) .. (-2,1.7)
        .. controls (-1.3,1.7) and (-.7,1) .. (.3,1);
        \begin{scope}[xshift=-2cm, yshift=.6cm, scale=1.2]
          \draw (-.5,0) .. controls (-.2,-.2) and (.2,-.2) .. (.5,0);
          \begin{scope}[yshift=-.07cm]
            \draw (-.35,0) .. controls (-.1,.1) and (.1,.1) .. (.35,0);
          \end{scope}
        \end{scope}
        \draw (.3,.5) ellipse (.1cm and .5cm);
        \node[vtx, label=left:{\tiny 1}] at (-.2,1) {};
        \node[vtx, label=left:{\tiny 2}] at (-1.4,-.4) {};
        \node[vtx, label=left:{\tiny 3}] at (-1.8,.8) {};
      \end{tikzpicture}
      \quad $\cdot$
      \begin{tikzpicture}[baseline=.5cm]
        \begin{scope}
          \clip (0,-.1) rectangle (-.5,1.1);
          \draw (0,.5) ellipse (.1cm and .5cm);
        \end{scope}
        \draw (.75,.5) ellipse (.1cm and .5cm);
        \draw (0,0)--(.75,0) (0,1)--(.75,1);
        \node[vtx, label=left:{\tiny 4}] at (.375,.75) {};
        \node[vtx, label=left:{\tiny 5}] at (.54,.175) {};
      \end{tikzpicture}
      $\longmapsto$
      \begin{tikzpicture}[baseline=.5cm]
        \draw (.3,0)
        .. controls (-.7,0) and (-1.3,-.7) .. (-2,-.7)
        .. controls (-4,-.7) and (-4,1.7) .. (-2,1.7)
        .. controls (-1.3,1.7) and (-.7,1) .. (.3,1);
        \begin{scope}[xshift=-2cm, yshift=.6cm, scale=1.2]
          \draw (-.5,0) .. controls (-.2,-.2) and (.2,-.2) .. (.5,0);
          \begin{scope}[yshift=-.07cm]
            \draw (-.35,0) .. controls (-.1,.1) and (.1,.1) .. (.35,0);
          \end{scope}
        \end{scope}
        \node[vtx, label=left:{\tiny 1}] at (-.2,1) {};
        \node[vtx, label=left:{\tiny 2}] at (-1.4,-.4) {};
        \node[vtx, label=left:{\tiny 3}] at (-1.8,.8) {};
        \begin{scope}[xshift=.3 cm]
          \draw[dotted] (0,.5) ellipse (.1cm and .5cm);
          \draw (.75,.5) ellipse (.1cm and .5cm);
          \draw (0,0)--(.75,0) (0,1)--(.75,1);
          \node[vtx, label=left:{\tiny 4}] at (.375,.75) {};
          \node[vtx, label=left:{\tiny 5}] at (.54,.175) {};
        \end{scope}
      \end{tikzpicture}
      \caption{The $A$-module structure}
      \label{fig:a-module-struct}
    \end{figure}
	
	We note that notationally there is a potential source of confusion: a module $M$ over an algebra object $A\in \Mod_{\op T}$ has two different kinds of module structure, namely the right $\op T$-module structure, and the $A$-module structure. We will refer to the latter structure as ``boundary module'', since all our examples will be modeled on the previous one, where that module structure is given by gluing at the boundary.
	
	\begin{example}\label{ex:mod_from_op}
		Given the operad $\op T$ we may build an algebra object $\op T_1$ in $\Mod_{\op T}$ such that
		\[
		\op T_1(r) = \op T(r+1). 
		\]
		The right $\op T$-action is then by operadic composition on the first $r$ ``inputs''. The algebra product is the operadic composition on the last input.
		
		We may similarly define a $\op T_1^{(1,p)}$-module (i.e., carrying a left action and $p$ right actions of $\op T_1$) $\op T_p$ such that 
		\[
		\op T_p(r) = \op T(r+p),
		\]
		with the right $\op T$-action operating on the first $r$ inputs, and the $p$ right $\op T_1$-actions operating on the others. 
	\end{example}
	
	\begin{remark}\label{rem:involution}
		Suppose that $N$ is an oriented $n-1$ dimensional manifold. We can then define an algebra object $A$ in right modules over the  framed little disks operad as in Example \ref{ex:conf_space_modules} above.
		Next suppose that in addition $N$ has an orientation reversing automorphism $I_N: N\to N^{op}$.
		Then we may build an isomorphism $A\xrightarrow{\cong} A^{op}$ which sends a point $(s,C)\in A$ to $(s,  I_{(0,s)} \times I_N  (C))$, where $I_{(0,s)}:(0,s)\to (0,s)$ just flips the interval around its midpoint.
		
		This identification of $A$ and $A^{op}$ in particular allows us to construct right modules over $A$ from left modules and vice versa.
		We will often use this construction below. In our case $N=S^1$ and $I_N$ is a reflection around one axis.
	\end{remark}
	
	The above definitions dualize without difficulty to the cooperadic context. In particular, if $\op C$ is a dg (Hopf) cooperad, we can consider the category of right $\op C$-comodules $\Mod^c_{\op C}$.
	It is again symmetric monoidal, and we will talk about coalgebra objects and their comodules in $\Mod^c_{\op C}$.

	\begin{remark}
		The data of an algebra $A$ in right $\op T$-modules is equivalent to the data of a $\op T$-module operad, or $\op T$-moperad~\cite[Def.~9]{Willwacher2016}.
		These moperads can be used to define the notion of ``$A$-shaped $\op T$-modules''~\cite[Def.~3.1]{Horel2014}, a notion related to (but different from) the notion of left $A$-modules we consider above.
	\end{remark}
	
	\subsection{Derived tensor products of boundary modules and gluing of manifolds}
	Suppose that $\op T$ is a topological operad and $A\in \Mod_{\op T}$ is an algebra object. Let $M\in \Mod_{\op T}$ be a right $A$-module and $N\in \Mod_{\op T}$ be a left $A$-module. 
    Then we consider the bar construction
	\[
	B(M,A,N) = M \otimes A^{\otimes \bullet} \otimes N,
	\]
	which is a simplicial object in $\Mod_{\op T}$. The simplicial boundary and degeneracy maps are given by the product and module structures and the unit of $A$ as usual.
	Then we define the derived relative tensor product $M\hotimes_A N$ as the fat realization of $B(M,A,N)$, i.e., as the coend
	\[
	 M\hotimes_A N
	 =
	 |B(M,A,N)|_+
	 =
	 \int^{[n]\in \Delta^+}
	 M \otimes A^{\otimes n} \otimes N \otimes \Delta^n.
	\]
	We note that this is defined even for non-unital $A$.
	
	Next, let $M$ be an $A$-bimodule. Then we define a simplicial object in $\Mod_{\op T}$
	\[
	B(M,A) = M\otimes A^{\otimes \bullet},
	\] 
	where the simplicial maps are defined as for $B(M,N)$ before, except that now the last boundary map multiplies the last factor $A$ into $M$ using the left module structure. In other words, this is the simplicial version of the Hochschild complex of $A$ with coefficients in the bimodule $M$. We then denote the fat realization by:
	\[
	\hotimes_A M = |B(M,A)|_+.
	\]
	
	Finally, suppose that $M$ is an $A^{(p,q)}$-module, i.e., it carries $p$ left actions of $A$ and $q$ right actions, that all commute. 
	Then we can do the above construction with respect to the $i$-th left and the $j$-th right module structure (with $1\leq i\leq p$, $1\leq j\leq q$), resulting in an $A^{(p-1,q-1)}$-module which we denote by
	\[
	\hotimes_A^{(i,j)} M.
	\]
	We can furthermore iterate the construction and pair off multiple, say $k$, of the left- and right- module structures to form an $A^{(p-k,q-k)}$-module
	\[
	\hotimes_A^{(i_1,j_1)(i_2,j_2)\dots (i_k,j_k)} M.
	\]
	Here $1\leq i_1,\dots,i_k \leq p$ are pairwise distinct, as well as $1\leq j_1,\dots,j_k \leq q$, and the $i_1$-st left module structure is paired with the $j_1$-st right module structure etc. This should be seen as a version of a higher Pirashvili--Hochschild complex.
	
	We are interested in the derived tensor product mostly in the context of Example \ref{ex:conf_space_modules}.
	Here one has the following result.
	\begin{lemma}
		Suppose that $M_1$ and $M_2$ are oriented manifolds with boundary $\partial M_1=N$, $\partial M_2=N^{op}$.
		Consider the algebra object $A$ of Example \ref{ex:conf_space_modules}, and the right and left $A$-modules $\op M_1$, $\op M_2$ as in that example.
		Then the right module over the little disks operad $\op M_1\hotimes_A \op M_2$ is homotopy equivalent to $\conf(M_1 \sqcup_N M_2)$.
		
		The analogous statements also hold for self-gluings and multiple gluings, in case $M$ has multiple distinct boundary components. 
	\end{lemma}	
	\begin{proof}
		One convinces oneself that points in $\op M_1\hotimes_A \op M_2$ are the same as configurations of disks on $M_1 \sqcup_N ((0,1)\times N)\sqcup_N M_2$ with additional data as to how the connecting cyclinder is extended and chopped up in pieces. This additional data is in a contractible space and hence one arrives at the result.
	\end{proof}
	
	The constructions of this section also hold if we replace topological spaces by simplicial sets.

	\subsection{Construction: Configuration spaces of surfaces from the little disks operad}\label{sec:conf_from_FFM}
	In this subsection we will apply the above constructions to build right modules over the framed little 2-disks operad that are weakly equivalent to configuration spaces of oriented surfaces. We will work with the compactified version $\FFM_2$ of the framed little disks operad.
	Our construction proceeds in 4 steps.
	
	First we build an algebra object $\FFM_C\in \Mod_{\FFM_2}$ that is homotopy equivalent to configurations of disks on the cylinder $(0,1)\times S^1$. Concretely, we set
	\[
	\FFM_C(r) = \fiber(\FFM_2(r+1)\to \FFM_2(1)=S^1)
	\]
	to be the fiber of the forgetful map, forgetting all but one point. This is the same as a sub-algebra of the algebra object $(\FFM_2)_1$ from example \ref{ex:mod_from_op}, obtained by fixing the framing at the distinguished input.
	
	Second, we build $\FFM_C^{(1,q)}$-modules $\FFM_{2}^{(1,q)}$, that are homotopy equivalent to configurations of disks on a sphere with $q+1$ boundary components:
	\begin{equation}
	\label{equ:FFM21qdef}
	\FFM_{2}^{(1,q)}(r) = \fiber(\FFM_2(r+q)\to \FFM_2(q)).
	\end{equation}
	Again this is a submodule of the module $(\FFM_2)_q$ constructed in example \ref{ex:mod_from_op}.
	The fiber here can be taken over an arbitrary basepoint.
	We shall take the liberty to set the basepoint later, depending on the desired application.
	
	Third, we build $\FFM_C^{(1+p,q)}$-modules $\FFM_{2}^{(1+p,q)}$ by starting with $\FFM_{2}^{(1,p+q)}$ from the second step and re-interpreting the right $\FFM_C$-action on $p$ of the distinguished inputs as a left action via the isomorphism
	\begin{equation}\label{equ:IFMC}
	I : \FFM_C \cong \FFM_C^{op}
	\end{equation}
	from Remark \ref{rem:involution}.
	Concretely, the above identification of $\FFM_C$ with its opposite is induced from the orientation preserving involution $z\mapsto \frac 1 z$ of $\C$, if we fix the marked point to be at the origin.
	
	Finally, in the fourth step we use the derived tensor products to build objects of $\Mod_{\FFM_2}$
	\begin{equation}\label{equ:ourFMg}
	\hotimes_{\FFM_C}^{(1,1)(2,2)\dots (g,g)}\FFM_{2}^{(1+(g-1),g)}
	\end{equation}
	that are homotopy equivalent to the configuration spaces of framed points $\FFM_{\Sigma_g}$ on an oriented surface of genus $g$, see Figure~\ref{fig.conf-surface}.

    \begin{figure}[htbp]
      \centering
      \begin{tikzpicture}
        \draw (0,0) circle[radius = 2];
        \draw
        (0,1.5)
        arc[x radius=.15, y radius=.2, start angle=270, end angle=90] node (x) {}
        arc[y radius=.7, x radius = 1, start angle=90, end angle=-90]
        arc[x radius=.15, y radius=.2, start angle=270, end angle=90] node(y) {}
        arc[y radius=.3, x radius=.6, start angle=-90, end angle=90];
        \draw[dashed] (x) arc[x radius=.15, y radius=.2, start angle=90, end angle=-90];
        \draw[dashed] (y) arc[x radius=.15, y radius=.2, start angle=90, end angle=-90];
        \node at (0,0) {$\vdots$};
        \draw[fill=white]
        (0,-.9)
        arc[x radius=.15, y radius=.2, start angle=270, end angle=90] node (x) {}
        arc[radius=.7, x radius= 1, start angle=90, end angle=-90]
        arc[x radius=.15, y radius=.2, start angle=270, end angle=90] node(y) {}
        arc[radius=.3, x radius = .6, start angle=-90, end angle=90];
        \draw[dashed] (x) arc[x radius=.15, y radius=.2, start angle=90, end angle=-90];
        \draw[dashed] (y) arc[x radius=.15, y radius=.2, start angle=90, end angle=-90];
        \draw[->] (-3,1) node[above left] {$\FFM_{2}^{(g,g)}$} -- (-1,0);
        \draw[->] (3,1.5) node[above right] {action of $\FFM_{C}$ on $(1,1)$} -- (.8,1.2);
        \draw[->] (3,-.5) node[above right] {action of $\FFM_{C}$ on $(g,g)$} -- (.8,-1);
      \end{tikzpicture}
      \caption{Our version of (the configuration space on) a genus $g$ surface is starting with (configurations on) a sphere with marked points, and then gluing ``thin'' handles at those points. The latter is algebraically represented by taking the higher-Hochschild-complex type construction \eqref{equ:ourFMg}.  }
      \label{fig.conf-surface}
    \end{figure}
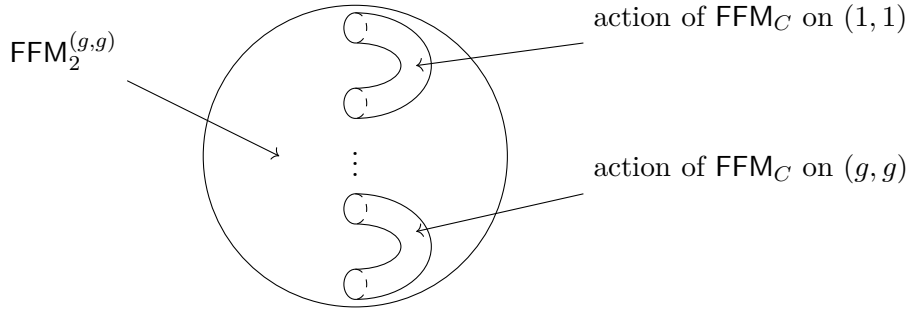
	
	\subsection{Dual construction for Hopf boundary comodules}
    \label{sec:dual-constr-hopf}

	Let $\op C$ be a dg Hopf cooperad and $B$ a coalgebra in right $\op C$-comodules.
	Let $M$ be a $B^{(p,q)}$-comodule. Then, dualizing the constructions above, we can define a $B^{(p-1,q-1)}$-comodule:
	\[
	\hotimes^{(i,j)}_B M \coloneqq \Tot \bigl( M\otimes B^{\otimes \bullet}\otimes \Omega(\Delta^\bullet) \bigr)
	\]
	as the fat totalization of the cosimplicial object on the right.
    Concretely, the collection $\{ M \otimes B^{\otimes k} \otimes \Omega(\Delta^{l})\}_{k,l \ge 0}$ forms a semi-simplicial object in semi-cosimplicial dg-modules.
    If $\Delta_{+} \subset \Delta$ is the semi-simplicial category (of injective nondecreasing maps), then $\hotimes^{(i,j)}_B M$ is defined as the end $\int_{[n] \in \Delta_{+}} M \otimes B^{\otimes n} \otimes \Omega(\Delta^{n})$.
    While the standard totalization of cosimplicial objects is not homotopy invariant, this fat totalization is homotopy invariant: $\Omega(\Delta^{\bullet})$ is a fibrant simplicial object, thus $\int_{[n] \in \Delta_{+}} (-)^{n} \otimes \Omega(\Delta^{\bullet})$ preserves weak equivalences of cofibrant semi-cosimplicial objects, but every semi-cosimplicial dg-module is cofibrant.
	
	\section{Ribbon braids and rational model of configuration space of surfaces}
	
	The goal of this section is to construct a rational dg Hopf comodule model for the configuration spaces of framed points on surfaces, essentially by dualizing the construction of section \ref{sec:conf_from_FFM}, see Proposition \ref{prop:thinmodel} below.
	This model will then later be simplified in section \ref{sec:theproof} to show our main result Theorem \ref{thm:main}.
	
	\subsection{Rational cyclic formality of the framed little disks operad}
	
	The formality of the framed little 2-disks operad is a well known fact.
    It has been shown in the dg Hopf setting and over $\R$ by Giansiracusa--Salvatore \cite{GiansiracusaSalvatore2010}, and for the chains operads over $\Q$ by \v Severa \cite{Severa2010}. 
	There are also statements about the formality as a cyclic operad (over $\R$) in \cite{GiansiracusaSalvatore2012}, however with a mistake, and a reference to an unpublished proof by \v Severa, which we likely reproduce here.
	In any case, we need the following formality theorem, over $\Q$ and in the cyclic dg Hopf setting, slightly extending previous results in the literature.
	
	\begin{theorem}[Formality of the framed little disks operad, extending \cite{GiansiracusaSalvatore2010, Severa2010}]
		\label{thm:FFMformality}
		The framed little 2-disks operad is rationally formal as a cyclic operad, i.e., the cyclic cohomology cooperad $H^*(\FFM_2)$ is a rational dg Hopf model for the cyclic operad $\FFM_2$. 
	\end{theorem}
	The proof is a version of the proof of \v Severa \cite{Severa2010} for the chains statement, with some modification along the lines of Fresse's proof \cite[section 14]{FresseBook} of the analogous statement for the non-framed little disks operad.
	
	To begin with, let us recall some relevant objects and facts from those proofs.
	First, we have an operad map $\FM_1\to \FFM_2$, where the image are configurations supported on a line (say the real axis), and all framings are appropriately aligned.
	The spaces $\FM_1(r)$ are the Stasheff associahedra, and the set of corners of those associahedra can be identified with the set of rooted planar binary trees with leaves numbered $1,\dots,r$. 
    
	The operad of parenthesized ribbon braids $\PaRB(r)$ is the fundamental groupoid of $\FFM_2(r)$ with basepoints the set of corners of the associahedra $\FM_1(r)\subset \FFM_2(r)$.
    The groupoids $\PaRB(r)$ assemble into an operad in groupoids $\PaRB$.
	More concretely, as an operad in groupoids, $\PaRB$ is generated by the three morphisms $R_{1} = R \in \PaRB(1), X_{12} = X \in \PaRB(2), A_{123} = A \in \PaRB(3)$, see Figure~\ref{fig:gen-parb}.
    \begin{figure}[htbp]
      \centering
      $R = $
      \begin{tikzpicture}[scale=.33, baseline=.5cm]
        \draw (0,0) -- (3,0);
        \draw (0,4) -- (3,4);
        \draw[fill=lightgray]
        (1,0)
        .. controls (1,1) and (2,1) ..
        (2,2)
        .. controls (2,3) and (1,3) ..
        (1,4) -- (2,4) node[midway, above] {1}
        .. controls (2,3) and (1,3) ..
        (1,2)
        .. controls (1,1) and (2,1) ..
        (2,0) -- (1,0) node[midway, below] {1};
        \begin{scope}
          \clip (1.4, .7) rectangle (1.6,1.3);
          \draw[lightgray,double=black,double distance=.4pt,thick]
          (1,0)
          .. controls (1,1) and (2,1) ..
          (2,2);
        \end{scope}
        \begin{scope}
          \clip (1.4, 2.7) rectangle (1.6,3.3);
          \draw[lightgray,double=black,double distance=.4pt,thick]
          (2,4)
          .. controls (2,3) and (1,3) ..
          (1,2);
        \end{scope}
      \end{tikzpicture}
      \qquad
      $X = $
      \begin{tikzpicture}[scale=.33, baseline=.5cm]
        \draw (0,0) -- (5,0);
        \draw (0,4) -- (5,4);
        \draw[fill=lightgray]
        (3,0) -- (4,0) node[midway, below] {1}
        -- (2,4) -- (1,4) node[midway, above] {1}
        -- (3,0);
        \draw[fill=lightgray]
        (1,0) -- (2,0) node[midway, below] {2}
        -- (4,4) -- (3,4) node[midway, above] {2}
        -- (1,0);
      \end{tikzpicture}
      \qquad
      $A = $
      \begin{tikzpicture}[scale=.33, baseline=.5cm]
        \draw (0,0) -- (9,0);
        \draw (0,4) -- (9,4);
        \draw[fill=lightgray] (1,4) -- (2,4) node[midway, above] {1}
        -- (2,0) -- (1,0) node[midway, below] {1}
        -- (1,4);
        \draw[fill=lightgray] (5,4) -- (6,4) node[midway, above] {2}
        -- (4,0) -- (3,0) node[midway, below] {2}
        -- (5,4);
        \draw[fill=lightgray] (7,4) -- (8,4) node[midway, above] {3}
        -- (8,0) -- (7,0) node[midway, below] {3}
        -- (7,4);
      \end{tikzpicture}
      \caption{Generators of $\PaRB$}
      \label{fig:gen-parb}
    \end{figure}
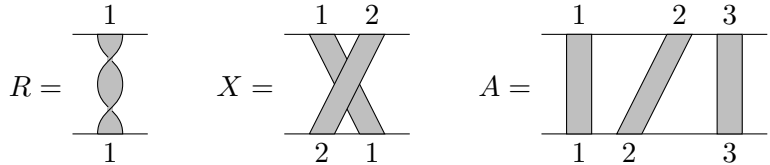

	The operads $\FM_1$ and $\FFM_2$ have cyclic structures, and the inclusion is a map of cyclic operads.
	Concretely, the cyclic structure map on $\FFM_2$ interchanging inputs $0\leftrightarrow 1$ (with $0$ being the ``output'' index) is realized by the involution of $\C$ via $z\mapsto - \frac 1 z$ if we fix the position of point $1$ in our configuration at $z=0$.
	This cyclic structure induces one on $\PaRB$.
    If we compute it, we get that the interchange $0 \leftrightarrow 1$ is given as follows:
    \begin{align*}
      R_{1} & \mapsto R_{1},
      & X_{12} & \mapsto X_{12}^{-1}R_{2}^{-1},
      & A_{123} & \mapsto A_{231}^{-1}.
    \end{align*}
	Next we consider the (associated graded of the) Lie algebra of the pure ribbon braid group $t^R(r)$. It generated by symbols $t_{ij}$ with $1\leq i,j\leq r$ with relations 
	\begin{equation}\label{equ:drinfeldkohno_rel}
	\begin{aligned}
	t_{ij} &= t_{ji} \\
	[t_{ij}, t_{kl}] &= 0 \quad \text{for $\{i,j\}\cap\{k,l\}=\emptyset$ } \\
	[t_{ij},t_{ki}+t_{kj}] &= 0\, .
	\end{aligned}
	\end{equation}
	The Lie algebras $t^R(r)$ assemble into an operad in Lie algebras, see \cite[1.3]{Severa2010} for explicit formulas for the composition morphisms (with notation $s_j=\frac 1 2 t_{jj}$). The operad in Lie algebras $t^R$ also has a cyclic structure. Concretely, the cyclic structure map corresponding to the transposition $0\leftrightarrow 1$ is given by 
	\begin{align*}
	t_{ij} & \mapsto t_{ij}  \\
	t_{1i} & \mapsto -\sum_{k=1}^r t_{ki} \\
	t_{11} & \mapsto \sum_{k,l=1}^r t_{kl}
	\end{align*}
	for $i,j\geq 2$.
	
	From the discussion it follows that the group-like elements in the complete universal enveloping algebra $GU(t^R)$ form a cyclic operad in groups.
	The key fact is now that choosing a rational Drinfeld associator one obtains a map of operads in groupoids 
	\begin{equation}\label{equ:barnatan_map}
	\PaRB \to GU(t^R)\, ,
	\end{equation}
	cf. \cite[2.2]{Severa2010}. 
	Then the map \eqref{equ:barnatan_map} is defined such that
	\begin{equation}\label{equ:barnatan_map2}
	\begin{aligned}
	R & \mapsto e^{t_{11}/2}, &
	X &\mapsto e^{t_{12}/2}, &
	A &\mapsto \Phi(t_{12},t_{23}),
	\end{aligned}
	\end{equation}
	where $\Phi$ is the chosen Drinfeld associator.
	Let us also verify the compatibility with the cyclic structure, which is not shown in \cite{Severa2010}.
	
	\begin{lemma}\label{lem:barnatan_cyc}
		The map \eqref{equ:barnatan_map} above is a map of cyclic operads.
	\end{lemma}
\begin{proof}
	We just have to verify that the cyclic structure map corresponding to the transposition $0\leftrightarrow 1$, which we shall denote by $\sigma$, is preserved. This can be done on the generators:
		\begin{align*}
	\sigma(R) = R &\mapsto e^{t_{11}/2} = \sigma(e^{t_{11}/2}) \\
	\sigma(X) = X^{-1} R^{-1}_2 &\mapsto e^{-t_{12}/2-t_{22}/2}
	= \sigma(e^{t_{12}/2}) \\
	\sigma(A_{123}) = A_{231}^{-1} &\mapsto \Phi(t_{23},t_{31})^{-1}
	=\Phi(t_{31},t_{23}) = \Phi(-t_{12}-t_{23},t_{23}) = \sigma(\Phi(t_{12},t_{23}) \\
	\sigma(A_{213}) = A_{312} &\mapsto \Phi(t_{31},t_{12})
	=\Phi(-t_{12}-t_{23},-t_{13}-t_{23}) = \sigma(\Phi(t_{21},t_{13}) \, .
	\end{align*}
	Here we used the fact that $t_{12}+t_{23}+t_{31}\in t^R(3)$ is central and the antisymmetry relation for the associator.
\end{proof}
	
	Now we are ready to give a proof of Theorem \ref{thm:FFMformality}.
	
	\begin{proof}[Proof of Theorem \ref{thm:FFMformality}]
		Following \cite{Severa2010} we will take as a simplicial model for the framed little disks operad the nerve of the operad in groupoids $\PaRB$
		, $N_\bullet \PaRB$. 
		Let us construct a zigzag
		\[
		H^*(\FFM_2) \xleftarrow{\simeq} C t^R \xrightarrow{\simeq}  \Omega(N_\bullet \PaRB).
		\] 
		Here $C$ denotes the Chevalley--Eilemberg dgca and the left-hand arrow is a quasi-isomorphism of dg Hopf cooperads. 
		The right-hand arrow is a quasi-isomorphism of dg Hopf collections, which additionally satisfies the conditions of section \ref{sec:model_definition}. 
		From the fact that $\FFM_2$ is a $K(\pi,1)$ space, we deduce that its singular chains $S_\bullet \FFM_2$ are weak equivalent to its fundamental groupoid $N_\bullet \PaRB$. 
		It follows that $H^*(\FFM_2)$ is indeed a model of $\FFM_2$ in the sense of section \ref{sec:model_definition}. 
		
		The map $C t^R\to H(\FFM_2)$ is simply the obvious projection, sending the generators $t_{ij}$ to $\omega_{ij}$ and all commutators $[x^R,y^R]$ to zero.
		
		For the map $C t^R\to \Omega(N_\bullet \PaRB)$ we first consider the composition of morphisms of simplicial operads
		\begin{equation}\label{equ:FFMcomp}
		N_\bullet \PaRB \to B_\bullet G U(t^R) \xrightarrow{\cong}  \eta_\bullet(t^R) \xhookrightarrow{\simeq} \MC_\bullet(t^R)=\G C(t^R).
		\end{equation}
		Here the first map is induced by the map \eqref{equ:barnatan_map} (using a choice of rational Drinfeld associator)
		. The object $\eta_\bullet(t^R)=\{m\in \MC_\bullet(t^R)\mid hm=0 \}$ is Getzler's nerve \cite{Getzler} (denoted by $\gamma_{\bullet}(-)$ there), consisting of those elements of the Maurer--Cartan space which are annihilated by the Dupont contraction. For the isomorphism $B_\bullet G U(t^R) \to \eta_\bullet(t^R)$ see \cite[Theorem 13.2.6]{FresseBook} and \cite{Getzler}.
		The final arrow in \eqref{equ:FFMcomp} is just the inclusion. For the identification $\MC_\bullet(t^R)=\G C(t)$, where $\G=\Hom_{sSet}(-,\Omega(\Delta^\bullet))$ is Sullivan's realization functor, we again refer to \cite[Theorem 13.1.9]{FresseBook}.
		Now the adjunction unit gives us a map of dg Hopf cooperads
		\[
		C(t^R) \to\Omega_\sharp( \G C(t^R)) 
		\]
		where the first map is a morphism of dg Hopf cooperads, and the composition is a morphism of dg Hopf collections satisfying \eqref{eq:hopf-model}. Composing with $\Omega_\sharp(-)$ of the morphism \eqref{equ:FFMcomp} we hence obtain a morphism of dg Hopf cooperads
		\[
		C(t^R) \to\Omega_\sharp(N_\bullet \PaRB),
		\]
		which is equivalent to a morphism of dg Hopf collections 
		\[
			C(t^R) \to\Omega(N_\bullet \PaRB)
		\]
		satisfying the conditions of section \ref{sec:model_definition}.
		The only thing we need to check is that the above morphism is a quasi-isomorphism.
		But we know $H(C(t^R))\cong H(\Omega(N_\bullet \PaRB))\cong \BVc$, and hence we just need to check that generators are mapped to generators, which is a simple explicit computation.
	\end{proof}

	\subsection{Rational formality for configurations on the cylinder}
	Our next goal is to find a dg Hopf model for the configuration space of points on the cylinder $\FFM_C$, together with the involution $I:\FFM_C\cong (\FFM_C)^{op}$ induced by the map $z\mapsto \frac 1 z$.
	
	\begin{proposition}\label{prop:FMC_formality}
		The algebra object in right $\FFM_2$-modules $\FFM_C$ is rationally formal, i.e., a dg Hopf model for the pair $(\FFM_2, \FFM_C)$ is given by the pair $(H(\FFM_2),H(\FFM_C))=(\BVc, \BVc_C)$.
	\end{proposition}
	
	Again we will use that $\FFM_C(r)$ is a $K(\pi,1)$-space for every $r$. Concretely, we consider the algebra in right $\PaRB$-modules $\PaRB_C$ such that
	\[
	\PaRB_C(r) = \fiber \bigl( \PaRB(r+1)\to \PaRB(1) \bigr).
	\] 
	Elements $\PaRB_C(r)$ can be seen as braidings of $r$ ribbons and one (non-ribbon) string, corresponding to the marked point $*$. This is a subalgebra of $\PaRB_{1}$ from Example~\ref{ex:mod_from_op}. The algebra in right $\PaRB$-modules $\PaRB_C$ is generated by the operations $X_*\in \PaRB_C(1)$ and $A_{*12},A_{1*2},A_{12*}$, cf. Figure \ref{fig:gen-parb1}.
    \begin{figure}[htbp]
      \centering
      $X_{*} = $
      \begin{tikzpicture}[scale=.33, baseline=.5cm]
        \draw (0,0) -- (5,0);
        \draw (0,4) -- (5,4);
        \draw[very thick]
        (3.5,0) node[below] {*}
        -- (1.5,4) node[above] {*};
        \draw[fill=lightgray]
        (1,0) -- (2,0) node[midway, below] {1}
        -- (4,4) -- (3,4) node[midway, above] {1}
        -- (1,0);
      \end{tikzpicture}
      \qquad
      $A_{*12} = $
      \begin{tikzpicture}[scale=.33, baseline=.5cm]
        \draw (0,0) -- (9,0);
        \draw (0,4) -- (9,4);
        \draw[very thick] (1.5,4) node[above] {*}
        -- (1.5,0) node[below] {*};
        \draw[fill=lightgray] (5,4) -- (6,4) node[midway, above] {1}
        -- (4,0) -- (3,0) node[midway, below] {1}
        -- (5,4);
        \draw[fill=lightgray] (7,4) -- (8,4) node[midway, above] {2}
        -- (8,0) -- (7,0) node[midway, below] {2}
        -- (7,4);
      \end{tikzpicture}
      \caption{Generators of $\PaRB_{C}$ ($A_{1*2}$ and $A_{12*}$ are similar)}
      \label{fig:gen-parb1}
    \end{figure}
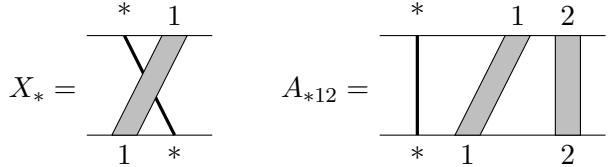

	Note also that $\PaRB_C$ comes with a natural involution 
	\[
	I: \PaRB_C \cong \PaRB_C^{op}
	\]
	induced by the map $z\mapsto \frac 1 z$.
	(Note also that the $(-)^{op}$ on the right-hand side refers solely to the algebra structure, not to the groupoid structure.)
	
	Furthermore, we can build an algebra object in operadic right $t^R$-modules composed of the Lie algebras
	\[
	t^R_C(r) = \fiber(t^R(r+1)\to t^R(1)).
	\]
	More concrely, $t^R_C(r)$ has a presentation in terms of generators $t_{ij}$, $t_{*i}$, $1\leq i,j\leq r$ and relations which are straightforward extensions of \eqref{equ:drinfeldkohno_rel}.
	The object $t^R_C$ comes with the involution 
	\[
	I: t^R_C \to (t^R_C)^{op}
	\]
	so that $I(t_{ij})=t_{ij}$ and $I(t_{*i})=-\sum_j t_{ji}$.
	(Note that again that the  $(-)^{op}$ refers only to the algebra structure.)
	
	\begin{lemma}
		The map \eqref{equ:barnatan_map} may be extended to a map of pairs 
		\[
		(\PaRB,\PaRB_C) \to (U(t^R), U(t^R_C))
		\]
		by the assignment on generators\najib{This should be $X_{*}$ right?}\ricardo{Agree. But actually shouldn't it be $ X_* \mapsto e^{t_{*1}/2} $?}\najib{I think so yes.}
		\begin{align*}
		  X_* & \mapsto e^{t_{*1}/2} 
		  &
		  A_{*12} & \mapsto \Phi(t_{*1},t_{12}) 
		  &
		  A_{1*2} & \mapsto \Phi(t_{*1},t_{*2}) 
		  &
		  A_{12*} & \mapsto \Phi(t_{12},t_{*2}) \, .
		\end{align*}
		The map furthermore intertwines the involutions $I$ above.
	\end{lemma}
\begin{proof}
	Note that the construction of $\PaRB_C$ out of $\PaRB$, and similarly of $t^R$ out of $t^R_C$ was functorial in the input cyclic operad, with the cyclic structure giving rise to the involution. Then we obtain our map from \eqref{equ:barnatan_map}, and it is clear that it is given by the stated formulas.
\end{proof}
	
	We now take the pair of nerves $(N_\bullet \PaRB, N_\bullet \PaRB_C)$ as our simplicial model for the pair $(\FFM_2,\FFM_C)$. We can then finish the proof of  Proposition \ref{prop:FMC_formality}.

	\begin{proof}[Proof of  Proposition \ref{prop:FMC_formality}]
	 We have morphisms of pairs of simplicial operads and algebras in right modules, with involution
     \begin{multline*}
       (N_\bullet \PaRB, N_\bullet \PaRB_C)
       \to 
       (B_\bullet GU(t^R), B_\bullet GU(t^R_C)) \to \\
       \to 
       (\eta_\bullet (t^R), \eta_\bullet (t^R_C))
       \to
       (\MC_\bullet (t^R), \MC_\bullet (t^R_C))
       =
       (\G C(t^R), \G C(t^R_C)).
     \end{multline*}
     The first arrow again uses a rational Drinfeld associator.
     As in the proof of Theorem~\ref{thm:FFMformality} above, we then get a morphism of pairs of dg collections
	\[
	(C(t^R), C(t^R_C)) \to (\Omega(N_\bullet \PaRB),\Omega(N_\bullet \PaRB_C) )
	\]
	compatible with the operadic structures and the involutions. Again one checks that the induced map on cohomology 
	\[
	(\BVc,\BVc_C) \to (\BVc,\BVc_C)
	\]
	is the identity, by tracing the generators, so that $(C(t^R), C(t^R_C))$ is a dg Hopf model for $(N_\bullet \PaRB, N_\bullet \PaRB_C)$.
	
	Composing with the obvious projection
	\[
	(C(t^R), C(t^R_C)) \xrightarrow{\simeq} (\BVc, \BVc_C) 
	\]
	then finishes the proof of the proposition.
  \end{proof}

\subsection{Rational formality for configuration space of sphere with marked points}
Exactly as in the previous subsection one shows the following result.
\begin{proposition}\label{prop:FM1p_formality}
	The $\FFM_C$-module $\FFM_2^{(1,p)}$ is rationally formal, i.e., a dg Hopf model for the triple $(\FFM_2, \FFM_C, \FFM_2^{(1,p)})$ is given by the triple $(H(\FFM_2),H(\FFM_C),H( \FFM_2^{(1,p)}))=(\BVc, \BVc_C, \BVc_{(1,p)})$.
\end{proposition}
\begin{proof}
	One just copies the proofs of Theorem \ref{thm:FFMformality} and Proposition \ref{prop:FMC_formality} with minimal (obvious) changes.
\end{proof}

\subsection{Models for configuration spaces on surfaces}\label{sec:confspacemodels}

The final goal of this section is to assemble the above pieces and show:

\begin{proposition}\label{prop:thinmodel}
  Let $\BVc_{(g,g)} = H(\FFM_{2}^{(1+(g-1),g)})$.
  The pair $(\BVc, \HH_g)$ with
	\[
	\HH_g = \otimes^{(1,1)\dots(g,g)}_{\BVc_C} \BVc_{(g,g)} 
	\]
	is a dg Hopf cooperad/comodule model for the pair $(\FFM_2, \FFM_{\Sigma_g})$ for $g\geq 1$.
\end{proposition}
\begin{proof}
	Assembling the morphisms of Theorem \ref{thm:FFMformality}, Propositions \ref{prop:FMC_formality}, \ref{prop:FM1p_formality} and their proofs we find the following morphisms of operadic modules, parallel to the morphisms of simplicial operads \eqref{equ:FFMcomp}.
	\begin{equation}\label{equ:mod_comp}
	\FFM_{\Sigma_g}
	\simeq 
	\hotimes_{N_\bullet \PaRB}^{(1,1)\dots (g,g)} N_\bullet \PaRB_{(g,g)}
	\to
	\hotimes_{B_\bullet GU(t^R)}^{(1,1)\dots (g,g)} B_\bullet G U(t_{(g,g)}^R)
	\to 
	\hotimes_{ \MC_\bullet(t^R)}^{(1,1)\dots (g,g)} \MC_\bullet(t_{(g,g)}^R).
	\end{equation}
	We furthermore have a morphism of dg Hopf collections
	\[
	\hotimes^{(1,1)\dots (g,g)}_{C(t^R)} C(t^R_{(g,g)})
	\to 
	\hotimes^{(1,1)\dots (g,g)}_{\Omega(\MC_\bullet(t^R))} 
	\Omega(\MC_\bullet(t_{(g,g)}^R))
	\to 
	\Omega( \hotimes_{ \MC_\bullet(t^R)}^{(1,1)\dots (g,g)} \MC_\bullet(t_{(g,g)}^R) )
	\]
	respecting the operadic module structures.
	
	Applying $\Omega(-)$ to \eqref{equ:mod_comp} and composing with the previous morphism hence gives a morphism of dg Hopf collections 
	\[
 \hotimes^{(1,1)\dots (g,g)}_{C(t^R)} C(t^R_{(g,g)})
 \to 
 \Omega(\hotimes_{N_\bullet \PaRB}^{(1,1)\dots (g,g)} N_\bullet \PaRB_{(g,g)}),
	\]
	respecting the operadic module structure.
	We claim it is a quasi-isomorphism.
	Note that forms of the (fat or not) realization are the (fat or not) totalization of the forms
	\[
	\Omega(| -|) = \Tot(\Omega(-))
	\]
	by adjunction. But by homotopy invariance of the fat totalization the statement then follows, since inside the end we have quasi-isomorphisms (see Section~\ref{sec:dual-constr-hopf}).
	
	Finally, we just compose with the quasi-isomorphism and projection
	\[
	\hotimes^{(1,1)\dots (g,g)}_{C(t^R)} C(t^R_{(g,g)}) \to 
	\otimes^{(1,1)\dots(g,g)}_{\BVc_C} \BVc_{(g,g)} 
	\]
	to finish the proof of the proposition.
\end{proof}

For later reference we shall also note that the coaction of the nullary $\BVc$-cogenerator produces maps
\begin{align}\label{equ:piHHg}
\pi_j^*:\HH_g(1)&\to \HH_g(r) 
& 
\pi_{ij}^*:\HH_g(2)&\to \HH_g(r) 
\end{align}
for $1\leq i\neq j\leq r$, which are (in some sense) dual to the forgetful maps \eqref{equ:piFFMSigma}, forgetting all but the $j$-th (and $i$-th) point from a configuration.

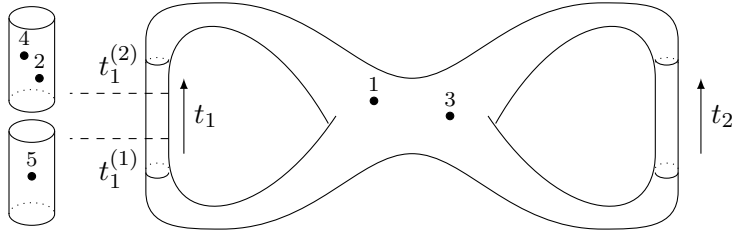
\begin{figure}
\[
\begin{tikzpicture}[baseline=1cm, scale=1,int1/.style={draw, circle,fill,inner sep=1pt, outer sep=0pt}, label distance=-.5mm]
\draw (0,0)
.. controls (-.7,0) and (-1.3,-1) .. (-2.5,-1)
.. controls (-3,-1) and (-3.5,-1) .. (-3.5,-.5)
--(-3.5,1.5)
.. controls (-3.5,2) and (-3,2) .. (-2.5,2)
.. controls (-1.3,2) and (-.7,1) .. (0,1)
.. controls  (.7,1) and (1.3,2) .. (2.5,2)
.. controls (3,2) and (3.5,2) ..  (3.5,1.5)
--(3.5,-.5)
.. controls (3.5,-1) and (3,-1) .. (2.5,-1)
.. controls (1.3,-1) and (.7,0)  .. (0,0);
\draw (-1,.5)
.. controls (-2,-1) and (-3.2,-1) .. (-3.2,-.0)
-- (-3.2,1)
.. controls (-3.2,2) and (-2,2) .. (-1.1,.4)
;
\begin{scope}[yshift=0cm,yscale=-1]
\draw (1,-.5)
.. controls (2,1) and (3.2,1) .. (3.2,-.0)
-- (3.2,-1)
.. controls (3.2,-2) and (2,-2) .. (1.1,-.4)
;
\end{scope}
\draw (-3.5,-0.25) arc (180:360:.15cm and .075cm);
\draw[dotted] (-3.5,-0.2) arc (180:0:.15cm and .075cm);
\draw (-3.5,1.25) arc (180:360:.15cm and .075cm);
\draw[dotted] (-3.5,1.2) arc (180:0:.15cm and .075cm);
\draw (-3,-.0) edge[-latex] node[right] {$t_1$} (-3,1);
\begin{scope}[yshift=0cm,xscale=-1]
\draw (-3.5,-0.25) arc (180:360:.15cm and .075cm);
\draw[dotted] (-3.5,-0.2) arc (180:0:.15cm and .075cm);
\draw (-3.5,1.25) arc (180:360:.15cm and .075cm);
\draw[dotted] (-3.5,1.2) arc (180:0:.15cm and .075cm);
\draw (-3.8,-.0) edge[-latex] node[right] {$t_2$} (-3.8,1);
\end{scope}

\draw[dashed] (-3.2,.2) edge node[below]{$t_1^{(1)}$} (-4.5,.2);
\begin{scope}[yshift=-.5cm,xshift=.2cm]
\draw (-5.5,-0.3) arc (180:360:.3cm and .15cm);
\draw[dotted] (-5.5,-0.3) arc (180:0:.3cm and .15cm);
\draw (-5.5,.8) arc (180:360:.3cm and .15cm);
\draw[] (-5.5,.8) arc (180:0:.3cm and .15cm);
\draw (-5.5,-.3)--(-5.5,.8) (-4.9,-.3)--(-4.9,.8);
\node [int1, label={$\scriptstyle 5$}]  at (-5.2,.2) {};
\end{scope}

\draw[dashed] (-3.2,.8) edge node[above]{$t_1^{(2)}$} (-4.5,.8);
\begin{scope}[yshift=1cm,xshift=.2cm]
\draw (-5.5,-0.3) arc (180:360:.3cm and .15cm);
\draw[dotted] (-5.5,-0.3) arc (180:0:.3cm and .15cm);
\draw (-5.5,.8) arc (180:360:.3cm and .15cm);
\draw[] (-5.5,.8) arc (180:0:.3cm and .15cm);
\draw (-5.5,-.3)--(-5.5,.8) (-4.9,-.3)--(-4.9,.8);
\node [int1, label={$\scriptstyle 2$}]  at (-5.1,0) {};
\node [int1, label={$\scriptstyle 4$}]  at (-5.3,.3) {};
\end{scope}

\node [int1, label={$\scriptstyle 3$}]  at (.5,.5) {};
\node [int1, label={$\scriptstyle 1$}]  at (-.5,.7) {};
\end{tikzpicture}
		\]
		\caption{\label{fig:quenches}  We will consider the coend construction \eqref{equ:ourFMg} topologically as a sphere with ``thin'' handles attached (shown here for $g=2$).
		 The longitudinal coordinates along each handle we denote be $t_j\in[0,1]$. Points in a configuration can either be in the ``bulk'' sphere (like 1 and 3 above), or in ``packets'' $\FFM_1$ on the handle. For example, here are 2 packets depicted on the first handle, one at longitudinal coordinate $t_1^{(1)}$ and one at $t_1^{(2)}$. }
	\end{figure}
	
	\subsection{Notation for elements of $\BVc_{(g,g)}$, $\BVc_C$ and $\BVc$ }\label{sec:BVcnotation}

	Recall the presentation \eqref{equ:BVcpresentation} for the dgca $\BVc(r)$. 
	We shall introduce similar notation for the related objects  $\BVc_C$ and $\BVc_{(g,g)}$, for later use.
	
	First, as a graded commutative algebra, $\BVc_C(r)$ is generated by degree $+1$-elements ${\omega_{*j}=\omega_{j*}},$ $\omega_{ij}=\omega_{ji}, \ \theta_i$ for $1\leq i\neq j\leq r$, with relations
	\begin{align*}
	\omega_{ij}&=\omega_{ji}	\\
	\omega_{ij}\omega_{jk} +\omega_{jk}\omega_{ki} +\omega_{ki}\omega_{ij} &=0 \\
	\omega_{*i}\omega_{ij} +\omega_{ij}\omega_{*j} +\omega_{*j}\omega_{*i} &=0\, .
	\end{align*}
	Here a representative of $\omega_{ij}$ in the configuration space of $r$ framed points on $\C^*$ is $\frac 1 {2\pi} d\arg(z_i-z_j)$. A representative of $\omega_{*i}$ is likewise $\frac 1 {2\pi} d\arg(z_i)$. Finally, if the framing (angle) at point $i$ is $\phi_i$, then a representative of $\theta_i$ is $\frac 1 {2\pi} d\phi_i$.
	
	The action of our involution $I:\BVc_C \to (\BVc_C)^{op}$, induced by the map $z\mapsto \frac 1 z$ of the cylinder is, as one can easily check:
	\begin{equation}\label{equ:IBVc1}
	\begin{aligned}
	I(\omega_{*i}) &= -\omega_{*i} \\
	I(\omega_{ij}) &= \omega_{ij} -\omega_{*i} -\omega_{*j} \\
	I(\theta_i) &= \theta_{i} -2\omega_{*i}
	\end{aligned}\, .
	\end{equation}
	
	Note also that the formulas above can be slightly streamlined by setting $\omega_{ii}\coloneqq \theta_i$.
	
	Second, we consider $\BVc_{(g,g)}(r)=H(\FFM_2^{(g,g)}(r))$.
	Note that to define $\FFM_2^{(g,g)}$ we have fixed $2g-1$ points.
	We will label these points by $\underline 1,\underline 2,\dots,\underline g, \overline 2,\dots, \overline g$.
	Here we have a right (boundary) $\FFM_C$-action at the points $\underline j$, and left actions at the points $\overline j$. (The missing left action is at $\infty$.)
	
	Accordingly, $\BVc_{(g,g)}(r)$ is generated as a graded commutative algebra by the following elements.
	\begin{itemize}
	\item $\omega_{ij}=\omega_{ji}$ for $1\leq i\neq j\leq r$.
	\item $\theta_i$ for $1\leq i\leq r$.
	\item $\omega_{i\underline j}=\omega_{\underline j i}$ for $1\leq i\leq r$, $1\leq j\leq g$.
	\item  $\omega_{i\overline j}=\omega_{\overline j i}$ for $1\leq i\leq r$, $2\leq j\leq g$.
	\end{itemize}
	We leave it to the reader to adapt the relations of $\BVc(r+2g-1)$ to relations in $\BVc_{(g,g)}(r)$.
	
	In case we use finite sets $S$ for indexing, e.g., $\BVc(S)$, $\BVc_C(S)$, $\BVc_{(g,g)}(S)$ we continue to use the same notation as above, just with the symbol set $\{1,\dots,r\}$ replaced by $S$.
	
	\subsection{Discussion of $\HH_g$}\label{sec:HHg_notation}
	Our model $\HH_g$ for the configuration space of points is defined as a categorial end -- a fat totalization.
	This yields a compact definition, but makes it notationally slightly hard to construct concrete elements in $\HH_g$, as we shall need to do in the next subsection.
	
	To facilitate the discussion, we shall study now in some more detail what these elements are, and introduce some suggestive notation.
	By definition, the fat totalization is a subspace of the direct product
	\begin{equation}\label{equ:HHgstrata0}
	\HH_g(S) \subset 
	\prod_{\underline r=(r_1,\dots,r_g)}
	\prod_{f : S \to B(\underline r)}
	\BVc_{(g,g)}(f^{-1}(0))
	\otimes 
	\bigotimes_{j=1}^g
	\left( 
	\BVc_C(f^{-1}( (j,1)))
	\otimes
	\cdots \otimes \BVc_C(f^{-1}( (j,r_j)))
	\otimes \Omega(\Delta^{r_j})
	\right)\, .
	\end{equation}
	Here $\Omega(\Delta^{r_j})$ are the polynomial forms on the $r_j$-simplices.
	The outer product in the formula is over tuples $\bar r=(r_1,\dots,r_g)$ of numbers $r_j=0,1,\dots$. The inner product is over functions from our set (of ``points'') $S$ into the set 
	\[
	B(\underline r) = \{0,(1,1),\dots,(1,r_1),(2,1),\dots,(2,r_2),\dots ,(g,1),\dots,(g,r_g)\}.
	\]	
	Geometrically, we may think that $\HH_g(S)$ is a version of ``differential forms'' on the embedding space of $S$ into $\Sigma_g$, i.e., on the configuration space of $|S|$ points in $\Sigma_g$.
	In this case we should think of $\Sigma_g$ as a sphere on which we glue $g$ thin  handles, see Figure \ref{fig:quenches}.
	Then the underlying model for the configuration space is \eqref{equ:ourFMg}, which is a quotient of
	\begin{equation}\label{equ:HHFFM}
	\coprod_{\underline r=(r_1,\dots,r_g)}
	\coprod_{f : S \to B(\underline r)}
	\FFM_2^{(g,g)}(f^{-1}(0))
	\times \bigtimes_{j=1}^g
	\FFM_C(f^{-1}( (j,1)))
    \times \cdots
    \times \FFM_C(f^{-1}( (j,r_j))) \times \Delta^{r_j}).
	\end{equation}
	The points can either be in the bulk, or they can be in one of the $g$ handles. In the latter case they have a coordinate $t_j\in [0,1]$ along the $j$-th handle, and they can be in various ``packets'' of the same coordinate $t_j$.
	Concretely, in the formula above, the function $f$ determines how the points are divided into packets and put on handles, with $f^{-1}(0)\subset S$ being the set of points in the bulk(-sphere), and $f^{-1}((j,k))$ being the points in the $k$-th packet on the $j$-th handle.
	
	Reformulating this, we may think that the outer products over $(\underline r, f)$ divides our ``configuration space'' into various strata, and the corresponding factor in \eqref{equ:HHgstrata0} are the ``differential forms on that stratum''.
	The conditions in the totalization that singles out the subspace $\HH_g(S)$ can then be interpreted as saying that $\HH_g(S)$ consists of the ``transversely continuous'' differential forms, i.e., those whose pieces on the strata agree on the mutual boundaries of the strata.
	
	Before we make this more explicit, we shall introduce one simplification. 
	For homotopy theoretic reasons we used above the fat totalization $\HH_g$. However, all elements in this space we will ever need to explicitly touch in this paper lie in the ordinary (non-fat) totalization 
	\[
	\HH_g' \subset \HH_g.
	\]
	Similarly to \eqref{equ:HHgstrata0}, we can see $\HH_g'(S)$ as a subset of the product
	\begin{equation}\label{equ:HHgstrata}
	\HH_g'(S) \subset 
	\prod_{\underline r=(r_1,\dots,r_g)}
	\prod_{f : S \to B(\underline r) \atop \text{surj.}}
	\BVc_{(g,g)}(f^{-1}(0))
	\otimes 
	\bigotimes_{j=1}^g
	\left( 
	\BVc_C(f^{-1}( (j,1)))\otimes \cdots \otimes \BVc_C(f^{-1}( (j,r_j)))
	\otimes \Omega_{poly}(\Delta^{r_j})
	\right).
	\end{equation}
	The difference to \eqref{equ:HHgstrata0} is that now the inner product is only over maps $f : S \to B(\underline r)$ that surject onto $B(\underline r)\setminus \{0\}$, instead of all maps.
	Geometrically speaking, we only define the ``differential form'' on nondegenerate strata. Then the codegeneracies are implicitly used to define the inclusion $\HH_g'\subset\HH_g$.
	For us this has the advantage that while there are infinitely many $(\underline r, f)$ in the products in \eqref{equ:HHgstrata0}, there are only finitely many in 
	\eqref{equ:HHgstrata}, making it much easier to construct elements explicitly.
	
	Indeed, below we will describe elements of $\HH_g'(S)$ by giving for each stratum $\sS=(\underline r,f)$ the value in the corresponding factor of \eqref{equ:HHgstrata}, and then check that the conditions imposed by the totalization are satisfied.
	We will say that we define the ``differential form'' on each stratum, and then check that they agree on strata boundaries.

	More precisely, say our task is to define an element $\alpha\in \HH_g'(S)$. Then we have to provide a collection of elements
	\[
	\alpha_{\sS}\in 
	\BVc_{1,g}(f^{-1}(0))
	\otimes 
	\bigotimes_{i=1}^g
	\bigotimes_{j=1}^{r_i}
	\BVc_C(f^{-1}( (i,j)))
	\otimes \Omega_{poly}(\Delta^{r_1}\times \cdots \times \Delta^{r_g}),
	\]
	one for every stratum $\sS=(\underline r,f)$ as above.
	This collection yields a well-defined element of $\HH_g'(S)$ iff the continuity conditions of the end are satisfied.
	More concretely, the conditions can be divided into two types.
	\begin{itemize}
		\item First, consider a stratum $\sS=(\underline r,f)$ and suppose that $\tilde \sS=(\tilde{\underline r},\tilde f)$ is almost the same stratum, except that the ``packets'' $(i,j)$ and $(i,j+1)$ have been merged.
		This means concretely that $\tilde r_i=r_i-1$ (with the other $r_p=\tilde r_p$), and that $\tilde f^{-1}((i,j))=f^{-1}((i,j))\cup f^{-1}((i,j))$. 
		Thinking of $\Delta^{r_i}$ as configurations of $r_i$ ordered points on an interval, this corresponds to taking a boundary stratum $\p_j\Delta^{r_i}$ of the simplex where the $j$-th point collides with the $j+1$st for $0 < j < r_{j}$.
		Then the continuity condition reads:
		\begin{equation}\label{equ:HHgcontinuity1}
		\alpha_{\sS} \mid_{\p_j\Delta^{r_i}}
		=
		\Delta
		\alpha_{\tilde \sS},
		\end{equation}
		where on the left-hand side we restrict the form piece in $\Omega_{poly}(\Delta^{r_i})$ to the boundary $\p_j\Delta^{r_i}$, and on the right-hand side we apply the cocomposition 
		\[
		\Delta : \BVc_C(\tilde f^{-1}((i,j)))
		\to \BVc_C(f^{-1}((i,j)))\otimes \BVc_C(f^{-1}((i,j+1)))
		\]
		to the factor corresponding to the ``merged packet''.
      \item The remaining conditions stem from the remaining boundary strata of the simplices, namely $\p_0\Delta^{r_i}$ and $\p_{r_i}\Delta^{r_i}$.
        These appear when one packet on a handle hits the front or back end of the handle.
		The relevant condition is very similar to the previous one, except that now one has to take a cocomposition of the bulk term.
		For example, consider the case that the first packet on the $i$-th handle hits the left end ($t=0$) of the $i$-th handle.  We consider the stratum  $\tilde \sS=(\tilde{\underline r},\tilde f)$ defined essentially like $\sS$, but such that $\tilde f^{-1}(0)=f^{-1}(0)\cup f^{-1}((i,1))$, so that the first packet on the $i$-th handle is merged with the bulk. 
		Then the continuity condition reads
		\begin{equation}\label{equ:HHgcontinuity2}
		\alpha_{\sS} \mid_{\p_0\Delta^{r_i}}
		=
		\Delta
		\alpha_{\tilde \sS},
		\end{equation}
		where now the cocomposition $\Delta$ is obtained from the cocomposition of the bulk term
		\[
		\Delta : \BVc_{(g,g-1)}(\tilde f^{-1}(0)) \to \BVc_{(g,g-1)}( f^{-1}(0)) \otimes 
		\BVc_{C}(\tilde f^{-1}((i,1)))\, .
		\]
		Here we use the $i$-th right $\BVc$-comodule structure on $\BVc_{(g,g-1)}$.
		Similarly, in the case that the $r_i$-th packet hits the right end ($t=0$) of the $i$-th handle we define $\tilde \sS=(\tilde{\underline r},\tilde f)$ with $\tilde f^{-1}(0)=f^{-1}(0)\cup f^{-1}((i,r_i))$, and the continuity condition reads
		\begin{equation}\label{equ:HHgcontinuity3}
		\alpha_{\sS} \mid_{\p_{r_i}\Delta^{r_i}}
		=
		\Delta
		\alpha_{\tilde \sS}.
		\end{equation}
		Now, for the cocomposition, we use the appropriate left $\BVc$-comodule structure on $\BVc_{(g,g-1)}$,
		\[
		\Delta : \BVc_{(g,g-1)}(\tilde f^{-1}(0)) \to \BVc_{(g,g-1)}( f^{-1}(0)) \otimes 
		\BVc_{C}(\tilde f^{-1}((i,r_i)))\, .
		\]
	\end{itemize}

	\subsubsection{Example: $\HH_g'(1)$ and $\HH_g'(2)$}\label{sec:HHg12_example_strata}
	For concreteness, and for later use we shall explicitly discuss the strata, i.e., the pairs $\sS\coloneqq(\underline r,f)$ occurring in \eqref{equ:HHgstrata} for $\HH_g'(1)$ and $\HH_g'(2)$.
	We begin with $\HH_g'(1)$. In this case we can have the following types of strata $\sS\coloneqq (\underline r,f)$.
	\begin{itemize}
		\item $\underline r=(0,\dots,0)$ and (necessarily) $f(1)=0$. This corresponds geometrically to one point being on the ``bulk'' sphere.
		\item $\underline r=(1,0,\dots,0)$ and (necessarily) $f(1)=(1,1)$. This means geometrically that the point is on the first handle.
		\item $\underline r=(0,\dots,0,1,0,\dots,0)$, with $r_j=1$ in the $j$-th position, and $j\geq 2$. Then, necessarily, $f(1)=(j,1)$. This means geometrically that the point is on the $j$-th handle.
	\end{itemize}
	Note that the last two cases could obviously be unified. However, below we shall always treat the first handle separate from the others, since the corresponding $\BVc_C$-coactions on $\BVc_{(g,g)}$ are defined differently for the first, and for the other handles.
	
	Next consider $\HH_g'(2)$. Here we shall generally need to distinguish 14 types of strata $\sS=(\underline r,f)$, again owed to the fact the we will need to take the first handle different from the others.
	We also introduce letters to be able to reference the strata later.
	\begin{itemize}
		\item (Stratum $A$) For $\underline r=(0,\dots,0)$ we necessarily have $f(1)=f(2) = 0$. Geometrically, both points are in the bulk sphere. 
		\item For $\underline r=(1,0,\dots,0)$ we have the options:
		\begin{itemize}
			\item (Stratum $B_1$) $f(1)=(1,1)$, $f(2)=0$. Point 1 is on the first handle, point 2 on the bulk sphere. 
			\item (Stratum $B_2$) $f(1)=0$, $f(2)=(1,1)$. The situation is reversed.
			\item (Stratum $B_{(12)}$) $f(1)=f(2)=(1,1)$. Here both points are ``infinitesimally close'' and on the first handle.
		\end{itemize}
	\item Likewise, for $\underline r=(0,\dots,1,\dots,0)$  (with 1 at position $j\geq 2$) we have the options:
	\begin{itemize}
		\item (Strata $C_1$) $f(1)=(1,1)$, $f(2)=0$. 
		\item (Strata $C_2$) $f(1)=0$, $f(2)=(1,1)$.
		\item (Stratum $C_{(12)}$) $f(1)=f(2)=(1,1)$. 
	\end{itemize}
	\item For $\underline r=(2,0,\dots,0)$ both points need to be on the first handle. Their ordering on the handle gives rise to 2 strata.
	\begin{itemize}
		\item (Stratum $B_{12}$) $f(1)=(1,1)$, $f(2)=(1,2)$. Both points are on the first handle, with coordinates $t_1^{(1)}<t_1^{(2)}$, cf. Figure \ref{fig:quenches}.
		\item (Stratum $B_{21}$) $f(1)=(1,2)$, $f(2)=(1,1)$. 
	\end{itemize}
	\item Likewise, for $\underline r=(0,\dots,2,\dots,0)$ we obtain the strata $C_{12}$ with $f(1)=(j,1)$, $f(2)=(j,2)$, and $C_{21}$ with $f(1)=(j,2)$, $f(2)=(j,1)$
	\item For $\underline r=(1,0,\dots,0,1,0,\dots,0)$, i.e., $r_1=r_j=1$, we must have $f(1)=(1,1)$, $f(2)=(j,1)$ (Strata $B_1C_2$) or $f(1)=(j,1)$, $f(2)=(1,1)$ (Strata $B_2C_1$)
	\item For $\underline r=(0,,\dots,0,1,0,\dots,0,1,0,\dots,0)$, i.e., $r_i=r_j=1$, with $i\ne j$ we must have $f(1)=(i,1)$, $f(2)=(j,1)$ or $f(1)=(j,1)$, $f(2)=(i,1)$. We collectively denote these strata by $C_1C_2$.
\end{itemize}
	
The various strata are depicted in a schematic way in Figure \ref{fig:strata}.
The diagram in this Figure shows in addition which strata are ``adjacent'' in the sense that the totalization imposes boundary conditions between the values of our form on these strata. 
Concretely, that means in order to define an element $\alpha\in\HH_g'(2)$ below, we will need to provide, for each stratum $\sS=(\underline r,f)$ as above an element $\alpha_{\sS}$ of the corresponding factor of \eqref{equ:HHgstrata}.
Then we need to check, for each edge in the adjacency diagram of  Figure \ref{fig:strata}, that the corresponding continuity condition is satisfied.

\begin{figure}
	\newcommand{\mouse}{
		\draw (1,1) circle (.5cm);
		\draw (1,-1) circle (.5cm);
		\draw[fill=white] (0,0) circle (1cm);
	}
	\begin{align*}
	A &:
	\begin{tikzpicture}[baseline=-.65ex, scale=.6]
	\mouse
	\node at (0,0) {1 2};
	\end{tikzpicture}
	&
	B_1&:
	\begin{tikzpicture}[baseline=-.65ex, scale=.6]
	\mouse
	\node at (1.7,1.7) {1};
	\node at (0,0) {2};
	\end{tikzpicture}
	&
	B_2&:
	\begin{tikzpicture}[baseline=-.65ex, scale=.6]
	\mouse
	\node at (1.7,1.7) {2};
	\node at (0,0) {1};
	\end{tikzpicture}
	&
	B_{12}&:
	\begin{tikzpicture}[baseline=-.65ex, scale=.6]
	\mouse
	\node at (1,1.8) {2};
	\node at (1.8,1) {1};
	\end{tikzpicture}
	&
	B_{21}&:
	\begin{tikzpicture}[baseline=-.65ex, scale=.6]
	\mouse
	\node at (1,1.8) {1};
	\node at (1.8,1) {2};
	\end{tikzpicture}
	\\
	B_{(12)} &:
	\begin{tikzpicture}[baseline=-.65ex, scale=.6]
	\mouse
	\node[int, label=45:{12}] at (45:1.91) {};
	\end{tikzpicture}
	&
	C_1&:
	\begin{tikzpicture}[baseline=-.65ex, scale=.6]
	\mouse
	\node at (1.7,-1.7) {1};
	\node at (0,0) {2};
	\end{tikzpicture}
	&
	C_2&:
	\begin{tikzpicture}[baseline=-.65ex, scale=.6]
	\mouse
	\node at (1.7,-1.7) {2};
	\node at (0,0) {1};
	\end{tikzpicture}
	&
	C_{12}&:
	\begin{tikzpicture}[baseline=-.65ex, scale=.6]
	\mouse
	\node at (1,-1.8) {2};
	\node at (1.8,-1) {1};
	\end{tikzpicture}
	&
	C_{21}&:
	\begin{tikzpicture}[baseline=-.65ex, scale=.6]
	\mouse
	\node at (1,-1.8) {1};
	\node at (1.8,-1) {2};
	\end{tikzpicture}
	\\
	C_{(12)} &:
	\begin{tikzpicture}[baseline=-.65ex, scale=.6]
	\mouse
	\node[int, label=-45:{12}] at (-45:1.91) {};
	\end{tikzpicture}
	&
	B_1C_2 &:
	\begin{tikzpicture}[baseline=-.65ex, scale=.6]
	\mouse
	\node at (1.7,1.7) {1};
	\node at (1.7,-1.7) {2};
	\end{tikzpicture}
	&
	B_2C_1 &:
	\begin{tikzpicture}[baseline=-.65ex, scale=.6]
	\mouse
	\node at (1.7,1.7) {2};
	\node at (1.7,-1.7) {1};
	\end{tikzpicture}
	\end{align*}
	\[
	\begin{tikzcd}
	& \ar[-]{r} B_{12}\ar[-]{ddr} & B_1 \ar[-]{ddl} \ar[-]{r} \ar[-]{dr} & B_1C_2\ar[-]{r}  & C_2\ar[-]{r} \ar[-]{ddr}\ar[-]{dl} & C_{21}\ar[-]{ddl} & \\
	B_{(12)} \ar[-]{rrr}\ar[-]{ur}\ar[-]{dr} & & & A \ar[-]{rrr} \ar[-]{dr}\ar[-]{dl} & & & C_{(12)} \ar[-]{ul}\ar[-]{dl}\\
	& B_{21} \ar[-]{r} & B_2\ar[-]{r}  & B_2C_1\ar[-]{r}  & C_1 \ar[-]{r} & C_{12} & C_1C_2 \ar[-]{lluu} \ar[-,bend right]{ll}
	\end{tikzcd}
	\]
	\caption{\label{fig:strata} Depiction of the various top (i.e., $4$-)dimensional strata in our version of the configuration space of two points. For example in stratum $A$ both points are in the bulk, in stratum $B_1$ point 1 is on the first handle and point 2 is in the bulk.
		In stratum $C_2$ point $1$ is in the bulk and point $2$ on a handle indexed $j=2,\dots ,g$.
		In stratum $B_{12}$ both points are on the first handle, with $t_1^{(1)}\leq t_1^{(2)}$ etc. The diagram below indicates the intersections between the various strata, with a line between strata if their intersection is $3$-dimensional. }
\end{figure}
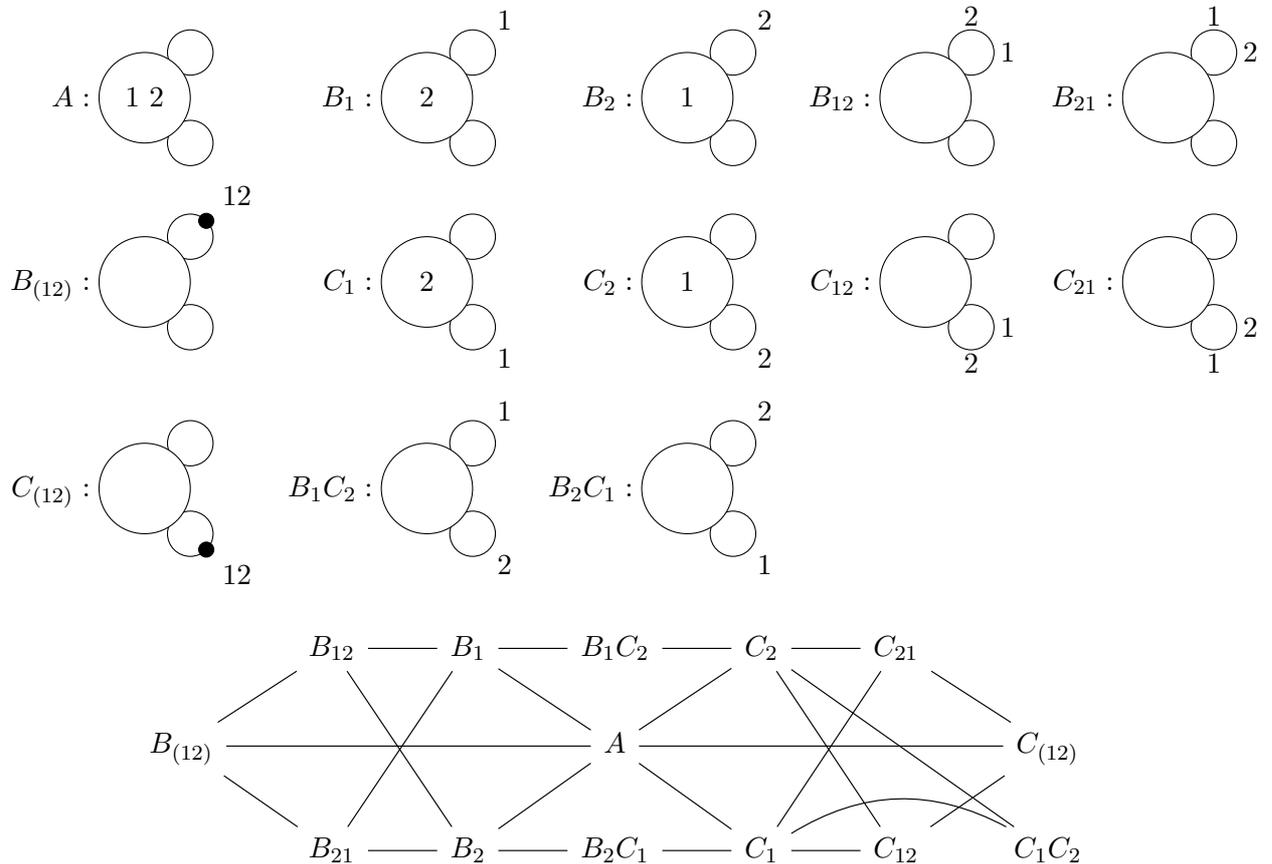

	\section{Proof of Theorem \ref{thm:main}}\label{sec:theproof}
	To unify the notation, we will first assume that the genus is positive, $g\geq 1$.
	The (simpler) case $g=0$ is then treated separately in section \ref{sec:g0case} below.
	(Note that the case $g=0$ has already been shown in \cite[Appendix B]{CamposWillwacher2016}, albeit over $\R$.)
	
	\subsection{Recollection of \cite{CamposWillwacher2016}-- another graphical model for configurations on $\Sigma_g$}\label{sec:recoll-citec-anoth}
	We begin with some recollections from \cite{CamposWillwacher2016}.
	Let $H_g\coloneqq H^*(\Sigma_g)$ be the cohomology of the surface of genus $g$, and let $1\in H_g^0$ denote class of the constant function $1$.
	Let similarly $\bar H_g=H^{\geq 1}(\Sigma_g)=H_g / \Q1$ be the reduced cohomology.
	We consider a collection of quasi-free graded commutative algebras
	\footnote{Note that in \cite{CamposWillwacher2016} the notation used for this algebra is $^*\BVGra$ whereas the notation $\BVGra$ is used for the linear dual.}
	\[
	\BVGra_{H_g}(r) = S(\omega_{ij}, e^1_j,\dots,e^{2g+1}_j)_{1\leq i\leq j\leq r},
	\]
	generated by symbols $\omega_{ij}\eqqcolon\omega_{ji}$ of degree 1 and symbols $e^k_j$ of degree 1 for $k<2g+1$ and of degree 2 for $k=2g+1$. We shall think of the $k$ as indexing a homogeneous basis $\{e^k\}_{k=0,\dots,2g+1}$ of $H_g$. It is also beneficial to allow the additional generator $e_j^0$, but then impose the relation $e_j^0=1$.
	Using this notation we define a differential such that 
	\[
	d\omega_{ij} = \sum_k e_i^k e_j^{k*}
	\]
	where $e_j^{k*}$ ranges over the Poincar\'e dual basis elements.
	
	Let us give a second, combinatorial definition of $\BVGra_{H_g}(r)$.
	Elements of $\BVGra_{H_g}(r)$ are linear combinations of diagrams with $r$ vertices numbered $1,\dots,r$, undirected edges and additionally decorations from $\bar H_g$ at vertices.
	An example is as follows:
	
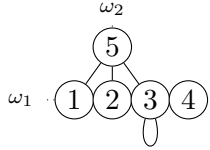
\[
\begin{tikzpicture}[baseline=-.65ex]
\node[ext] (v1) at (0,0) {1};
\node[ext] (v2) at (0.5,0) {2};
\node[ext] (v3) at (1,0) {3};
\node[ext] (v4) at (1.5,0) {4};
\node[ext] (w1) at (0.5,.7) {5};
	\node (i2) at (-.7,0) {$\scriptstyle \omega_1$};
\node (i3) at (.5,1.2) {$\scriptstyle \omega_2$};
\draw (v1) edge[dotted] (i2) edge (w1) (v2) edge (w1) (v3) edge (w1) (i3) edge[dotted] (w1) (v3) edge[loop below] (v3);
\end{tikzpicture}
	\] 
	Algebraically, a graph corresponds to a monomial in generators, with $\omega_{ij}$ being an edge between $i$ and $j$, and $\alpha_j$ the decoration $\alpha\in \bar H_g$ at vertex $j$. 
	
	The differential is then given graphically by removing one edge and replacing it by a diagonal
	\begin{equation}\label{equ:dongra}
	d\, 
\begin{tikzpicture}[baseline=-.65ex]
\node[ext] (v) at (0,0) {};
\node[ext] (w) at (0.7,0) {};
\draw (v) edge +(-.5,0) edge +(-.5,.5) edge +(-.5,-.5) edge (w)
(w) edge +(.5,0) edge +(.5,.5) edge +(.5,-.5);
\end{tikzpicture}
=
\sum_q
\begin{tikzpicture}[baseline=-.65ex]
\node[ext] (v) at (0,0) {};
\node[ext] (w) at (1,0) {};
\node (i1) at (0.3,.5) {$\scriptstyle e^q$};
\node (i2) at (0.7,-0.5) {$\scriptstyle e^{q*}$};
\draw (v) edge +(-.5,0) edge +(-.5,.5) edge +(-.5,-.5) 
(w) edge +(.5,0) edge +(.5,.5) edge +(.5,-.5);
\draw[dotted] (v) edge (i1) (w) edge (i2); 
\end{tikzpicture}
	\end{equation}
	Here the $e^q$ run over a basis of $H_g$, while $e^{q*}$ run over the corresponding dual basis.
	Furthermore, we omit decorations by the unit element $1$, i.e., we formally identify decorations by $1$ with no decoration.
	The product is graphically interpreted by superposing graphs at their vertices.
\begin{equation}\label{equ:prodgra}
\begin{tikzpicture}[baseline=-.65ex]
\node[ext] (v1) at (0,0) {1};
\node[ext] (v2) at (0.5,0) {2};
\node[ext] (v3) at (1,0) {3};
\node[ext] (v4) at (1.5,0) {4};
\node[ext] (w1) at (0.5,.7) {5};
\node[ext] (w2) at (1.0,.7) {6};
\draw (v1) edge (v2) edge (w1) (v2) edge (w1) (v3) edge (w1) ;
\end{tikzpicture}
\wedge
\begin{tikzpicture}[baseline=-.65ex]
\node[ext] (v1) at (0,0) {1};
\node[ext] (v2) at (0.5,0) {2};
\node[ext] (v3) at (1,0) {3};
\node[ext] (v4) at (1.5,0) {4};
\node[ext] (w1) at (0.5,.7) {5};
\node[ext] (w2) at (1.0,.7) {6};
\draw   (v2)  edge (w2) (v3)  edge (w2) (v4) edge (w2);
\end{tikzpicture}
=
\begin{tikzpicture}[baseline=-.65ex]
\node[ext] (v1) at (0,0) {1};
\node[ext] (v2) at (0.5,0) {2};
\node[ext] (v3) at (1,0) {3};
\node[ext] (v4) at (1.5,0) {4};
\node[ext] (w1) at (0.5,.7) {5};
\node[ext] (w2) at (1.0,.7) {6};
\draw (v1) edge (v2) edge (w1) (v2) edge (w1) edge (w2) (v3) edge (w1) edge (w2) (v4) edge (w2);
\end{tikzpicture}
\end{equation}
Furthermore, $\BVGra_{H_g}$ is a cooperadic right $\BVc$-comodule.
Concretely, consider the coaction $\Delta_T$ corresponding to the tree 
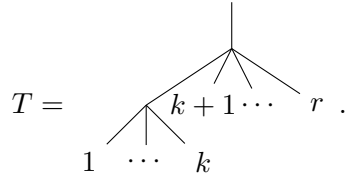
\[
T=
\begin{tikzpicture}[baseline=-.65ex,scale=.5]
\node at (0,3) {}
child {
	child{ child { node {$1$}}
		child { node {$\cdots$}} 
		child { node {$k$}} 
	}
	child{ node {$k+1$} }
	child{ node {$\cdots$} }
	child{ node {$r$} }
};
\end{tikzpicture}
\,.
\]
Then for $\Gamma\in\BVGra_{H_g}$ one sets
\[
\Delta_T \Gamma = \sum_{\gamma\subset \Gamma \atop V\gamma=\{1,\dot,k\}}\pm \underbrace{(\Gamma / \gamma)}_{\in \BVGra(r-k+1)} \otimes \underbrace{\prod_{(ij)\in E\gamma}\omega_{ij}}_{\in \BVc(k)},
\]
where the sum is over subgraphs $\gamma\subset \Gamma$ with vertex set $V\gamma=\{1,\dots,k\}$, not containing any $H_g$-decorations.
Then $\Gamma/\gamma$ is obtained by contracting the subgraph $\gamma$ to one new vertex. Finally, the product on the right is over the set of edges $E\gamma$ of $\gamma$ and the $\omega_{ij}$ are the generators of $\BV^c(k)$ as in \eqref{equ:BVcpresentation}, with $\omega_{ii}\coloneqq\theta_i$.

One of the main results of~\cite{CamposWillwacher2016} is that there is a map of collections of dgcas 
\[
A : \BVGra_{H_g} \to \Omega_{PA}(\FFM_{\Sigma_g})
\]
into the piecewise semialgebraic
forms\footnote{We refer to \cite{HardtLambrechtsTurchinVolic2011} for more details on semi-algebraic forms, and to \cite[Section 2]{CamposWillwacher2016} for the technical reasons why they are needed. However, for the present paper they will be of no relevance, and the reader can safely just consider them as some version of differential forms.} on the framed configuration space $\FFM_{\Sigma_g}$, which is also (in the sense of Section \ref{sec:model_definition}) compatible with the cooperadic comodule structures.
Note that $\BVGra$ is a quasi-free graded commutative 
algebra generated by the graphs with exactly one edge.
Hence to define $A$ it suffices to specify the image of one edge graphs.

To this end one picks a degree 1 differential form $\omega\in \Omega_{PA}(\FM_{\Sigma_g}(2))$, the \emph{propagator}, with the following properties.
\begin{itemize}
	\item $d\omega= \sum_j (\pi_1^*e^j)\wedge (\pi_2^*e^{j*})$ for some chosen representatives $e^j$ of a basis of $H_g$, with $e^{j*}$ the representatives of the Poincar\'e dual elements.
	\item The restriction $\omega\mid_{\p \FM_{\Sigma_g}(2)}\eqqcolon\eta$ is a fiberwise volume form on the unit tangent bundle $\p \FM_{\Sigma_g}(2)$. We may assume that this volume form is the one induced by a suitable Riemannian metric on $\Sigma_g$.
\end{itemize}
For a construction of $\omega$ see \cite[Sections 2 and 9]{CamposWillwacher2016}.
We may then construct the dgca map
\[
A : \BVGra_{H_g}(r) \to \Omega_{PA}(\FFM_{\Sigma_g}(r))
\]
by sending an edge between two vertices $i$ and $j$ to $\pi_{ij}^*\omega$, a tadpole at $i$ to $\pi_i^*\eta$, and extending to an algebra map, i.e., 
\begin{equation}\label{equ:AformulaR}
A(\Gamma) = \pm \prod_{(i,j)\in E\Gamma\atop i\neq j}\pi_{ij}^*\omega \wedge \prod_{(i,i)\in E\Gamma} \pi_i^*\eta \wedge \prod_{\alpha\text{ decoration at $i$}}\pi_i^*\alpha.
\end{equation}
Here $\pi_i$ (resp. $\pi_{ij}$) are the forgetful maps \eqref{equ:piFFMSigma} that forget all points except $i$ (resp. $i$ and $j$) from the configuration. Furthermore, the $\alpha$ in the last product runs over all $\bar H_g$-decorations in the graph, and is to be interpreted as the chosen representative in $\Omega_{PA}(\Sigma_g)$ of the cohomology class $\alpha$.

Note however, that $A$ is not a quasi-isomorphism.
To repair this one twists $\BVGra_{H_g}$. We will describe this combinatorially.
We consider a graph complex $\BVGraphs_{\Sigma_g}(r)$ whose elements are linear combinations of graphs similar to those before, but with two different kinds of vertices. There are $r$ numbered ``external'' vertices and an arbitrary (finite) number of unlabelled ``internal'' vertices, which we draw as black dots in pictures.
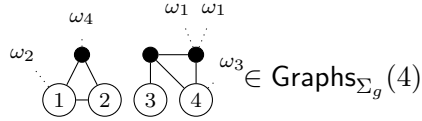
\[
\begin{tikzpicture}[scale=1.2]
\node[ext] (v1) at (0,0) {$\scriptstyle 1$};
\node[ext] (v2) at (.5,0) {$\scriptstyle 2$};
\node[ext] (v3) at (1,0) {$\scriptstyle 3$};
\node[ext] (v4) at (1.5,0) {$\scriptstyle 4$};
\node[int] (w1) at (.25,.5) {};
\node[int] (w2) at (1.5,.5) {};
\node[int] (w3) at (1,.5) {};
\node (i1) at (1.7,1) {$\scriptstyle \omega_1$};
\node (i2) at (1.3,1) {$\scriptstyle \omega_1$};
\node (i3) at (-.4,.5) {$\scriptstyle \omega_2$};
\node (i4) at (1.9,.4) {$\scriptstyle \omega_3$};
\node (i5) at (0.25,.9) {$\scriptstyle \omega_4$};
\draw (v1) edge (v2) edge (w1) (w1)  edge (v2) (v3) edge (w3) (v4) edge (w3) edge (w2) (w2) edge (w3);
\draw[dotted] (v1) edge (i3) (w2) edge (i2) edge (i1) (v4) edge (i4) (w1) edge (i5);
\node at (3,.2) {$\in \Graphs_{\Sigma_g}(4)$};
\end{tikzpicture}
\]
We require in addition:
\begin{itemize}
	\item (Connectivity condition) Every connected component has at least one external vertex.
	\item The internal vertices do not have tadpoles. 
\end{itemize}

The graded commutative product by gluing at external vertices \eqref{equ:prodgra} naturally extends to $\Graphs_{\Sigma_g}(r)$.
We can furthermore construct a map of graded commutative algebras
\[
F : \BVGraphs_{\Sigma_g}(r) \to \Omega_{PA}(\FFM_{\Sigma_g}(r))
\]
by sending a graph $\Gamma$ with $k$ internal vertices to the fiber integral
\begin{equation}\label{equ:fiberintR}
F(\Gamma) = \int_{\FFM_{\Sigma_g}(r+k)\to \FFM_{\Sigma_g}(r)} A(\Gamma'),
\end{equation}
where $\Gamma'$ is obtained by labelling the internal vertices by $r+1,\dots,r+k$ and adding tadpoles on those vertices.
(We are slightly imprecise with signs here, but refer to \cite{CamposWillwacher2016} for a more detailed treatment.)

One can furthermore define a differential on $\Graphs_{\Sigma_g}$ such that the map $F$ intertwines differentials.
The form of the differential is essentially dictated by Stokes' Theorem. More precisely, one can check that 
\[
d= d_{s} + d_{c}
\]
where $d_s$ removes an edge and inserts a diagonal as in \eqref{equ:dongra} above, and $d_c$ contracts an edge
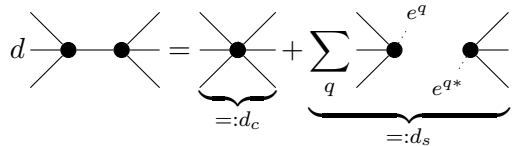
\[
d\, 
\begin{tikzpicture}[baseline=-.65ex]
\node[int] (v) at (0,0) {};
\node[int] (w) at (0.7,0) {};
\draw (v) edge +(-.5,0) edge +(-.5,.5) edge +(-.5,-.5) edge (w)
(w) edge +(.5,0) edge +(.5,.5) edge +(.5,-.5);
\end{tikzpicture}
=
\underbrace{
\begin{tikzpicture}[baseline=-.65ex]
\node[int] (v) at (0,0) {};
\node[int] (w) at (0,0) {};
\draw (v) edge +(-.5,0) edge +(-.5,.5) edge +(-.5,-.5) edge (w)
(w) edge +(.5,0) edge +(.5,.5) edge +(.5,-.5);
\end{tikzpicture}
}_{=:d_c}
+
\underbrace{
\sum_q
\begin{tikzpicture}[baseline=-.65ex]
\node[int] (v) at (0,0) {};
\node[int] (w) at (1,0) {};
\node (i1) at (0.3,.5) {$\scriptstyle e^q$};
\node (i2) at (0.7,-0.5) {$\scriptstyle e^{q*}$};
\draw (v) edge +(-.5,0) edge +(-.5,.5) edge +(-.5,-.5) 
(w) edge +(.5,0) edge +(.5,.5) edge +(.5,-.5);
\draw[dotted] (v) edge (i1) (w) edge (i2); 
\end{tikzpicture}
}_{=:d_s}
\]
The key point is now that if the piece $d_s$ potentially violates the connectivity condition above, because it can produce a new connected component without external vertices. In this case, by convention, the component is removed, and replaced by a number, obtained by applying \eqref{equ:fiberintR} to the component.
Hence the differential combinatorially depends on a map 
\[
Z : G_{H_g} \to \R,
\]
from the complex of connected graphs without external vertices to numbers. For physical reasons $Z$ is called partition function in \cite{CamposWillwacher2016}\footnote{In loc. cit. the graph complex $G_{H_g}$ is denoted $^*\mathrm{GC}_{H_g}$.}, and we shall follow this notation. We just emphasize that $Z$ is given by an integral
\begin{equation}\label{equ:Zintegral}
Z(\Gamma) = \int_{\FFM_{\Sigma_g}} A(\Gamma'),
\end{equation}
and hence generally transcendental and complicated to compute. 
If we want to emphasize the role of $Z$ in the definition of the differential of $\BVGraphs_{\Sigma_g}$, we will use the notation $\BVGraphs_{\Sigma_g}^Z$.
	
At the end, one can check that the resulting map $F: \BVGraphs_{\Sigma_g} \to \Omega_{PA}(\FFM_{\Sigma_g})$
is a quasi-isomorphism.
Furthermore, $\BVGraphs_{\Sigma_g}$ carries a natural coaction of a resolution $\BVGraphs_2$ of $\BVc$, and the map $F$ is compatible with these coactions.
Hence, the authors of \cite{CamposWillwacher2016} have built a real model of the configuration spaces of points on $\Sigma_g$.
There are however two question left open by \cite{CamposWillwacher2016}:
\begin{enumerate}
	\item The model depends on the partition function $Z$, which is not so easy to evaluate.
	\item The construction is performed over $\R$, not $\Q$.
\end{enumerate}
We settle both shortcomings in this paper.
In particular, we show that for the partition function one can take the following:
\begin{equation}\label{equ:Ztriv}
Z^{triv}(\Gamma)=
\begin{cases}
0 & \text{if $\Gamma$ has more than one vertex} \\
\int_{\Sigma_g} \alpha_1\cdots \alpha_s & \text{if $\Gamma$ has is one vertex decorated by $\alpha_1\cdots \alpha_s$} 
\end{cases}\,.
\end{equation}

	
	\subsection{Outline of proof of Theorem \ref{thm:main}}\label{sec:proof_outline}
	We will show Theorem \ref{thm:main} by establishing the following zigzag of quasi-isomorphisms of right $\BVc$-comodules
	\[
	\Mo_g
	\xleftarrow{\simeq}
	\BVGraphs_{\Sigma_g}^{Z_{triv}} 
	\xrightarrow{\simeq}
	\HH_g=
	\hotimes^{(1,1)\cdots (g,g)}_{\BVc_{C}}
	\BVc_{g,g}
    \xleftrightarrow[\text{Prop. }\ref{prop:thinmodel}]{\simeq}
    \Omega^{*}(\FFM_{\Sigma_{g}}).
	\]
	Here $\BVGraphs_{\Sigma_g}^{Z_{triv}}$ is the graphical right $\BVc$-comodule $\BVGraphs$ from before, but defined over $\Q$, and we emphasize that the partition function $Z$ entering the definition of the differential is taken to be \eqref{equ:Ztriv}.
	The left-hand map in the zigzag is then just the projection, sending all diagrams with internal vertices to zero. It already appeared in similar form in \cite[Section 6]{Idrissi2018b} and \cite[Appendix A]{CamposWillwacher2016} and is known to be a quasi-isomorphism.
	
	Also, we already know that the third object in the zigzag is a model for $\FFM_{\Sigma_g}$ from Proposition \ref{prop:thinmodel}.
	Hence it is enough to construct the second arrow.
    The formula defining that map will be formally very similar to \eqref{equ:fiberintR}, except that all integrals can be combinatorially evaluated. At worst, integrals of polynomials over simplices need to be computed explicitly.
	
	Concretely, the construction of the map $F:\BVGraphs_{\Sigma_g}\to \HH_g$ is in two steps, following the previous subsection.
	Let $\Gamma\in \BVGraphs_{\Sigma_g}(r)$ be a graph with $k$ internal vertices.
	Then we will first define an element 
	\[
	A(\Gamma) \in \HH_g'(r+k).
	\]
	We furthermore define a formal ``fiber integration'' map $\int: \HH_g'(r+k)\to \HH_g'(r)\subset \HH_g(r)$ that satisfies a version of the Stokes formula. In fact, this map is essentially combinatorial and defined over $\Q$, involving only integrals of polynomial forms over simplicies.
	Finally, our graph $\Gamma$ is sent to
	\[
	F(\Gamma)\coloneqq  \int A(\Gamma).
	\]
	The key point here is the following.
	Implicitly, $\BVGraphs_{\Sigma_g}$ still depends on a partition function $Z$ as before. However, with our formal integrals the expression \eqref{equ:Zintegral} can be fully evaluated and shown to agree with \eqref{equ:Ztriv}.

	\subsection{Some elements of $\BVc(1)$}
	\label{sec:cohom reps}
	We want to define a map of dg Hopf collections (and in fact $\BV^c$ comodules)
	\[
	A: \BVGra_{H_g} \to \HH_g',
	\]
	where $\BVGra_V$ is as recalled in section \ref{sec:recoll-citec-anoth}, but considered over $\Q$ instead of $\R$, and $H_g=H^*(\Sigma_g)$.
	To this end we want to adapt \eqref{equ:AformulaR}.
	Hence we need to pick suitable representatives of the cohomology $H_g$ of $\Sigma_g$ in $\HH_g(1)$, and a propagator.
	In this section we shall do the former.
	To this end we consider a fixed basis of $H_g$,
	\[
	1, a^1,\dots,a^g, b^1,\dots,a^g ,\nu,
	\]
	defined such that the $a^j$ and $b^j$ have degree 1, $\nu$ has degree 2, and the Poincar\'e duality pairing satisfies
	\[
	1=\int_{\Sigma_g} a^j b^j = \int_{\Sigma_g}  \nu,
	\]
	with all other pairings zero.
	The goal is then to define corresponding elements $a^1,\dots,a^g, b^1,\dots,a^g ,\nu\in \HH_g(1)$, abusively denoted by the same letters.
	To define these elements, we will use the notation and geometric language of section \ref{sec:HHg_notation}.
	\begin{itemize}
		\item We define the degree 1 element
		\begin{equation}\label{equ:ajdef}
		a^j = e dt_j \in \HH_g'(1)
		\end{equation}
	    where 
		\[
		e=1
		\in \BVc_C(1)
		\]
		is the unit element
		and $t_j$ is the longitudinal coordinate along the $j$-th handle.
		The form $a^j$ supported on the $j$-th handle, i.e., on all other strata the form is zero.
		It is clearly continuous, since the restriction to the boundary ($t_j=0$ or $t_j=1$) vanishes.
		
		In fact, later on we will omit units like $e$ above from formulas, if no confusion can arise.
		\item We define the degree 1 element 
		\[
		b^j \in  \HH_g'(1)
		\]
		as follows:
		\begin{itemize}
			\item On the $j$-th handle we set $b^j=\omega_{1*}\in \BVc_C(1)$,
			using the notation of section  \ref{sec:BVcnotation} for elements of $\BVc_C(r)$.
			\item On the other handles we set $b^j=0$.
			\item On the bulk stratum we set
			\[
			b^j=
			\begin{cases}
			\omega_{1\underline 1} & \text{for $j=1$} \\
			\omega_{1\underline j} - \omega_{1\overline j} & \text{for $j\geq 2$} 
			\end{cases}
			\in \BVc_{(g,g)}(1),
			\]
			using the notation of section  \ref{sec:BVcnotation} for elements of $\BVc_{(g,g)}(r)$.
		\end{itemize}
		One can check that this is continuous.
		For $j=1$. The only nontrivial cocompositions of $\omega_{1\underline 1}$ are when the point $1$ approaches the fixed vertex $\underline 1$, and $\infty$.
		These precisely produce the two boundary values of $b^1$ on the first handle.
		
		For $j\geq 2$ the non-trivial cocompositions are when the point approaches either the fixed vertex $2i-2$ or $2j-1$.
		They produce the two boundary values of $b^j$ on the $j$-th handle. Here the additional sign comes in by applying the involution \eqref{equ:IBVc1} on the $t_j=1$-end of the handle.
		Also note that there is no contribution of our point moving to $\infty$ in the bulk stratum, since the two terms in $b^j$ cancel each other.
	\item We finally define the degree 2 element $\nu= a^1 b^1$, which will play the role of a volume form on $\Sigma_g$.
	\end{itemize}

Note that none of the above forms involved the framings.
Since we are working with framed configuration spaces, the Euler class $2\nu - 2\sum_{k=1}^ga^k b^k$, which is cohomologous to $(2-2g)\nu$, has a primitive $\eta\in \HH_g(1)$, given by a fiberwise volume form.
Concretely, we define this form as follows:
\begin{itemize}
	\item On the bulk stratum we set
	\[
	\eta = \theta_1 - \omega_{1\underline 1} \in \BVc_{(g,g)}(1).
	\]
	\item On the first handle we set 
	\[
	\eta=\theta_1-\omega_{1*} \in \BVc(1).
	\]
	\item On the $j$-th handle, with $j\geq 2$, we set 
	\[
	\eta = \theta_1 -2t_j \omega_{*1},
	\]
	where $t_j$ is again the longitudinal coordinate along the $j$-th handle.
\end{itemize}
We note that $\eta$ is continuous, as is easily checked.
For example, if one passes from to the $t_j=1$-end of the $j$-th handle the form becomes $\theta_1 -2\omega_{*1}$. This has to be compared with the corresponding cocomposition of the bulk term, which is $\theta_1$. However, we need to apply the involution \eqref{equ:IBVc1}, and hence obtain that both expressions agree, as desired.

One also checks easily that 
\[
d\eta = 2\nu - 2\sum_{k=1}^ga^k b^k
\]
as desired. In particular, note that both the right- and the left-hand sides are supported on the handles $2,\dots, g$.

Furthermore, the definition above is such that $\eta$ is the boundary-value of the propagator $\omega\in \HH_g'(2)$ as the two points collide, with $\omega$ to be defined in the next section.

	\subsection{The propagator}
	
	We shall again use the notation of section \ref{sec:HHg_notation} to define elements of $\HH_g'(2)$.
	In particular we shall construct now the propagator $\omega\in \HH_g'(2)$.
	We have to specify $\omega$ on the 14 types of ``strata'' of our space, as described in section \ref{sec:HHg12_example_strata}.
	\begin{itemize}
		\item (Stratum $A$) If both point 1 and 2 are in the bulk, we set
		\begin{equation}\label{equ:omega_A}
		\omega = \omega_{12} -\frac 1 2 (\omega_{1\underline 1}+\omega_{2\underline 1}) \in \BVc_{(g,g)},
		\end{equation}
		using again the notation of section \ref{sec:BVcnotation}.
		\item (Stratum $B_2$) If point 1 is in the bulk and point 2 on the first handle (with longitudinal coordinate $t_1^{(2)}$, see Figure \ref{fig:quenches}) the form is 
		\begin{equation}\label{equ:omega_B2}
		\omega = (\omega_{1\underline 1}) (\frac 1 2-t_1^{(2)}) + (\omega_{2*}) (-\frac 1 2 +t_1^{(2)} )
		\in \BVc_{(g,g)}(\{1\})\otimes \BVc_1(\{2\})\otimes \Omega_{poly}(\Delta^1).
		\end{equation}		
		
		\item (Stratum $C_2$) If instead point 2 is on the $j$-th handle, $j=2,\dots g$, then the form is  
		\begin{equation}\label{equ:omega_C2}
		(\omega_{1\underline j})(1-t_j^{(2)})
		+
		(\omega_{1\overline j})( t_j^{(2)})
		-\frac 1 2 (\omega_{1\underline 1}).
		\end{equation}		
		\item (Strata $C_1C_2$) If one point is on the $i$-th handle and the other on the $j$-th, with $j\neq i$ and $i,j\geq 2$, we set $\omega$ to be zero.
		
		\item (Strata $B_1C_2$) If point 1 is on the first handle and point 2 on the $j$-th, with $j\geq 2$, we set 
		\begin{equation}\label{equ:omega_B1C2}
		\omega=(-\frac 1 2 +t_1^{(1)})\omega_{1*}.
		\end{equation}

		\item (Stratum $B_{(12)}$) If both points are together on the first handle and infinitesimally close, then 
		\begin{equation}\label{equ:omega_B_12}
		\omega = \omega_{12} -\frac 1 2(\omega_{1*} +\omega_{2*}).
		\end{equation}
		\item (Stratum $C_{(12)}$) If both are again infinitesimally close and on the $j$-th handle, $j\geq 2$, then $\omega$ is
		\begin{equation}\label{equ:omega_C_12}
		\omega_{12} - t_j (\omega_{1*}+\omega_{2*}).
		\end{equation}
		Note that here $t_j^{(1)}=t_j^{(2)}\eqqcolon t_j$.
		
		\item (Stratum $B_{12}$) If both points are together on the first handle, with coordinates $t_1^{(1)}<t_1^{(2)}$, but not infinitesimally close, then the form is 
		\begin{equation}\label{equ:omega_B12}
		(\pi_1^* f) (\frac 1 2+ t_1^{(1)}- t_1^{(2)})
		+
		(\pi_2^* f)(-\frac 1 2 + t_1^{(2)}- t_1^{(1)})
		\end{equation}
		
		\item (Stratum $C_{12}$) If both points are together on the $j$-th handle ($j\geq 2$), with coordinates $t_j^{(1)}<t_j^{(2)}$, but not infinitesimally close, then the form is 
		\begin{equation}\label{equ:omega_C12}
		(\omega_{1*} ) (1- t_j^{(2)})
		+
		(\omega_{2*}) (- t_j^{(1)})
		\end{equation}
		
		\item (Strata $B_1$, $C_1$, $B_{21}$, $C_{21}$, $B_2C_1$) We define $\omega$ to be symmetric under the $S_2$-action, so that the remaining cases are determined. For example, if point 1 is on the $j$-th handle and point 2 on the first (strata $B_2C_1$) we set, symmetrically to \eqref{equ:omega_B1C2},
		\begin{equation}\label{equ:omega_B2C1}
		\omega=(\frac 1 2 -t_1^{(2)})\omega_{2*}.
		\end{equation}
	\end{itemize}
	
	\begin{lemma}\label{lem:prop}
		The degree $1$ element $\omega\in \HH_g(2)$ is well defined, symmetric under the $S_2$-action and satisfies the equation
		\begin{equation}\label{equ:domega}
		d\omega = \pi_1^*\nu+\pi_2^*\nu - \sum_{k=1}^g (\pi_1^* a^k \pi_2^*b^k+\pi_2^* a^k \pi_1^*b^k).
		\end{equation}
	\end{lemma}
	\begin{proof}
	The element is symmetric under the interchange of point 1 and 2 by definition. We next need to check that it is well-defined as an element of the totalization $\HH_g(2)$. In the geometric language of section \ref{sec:HHg_notation} this means that have to check that the values of $\omega$ we gave on each of the 14 types of top-(i.e., 4-)dimensional strata are continuous across strata boundaries.
	More precisely, for each pair of strata in Figure \ref{fig:strata} connected by a line in the diagram, we have to check that the two restrictions of the forms on the 4 dimensional pieces to their 3-dimensional intersection are the same.
	In principle we need to check 24 cases. Fortunately, some of them are covered by the symmetry under interchange of point 1 and 2, or are otherwise trivial.
	\begin{itemize}
		\item Neighbors of stratum $A$:
		Stratum $A$ has 6 neighbors, as can be seen from the diagram of Figure \ref{fig:strata}. First, one of the points, say the 2nd, can move to one side of one handle. Hence one needs to take the corresponding cooperadic cocomposition of \eqref{equ:omega_A}, and check that it agrees with the evaluation of \eqref{equ:omega_B2} (respectively \eqref{equ:omega_C2}) at $t_j^{(2)}=0$ or $t_j^{(2)}=1$. This is easily accomplished and handles the four diagonal neighbors of $A$ in the diagram of Figure \ref{fig:strata}.
		For the two horizontal neighbors both points together approach one handle. Again one takes the corresponding cocomposition of \eqref{equ:omega_A}, and compares to the evaluation of \eqref{equ:omega_B_12} (respectively \eqref{equ:omega_C_12}) at $t_j=0$ and $t_j=1$.
		We note that in particular, on the $j$-th handle with $j\geq 2$, at $t_j=1$ one needs to apply the involution \eqref{equ:IBVc1} to \eqref{equ:omega_C_12}, and this explains the presence of the second summand in the case $t_j=1$.
		\item Next we discuss the neighbors of $B_{12}$.
		Concretely, this means we have two points on the first handle, with longitudinal coordinates $t_1^{(1)}\leq t_1^{(2)}$.
		There are three boundary strata, corresponding to $t_1^{(1)}\to 0$, $t_1^{(2)}\to 1$ and $t_1^{(1)}\to t_1^{(2)}$.
		 We hence need to compare the evaluations of \eqref{equ:omega_B12} at these boundary points 
		 \begin{align*}
		 &(\omega_{1*} ) (\frac 1 2- t_1^{(2)})
		 +
		 (\omega_{2*})(-\frac 1 2+ t_1^{(2)})
		 \\
		 &
		 		 (\omega_{1*}) (-\frac 1 2+t_1^{(1)})
		 +
		 (\omega_{2*})( \frac 1 2- t_1^{(1)})
		 \\
		 &
		 		 \frac 1 2 \omega_{1*} -\frac 1 2 \omega_{2*} 
		 \end{align*}
		 to the corresponding cocomposition of \eqref{equ:omega_B2}, or the corresponding cocomposition of the same element with $1$ and $2$ interchanged, or to the cocomposition of \eqref{equ:omega_B_12}, respectively.
		 The expressions agree.
		 \item By symmetry we have hence covered the three intersections of stratum $B_{21}$ with its neighbors as well.
		 \item The analysis for stratum $C_{12}$ and its three neighbors is similar. Here the boundary values of \eqref{equ:omega_C12} at $t_j^{(1)}=0$, $t_j^{(2)}=1$ and $t_j^{(1)}=t_j^{(2)}=t_j$ are
		  		 \begin{align*}
		  &(\omega_{1*}) (1- t_j^{(2)})
		  &
		  &
		  (\omega_{2*})( - t_1^{(1)})
		  &
		  &
		  (\omega_{1*})(1-t_j)+(\omega_{2*})( - t_j)
		   \,.
		  \end{align*}
		  This is to be compared the corresponding cocompositions 
		  of \eqref{equ:omega_C2}, the same element with 1 and 2 interchanged and \eqref{equ:omega_C_12}.
		  For example, cocomposing \eqref{equ:omega_C2}, its second and third summands become zero, since here point 1 ``approaches'' the marked point $\underline j$, away from the fixed vertices $\overline j$ and $\underline 1$. We hence see that the expression agrees with the first of the three above.
		  If instead point 1 approaches the marked point $\overline j$ the first and third summands become zero, and if we interchange 1 and 2 we obtain $(\omega_{2*})( t_1^{(1)})$. Applying the involution of \eqref{equ:IBVc1} we then obtain the second expression of the three above, as desired.
		  The third case is checked similarly.
		  \item By symmetry, we have also handled the stratum $C_{21}$, and its three neighbors.
		  \item Next look at stratum $B_2C_1$, where we have defined $\omega$ as in \eqref{equ:omega_B2C1}.
		  The boundary values at $t_j^{(1)}=0$ and $t_j^{(1)}=1$ are given by the same formula.
		  We need to compare this with the cocomposition of \eqref{equ:omega_B2} such that point 1 approaches the marked points $\underline j$ or $\overline j$ with $j=2,\dots,g$. The first summand in \eqref{equ:omega_B2} then becomes zero, and the second is just \eqref{equ:omega_B2C1}
		  
		  Similarly one considers the intersection of $B_1C_2$ with $C_2$.
		  To this end we start with \eqref{equ:omega_C2}. We first consider the cocomposition obtained by letting point 1 approach the fixed vertex 1. The result is $-\frac 1 2 \omega_{1*}$.
		  This agrees with the evaulation of \eqref{equ:omega_B1C2} at $t_1^(1)=0$ as desired.
		  Next we look at \eqref{equ:omega_C2} and let vertex 1 go to $\infty$ (i.e., take the corresponding cocomposition).
		  This yields
		  \[
		  (1-t_j^{(2)})(\omega_{1*}) +t_j^{(2)}(\omega_{1*})-\frac 1 2 (\omega_{1*})
		  =
		  \frac 1 2 (\omega_{1*}).
		  \]
		  This agrees with the evaluation of \eqref{equ:omega_B1C2} at $t_1^{(1)}=1$ as desired.
		  \item Finally, consider the strata $C_1C_2$, with point 1 on the $i$-th and point 2 on the $j$-th handle, with, $i\neq j$, $i,j\geq 2$.
		  Here the form $\omega$ is zero, so we need to check that the corresponding cocompositions of $\eqref{equ:omega_C2}$ are as well. Indeed, when point 1 approaches the marked points $\underline i$ or $\overline i$, all terms become zero, as desired.
	\end{itemize}
	Overall, we have checked that $\omega\in \HH_g'(2)$ is a well-defined element.
	Let us compute its differential. Again we have to do this separately on each of the strata, and verify in each case that \eqref{equ:domega} holds. 
	\begin{itemize}
		\item On stratum $A$ the differential is zero, as is the right-hand side of \eqref{equ:domega}.
		\item On $B_2$ the differential is (cf. \eqref{equ:omega_B2})
		\[
		(dt_1^{(2)})(\omega_{2*})  - (dt_1^{(2)})(\omega_{1\underline 1})
		\]
		The first term is the restriction of $\pi_2^*(a^1 b^1)=\pi_2^*\nu$ to the first handle, and the second one is $- \pi_2^*a^1\pi_1^*b^1$.
		Note that $\pi_1^* a^1$ and hence also $\pi_1^* \nu$ vanish on the first handle.
		By symmetry, \eqref{equ:domega} also holds on the stratum $B_1$.
		\item On the stratum $C_2$ (i.e., point 2 is in the $j$-th handle with $j\geq 2$) one obtains from \eqref{equ:omega_C2} that $d\omega$ is 
		\[
		(dt_j^{(2)} )(-\omega_{1\underline j}+\omega_{1\overline j}) .
		\]
		This is just the restriction of $-\pi_2^* a^j \pi_1^*b^j $ to this stratum. The other terms on the right-hand side of \eqref{equ:domega} all restrict to zero on $C_2$, so that \eqref{equ:domega} holds there.
		By symmetry, \eqref{equ:domega} also holds on the stratum $C_1$.
		\item On $B_{12}$ we have that (cf. \eqref{equ:omega_B12})
		\[
		d\omega = (\pi_1^*a^1 -\pi_2^*a^1)(\pi_1^*b^1-\pi_2^*b^1) 
		\]
		In agreement with \eqref{equ:domega}.
		By symmetry \eqref{equ:domega} also holds on $B_{21}$.
		\item On $C_{12}$ we have (cf. \eqref{equ:omega_C12})
		\[
		d\omega = -\pi_2^*a^j \pi_1^*b^j - \pi_1^*a^j\pi_2^*b^j ,
		\]
		in agreement with \eqref{equ:domega}, which hence also holds on $C_{21}$ by symmetry.
		\item On $B_{(12)}$ we have  (cf. \eqref{equ:omega_B_12})
		\[
		d\omega = 0,
		\]
		and this agrees with the right-hand side of \eqref{equ:domega}, which reads on this stratum
		\[
		(\pi_1^*a^1 -\pi_2^*a^1) (\pi_1^*b^1-\pi_2^*b^1) =0,
		\]
		since here $\pi_1^*a^1 =\pi_2^*a^1$.
		\item On $C_{(12)}$ we have  (cf. \eqref{equ:omega_C_12})
		\[
		d\omega = -\pi_2^*a^j \pi_1^* b^j -\pi_1^*a^j\pi_2^*b^j,
		\]
		as desired.
		\item On $B_1C_2$ we obtain (cf. \eqref{equ:omega_B1C2}) we have
		\[
		d\omega = dt_1^{(1)}\omega_{1*} = \pi_1^*\nu.
		\] 
		This agrees with the right-hand side of \eqref{equ:domega}.
		The case $B_2C_1$ follows by symmetry.
      \item On the strata of type $C_1C_2$ we have $\omega=0$, as is the right-hand side of \eqref{equ:domega}.
        \qedhere
	\end{itemize}
	\end{proof}

	For later reference we shall also note that the $\BVc$ coaction $\Delta:\HH_g'(2)\to\HH_g'(1)\otimes \BVc(2)$ applied to $\omega$ produces
	\begin{equation}\label{equ:delta_omega_eta}
	\Delta \omega = 
	\eta \otimes 1 + 1\otimes \omega_{12},
	\end{equation}
	as can be easily seen from the definitions.
	In fact, $\eta$ is defined precisely to make this equation hold.
	
	Now let us return to the definition of our map $A:\BVGra_{H_g}(r)\to \HH_g'(r)$. It is sufficient to describe the map on commutative algebra generators. 
	If we denote, as in section \ref{sec:recoll-citec-anoth}, an edge in $\BVGra_{H_g}$ between vertices $i$ and $j$ by $\omega_{ij}$ and a decoration $\alpha\in \bar H_g$ at vertex $j$ by $\alpha_j$, then our map $A$ is defined on generators as
	\begin{align}\label{equ:feynman_rules}
	a_j^k &\mapsto \pi_j^* a_k &
	b_j^k &\mapsto \pi_j^* b_k &
	\nu_j&\mapsto \pi_j^* \nu &
	\omega_{ij} &\mapsto \pi_{ij}^* \omega &
	\omega_{jj} &\mapsto \pi_j^* \eta .
	\end{align}
	
	From the discussion above the following is then obvious.
	\begin{lemma}\label{lem:def-A}
		This prescription yields a well defined map of collections of dgcas
		\[
		A : \Gra_{H_g} \to \HH'_g.
		\]
	\end{lemma}

	\subsection{Combinatorial fiber integrals on $\HH_g$}
	\subsubsection{Construction of the fiber integral}
	\label{sec:fibintdef}

	We next define morphisms
	\[
	\int_k : \HH_g'(r+k) \to \HH_g'(r),
	\]
	which we shall interpret as a combinatorial version of a fiber integral, and which satisfy a version of the Stokes' Theorem, see Proposition \ref{prop:stokes} below.
	
	Concretely, this will be defined by the following formula.
	Given $\alpha\in \HH_g'(r+k)$ we need to define $\int_k\alpha$. This will again be an element in the end. 
	We shall again use the notation of section \ref{sec:HHg_notation} to define this element.
	Recall that we need to define $\int_k\alpha$ on every ``stratum'', 
	which means that (see \eqref{equ:HHgstrata}) for every pair $\sS\coloneqq (\underline r, f)$ with $f: S=\{1,\dots,r\}\to B(\underline r)$ we need to provide an element 
	\[
	\left(\int_k\alpha\right)_{\sS}\in 
	\BVc_{(g,g)}(f^{-1}(0)) \otimes \bigotimes_{i=1}^n\bigotimes_{j=1}^{r_j} \BVc_1(f^{-1}((i,j))) \otimes \Omega_{poly}(\Delta^{r_1}\times \cdots \times \Delta^{r_g}), 
	\]
	and we need to check that the collection of such elements specified 
	satisfies suitable boundary conditions so as to assemble to an element in the end $\HH_g'(S)$.
	
	Now let $K=\{r+1,\dots,r+k\}$, so that $\alpha\in \HH_g'(S\cup K)$.
	We will consider strata $\sS'=(\underline r', f')$ for $r+k$ points that satisfy the following conditions:
	\begin{itemize}
		\item Forgetting the last $k$ points we end up in stratum $\sS$, where we define the forgetful map on $\sS'$ in the obvious way.
		\item Each of the $k$ last points is on a handle (i.e., not in the bulk), and is the only element of its group on this handle. More precisely,this means that we have $f'(i)=(p,q)$ with $(f')^{-1}((p,q)) =\{i\}$ for all $i=r+1,\dots, r+k$.,
	\end{itemize}
	For concreteness we have that 
	\[
	\underline r =(r_1+k_1,\dots, r_g+k_g), 
	\]
	where the conditions above imply that $k_1+\cdots+k_g=k$.
	Then the value of our given form $\alpha\in \HH_g'(S\cup K)$ 
	on the stratum $\sS'$ as above is
	\[
	\alpha_{\sS'}\in 
	\BVc_{(g,g)}(f^{-1}(0)) 
	\otimes \bigotimes_{i=1}^n\bigotimes_{j=1}^{r_j} \BVc_C(f^{-1}((i,j)))
	\otimes
	\underbrace{  
	\BVc_C(1)^{\otimes k}
	}_{\text{from $K$}} \otimes \Omega_{poly}(\Delta^{r_1+k_1}\times \cdots \times \Delta^{r_g+k_g}). 
	\] 
	Now note that we have the degree -1 map
	\begin{equation}\label{equ:pBVQdef}
	\p : \BVc_C(1) \to \Q
	\end{equation}
	projecting onto the ``Lie cobracket'' cogenerator.
	Furthermore, the forgetful map, forgetting the last $k$ points, induces a map 
	\[
	\pi_\Delta: \Delta^{r_1+k_1}\times \cdots \times  \Delta^{r_g+k_g} 
	\to \Delta^{r_1}\times \cdots \times  \Delta^{r_g}.
	\] 
	Since polynomial forms are preserved under pushforwards along simplicial maps, we can consider the fiber integral
	\begin{equation}\label{equ:intpiDeltadef}
	\int_{\pi_\Delta}=(\pi_\Delta)_{!}:  \Omega_{poly}(\Delta^{r_1+k_1}\times \cdots \times \Delta^{r_g+k_g} )
	\to  \Omega_{poly}(\Delta^{r_1}\times \cdots \times \Delta^{r_g}).
	\end{equation}
	Note that this integral is defined over $\Q$ and essentially combinatorial.
	We thus define the map 
	\begin{multline*}
	\int_{\sS'\to \sS}:
		\BVc_{(g,g)}(f^{-1}(0)) 
	\otimes \bigotimes_{i=1}^n\bigotimes_{j=1}^{r_j} \BVc_C(f^{-1}((i,j)))
	\otimes 
	\underbrace{  
		\BVc_C(1)^{\otimes k}
	}_{\text{from $K$}} \otimes \Omega_{poly}(\Delta^{r_1+k_1}\times \cdots \times  \Delta^{r_g+k_g})
\\\to 
\BVc_{(g,g)}(f^{-1}(0)) \otimes \bigotimes_{i=1}^n\bigotimes_{j=1}^{r_j} \BVc_C(f^{-1}((i,j))) \otimes \Omega_{poly}(\Delta^{r_1}\times \cdots \times  \Delta^{r_g})
	\end{multline*}
	to be given by $\int_{\sS'\to \sS}=\mathit{id}\otimes (\p)^{\otimes k}\otimes \int_{\pi_\Delta}$. 
	In words, it is the identity on the first few factors, the factors ``from $K$'' are projected to $\Q$ via $\p$, and on the remaining factors we apply the fiber integral on simplices.

	Finally, we define 
	\begin{equation}\label{equ:int_def_sector}
	\left(\int_k\alpha\right)_{\sS}
	=
	\sum_{\sS' \supset \sS} \int_{\sS'\to \sS} \alpha_{\sS'},
	\end{equation}
	where the sum is over all strata $\sS'$ over $\sS$ as above.

	\begin{lemma}
		For each $\alpha \in\HH_g'(S\cup K)$ the element $\int_k\alpha\in \HH_g'(S)$ is well defined, i.e., continuous in the language of section \ref{sec:HHg_notation}. 
	\end{lemma}
\begin{proof}
	We have to check that our collection $\left(\int_k\alpha\right)_{\sS}$ satisfies the continuity equations \eqref{equ:HHgcontinuity1}-\eqref{equ:HHgcontinuity3} of section \ref{sec:HHg_notation}.
	In fact, they follow directly from the continuity equations for $\alpha$, and the fact that out fiber integration map suitably commutes with the restriction to the boundary of the simplices (left-hand side) and the cocomposition (right-hand side).
\end{proof}
	
	As a first property we note the following.
	\begin{lemma}[Fiberwise Fubini's Theorem]\label{lem:fubini}
		For $\alpha\in \HH_g'(r+k)$ and $\beta \in \HH_g'(r+l)$ we have that 
		\[
		\int_{k+l} \alpha \wedge \beta
		=
		\pm \left( \int_{k} \alpha \right)\wedge \left(\int_l \beta\right).
		\]
	\end{lemma}
\begin{proof}
	This just follows from the usual Fubini formula for the fiber integral over the simplices.
\end{proof}
	
	\subsubsection{The (co-)Leibniz identity for the cooperad $\BVc$ and related objects}
	To prepare for the proof of the fiberwise Stokes' Theorem in the next subsection, we need some recollections about the cooperad $\BVc$ and the objects $\BVc_C$ and $\BVc_{(g,g)}$.
	
	First consider the cooperad $\BVc$, and the space of $r+1$-ary cooperations $\BVc(r+1)$.
	For $1\leq i\neq j\leq r+1$ consider the cocompositions 
	\[
	\Delta_{ij} : \BVc(r+1) \to \BVc(2) \otimes\BVc(r)
	\]
	described by the tree
	\[
	\begin{tikzpicture}[scale=.5]
	\node {}
		child {
			child{ child { node {$i$}}
				     child { node {$j$}} 
				    }
			child{ node {$1$} }
			child{ node {$\cdots$} }
			child{ node {$\hat i$} }
			child{ node {$\cdots$} }
			child{ node {$\hat j$} }
			child{ node {$\cdots$} }
			child{ node {$r+1$} }
	};
	\end{tikzpicture}
	\]
	In geometric terms ``points $i$ and $j$ approach each other''.
	Similarly, consider the cocompositions 
	\[
	\Delta_{0j} : \BVc(r+1) \to \BVc(2) \otimes\BVc(r)
	\]
	described by the tree 
	\[
		\begin{tikzpicture}[scale=.5]
	\node {}
	child {
		 child { node {$j$}} 
		 child {
			child{ node {$1$} }
			child{ node {$\cdots$} }
			child{ node {$\hat j$} }
			child{ node {$\cdots$} }
			child{ node {$r+1$} }
		}
	};
	\end{tikzpicture}
	\]
	In other words ``point $j$ approaches $\infty$''.
	Let $\pi:  \BVc(2)\to \Q$ be the projection to the cobracket cogenerator.
	Geometrically, this can be described as 
	\[
	\alpha \mapsto \pi\alpha\coloneqq \int_{\FFM_2(2)} \alpha \wedge \theta_1\theta_2.
	\]
	We then describe, for $0\leq i\neq j\leq r+1$, the following operators of degree $-1$
	\[
	\p_{ij} : \BVc(r+1) \to \BVc(r),
	\]
	defined as the composition
	\[
	\BVc(r+1) \xrightarrow{\Delta_{ij}} 
	 \BVc(2) \otimes\BVc(r)
	 \xrightarrow{\pi \otimes \mathit{id}} \BVc(r).
	\]
	By a simple computation, one than has:
	\begin{lemma}[co-Leibniz identity]\label{lem:leibniz1}
		For any $\alpha\in \BVc(r+1)$ we have that 
		\[
		\p_{0,r+1} \alpha = \sum_{i=1}^r \p_{i,r+1}\alpha.
		\]
	\end{lemma}
	Similarly, note that for any right $\BVc$-comodule $\op M$ we can define the operations 
	\[
	\p_{ij} : \op M(r+1) \to \op M(r)
	\] 
	by (essentially) the same formula as above.

	Next, consider our model for configurations of points on the cylinder $\BVc_C$, and the objects $\BVc_{(p,q)}$ from above.
	In particular we have $p$ left $\BVc_C$-coactions and $q$ right $\BVc_C$-coactions, and additionally coactions by the cooperad $\BVc$.
	Let us label the $p$ left $\BVc_C$-coations by $\bar 1,\dots \bar p$, and the $q$ right $\BVc_C$-coactions by $\underline 1,\dots \underline q$.
	For $a\in \{\bar 1,\dots \bar p, \underline 1,\dots \underline q\}$ and $j\in \{1,\dots,r+1\}$ we then have the cocompositions
	\[
	\Delta_{aj}\BVc_{(p,q)}(r+1) \to \BVc_C(1) \otimes \BVc_{(p,q)}(r),
	\]
	such that point $j$ ``approaches $a$''.
	As before, we extend the definition of the projection to the cobracket to
	\[
	\pi :  \BVc_C(1) \to \Q.
	\]
	And we define, just as before, the operations $\p_{aj}$ by composing the operation $\Delta_{aj}$ with $\pi$ applied to the first factor in the tensor product.
	Note that the operations $\p_{ij}$ from the right $\BVc$-comodule structure are still defined on $\BVc_{(p,q)}$.
	We then have the following result, which is in fact nothing but a disguised form of Lemma \ref{lem:leibniz1}
	\begin{lemma}[co-Leibniz identity]\label{lem:leibniz2}
		For any $\alpha\in \BVc_{(p,q)}(r+1)$ we have that 
		\[
		\sum_{j=1}^p \p_{\bar j,r+1} \alpha 
		=
		\sum_{k=1}^q\p_{\underline k,r+1} \alpha 
		+
		 \sum_{i=1}^r \p_{i,r+1}\alpha.
		\]
		In particular, for $\BVc_C\cong \BVc_{(1,1)}$ this specializes to
		\[
		\p_{0,r+1} \alpha = \p_{*,r+1}\alpha + \sum_{i=1}^r \p_{i,r+1}\alpha,
		\]
		where the index $0$ corresponds to using the left coaction, $*$ corresponds to the right coaction and $i$ the right $\BVc$-coactions.
	\end{lemma}

	\subsubsection{Stokes formula}
	We want to show that our combinatorial fiber integral satisfies a version of Stokes' formula. To formulate the statement we need the additional degree -1 operations (boundary operators in some sense)
	\[
	\p_{ij} : \HH_g'(r+k) 
	\to \HH_g'(r+k-1)
	\]
	defined in the previous subsection for any right $\BVc$-comodule.
	Then our result is the following.
	
	\begin{proposition}[Combinatorial Stokes formula]\label{prop:stokes}
		The combinatorial fiber integral $\int_k:\HH_g'(r+k) \to \HH_g'(r)$ of section \ref{sec:fibintdef} satisfies the following Stokes' type formula 
		\begin{equation}\label{equ:stokes}
		d\int_k  \alpha \mp \int_k d\alpha=
		 \sum_{i=r+1}^{r+k}
		 \sum_{j=1}^{i-1}
		 \int_{k-1}
		 \p_{ij} \alpha\, .
		\end{equation}
	\end{proposition}
	\begin{proof}
		Recall the definition of the fiber integral \eqref{equ:int_def_sector}.
		Note that the first ingredient, the map \eqref{equ:pBVQdef}, commutes with the differential, simply because the differential is zero on $\BVc_C$.
		The second ingredient, the fiber integral on simplices \eqref{equ:intpiDeltadef}, satisfies the ordinary Stokes' formula
		\[
		d\int_{\pi_\Delta} \beta-\int_{\pi_\Delta} d\beta = 
		\int_{\p} \beta\mid_\p,
		\]
		where $\beta\in \Omega_{poly}(\Delta^{r_1+k_1}\times \cdots \times \Delta^{r_1+k_1})$, the integral on the right-hand side is over the fiber-wise boundary, and $\beta\mid_\p$ is the restriction of $\beta$ to said boundary. In fact, we will henceforth set $\beta\mid_\p=\beta$ for simplicity.
		We shall again think of the simplex $\Delta^{r_i+k_i}$ as a space of configurations of $r_i+k_i$ points on the unit interval.
		Similarly, the base $\Delta^{r_i}$ corresponds to $r_i$ points on the interval, and the fiber of the forgatful map $\Delta^{r_i+k_i}\to \Delta^{r_i}$ corresponds to a product of smaller simplices, corresponding to the forgotten points moving between the points of the base configuration.
		There are hence three possible types of fiberwise boundary strata:
		\begin{itemize}
			\item A moving point collides with one of the fixed points from the base configuration.
			\item A moving point collides with one of the ends of the unit interval.
			\item Two moving points collide with each other.
		\end{itemize}
		Let us compute the contribution to the right-hand side of \eqref{equ:stokes} in each of the three cases above.
		First, fix the colliding point, say $p\in \{r+1,\dots, r+k\}$.
		The restriction to the relevant boundary of our $\alpha_{\sS'}$ (with $\sS'$ as in \eqref{equ:int_def_sector}) is then of the form $\alpha_{\sS'}\mid_{\p_j\Delta^{r_i}}$.
		By the continuity equation \eqref{equ:HHgcontinuity1} we can equate this to the cocomposition $\Delta\alpha_{\sS''}$.
		We note that in \eqref{equ:int_def_sector} this term is subjected to a further application of the projection map $\p$.
		Furthermore, there occur two similar boundary terms, one in which our point $p$ approaches the packet from the left and one where it approaches from the right.
		Hence we are precisely in the situation that we can apply the co-Leibniz identity for $\BVc_C$, see Lemma \ref{lem:leibniz2}.
		
		Before we go on, we consider also the second type of fiberwise boundary, namely when our moving point $p$ in the fiber hits the endpoints 0 or 1 of the interval. Arguing as before, but using the continuity equations \eqref{equ:HHgcontinuity2}, \eqref{equ:HHgcontinuity3} in this case, we see that these terms can be seen as cocompositions $\Delta \alpha_{\sS'''}$, now acting on the bulk term.
		Again we can apply the co-Leibniz identiy, but this time for $\BVc_{(g,g)}$, see Lemma \ref{lem:leibniz2}.
		Overall, combining these terms, we arrive at 
		\[
		d\int_k  \alpha \mp \int_k d\alpha=
		 \sum_{i=r+1}^{r+k}
		\sum_{j=1}^{r}
		\int_{k-1}
		\p_{ij} \alpha
		+(X),
		\]
		 where $(X)$ denotes the yet untreated boundary terms of the third kind.
		 Those can be treated in a parallel manner, again using the co-Leibniz identity for $\BVc_C$, and thus the result follows.
	
	\end{proof}
	
	\subsection{The map and end of the proof of Theorem \ref{thm:main} for $g\geq 1$}\label{sec:themapF}
	We finally define our map $F:\BVGraphs_{\Sigma_g}\to \HH_g$ on a graph $\Gamma\in\BVGraphs_{\Sigma_g}(r)$ with $k$ internal vertices as
	\[
	F(\Gamma) = \int_{k} A(\Gamma).
	\]
	We note that this map is defined over $\Q$, since in the worst case one needs to compute integrals of rational polynomial forms over simplices.
	
	\begin{proposition}
		The map $F$ is a well-defined map of right dg Hopf $\BVc$-comodules.
	\end{proposition}
	\begin{proof}
		We first check that the map is compatible with the differentials.
		To this end temporarily define the partition function (cf. section \ref{sec:recoll-citec-anoth})
		\[
		Z : G_{H_g} \to \mathbb Q
		\]
		for $\Gamma\in G_{H_g}$ a graph with $k$ vertices as
		\[
		Z(\Gamma) = \int_k \Gamma,
		\]
		using our combinatorial fiber integral.
		Then it is immediate from the Stokes formula (Proposition \ref{prop:stokes}) that the map 
		$$
		F: \BVGraphs_{\Sigma_g}^Z \to \HH_g
		$$ 
		commutes with the differentials.
		We hence just have to check that the partition function is equal to $Z_{triv}$ of \eqref{equ:Ztriv}.
		This step is a technical computation and will be postponed and conducted in the next subsection, see Lemma \ref{lem:partition_function} below. For now, we assume this done, and hence obtain a map of collections of dg vector spaces 
		\[
		F: \BVGraphs_{\Sigma_g}=\BVGraphs_{\Sigma_g}^{Z_{triv}} \to \HH_g.
		\]
		
		The map $F$ is a morphism of dg Hopf collections. This immediately follows from the symmetric group-equivariance of the definition, and from the ``Fubini Theorem'', Lemma \ref{lem:fubini}.
		
		Finally we need to check that the map $F$ is compatible with the $\BVc$-comodule structure.
		We hence have to check that the following diagrams commute
		\begin{equation}\label{equ:coactioncd}
		\begin{tikzcd}
		\BVGraphs_{\Sigma_g}(r+s-1)\ar{r}{F} \ar{d}{\Delta} & \HH_g(r+s-1) \ar{dd}{\Delta} \\
		\BVGraphs_{\Sigma_g}(r)\otimes \BVGraphs_2(s) \ar{d}{\mathit{id}\otimes p} & \\
		\BVGraphs_{\Sigma_g}(r)\otimes \BVc(s)  \ar{r}{F\otimes\mathit{id}}& \HH_g(r)\otimes \BVc(s)
		\end{tikzcd}\,.
		\end{equation}
		On the left-hand side we have written the $\BVc$-coaction as the composition of the natural $\BVGraphs_2$-coaction and the projection $p:\BVGraphs_2\to \BVc$.
		In fact, it suffices to show the diagrams obtained from the above by the composition with the projection from $\BVc$ to the space of (cooperadic) cogenerators of $\BVc$.
		Concretely, $\BVc$ is cogenerated by the ``co-BV operator'' $\BVDelta$ in $\BVc(1)$ of degree 1 and the coproduct of degree 0 in $\BVc(2)$, and the ``forgetful'' cooperation in arity $0$.
			\ricardoline{We should be precise at least once in the paper about arity 0, both for $\BVc$ and for $\FFM_2$. 
			Maybe in the Generalities section.}
		Let us begin with the first case, for which we consider the diagram above for $s=1$, and we are only interested in those terms that produce the BV cogenerator in the $\BVc(1)$-factor.
		Consider concretely a graph $\Gamma\in \BVGraphs_{\Sigma_g}(r)$, and say we take the $\BVc$-coaction at the first vertex (w.l.o.g.). 
		Denote by $\Delta_1$ the composition of this coaction with the projection of $\BVc(1)$ to the degree 1 cogenerator. By definition, if the first vertex in $\Gamma$ has a tadpole than
		\[
		\Delta_1(\Gamma) = \pm \Gamma' \otimes \Delta,
		\]
		where $\Gamma'$ is obtained from $\Gamma$ by removing the tadpole. Otherwise, $\Delta_1(\Gamma)=0$.
		On the other hand, $\Delta_1F(\Gamma)$ operates, separately on each stratum, by removing a ``tadpole'' $\theta_1$ as well, and otherwise sending the form to zero. Eventually, this means we only need to check the following properties of our map $F$:
		\begin{itemize}
			\item If $\Gamma$ has no tadpole at the external vertex $1$, then $F(\Gamma)$ has no terms containing $\theta_1$. This is obviously true: The only term introducing $\theta_1$'s is $\pi_1^*\eta$ above.
			\item On each stratum, the form $\eta$ to which the tadpole is sent, contains $\theta_1$ with coefficient 1. This is also clearly true, see section \ref{sec:cohom reps}.
		\end{itemize}
		
		Next we consider the coaction of the coproduct cogenerator.
		To this end we consider again diagram 
		\eqref{equ:coactioncd}, for $s=2$.
		Furthermore we project $\BVc(2)$ to the degree 0 part, or equivalently only pick out terms in the coactions that contribute to this part.
		Again we start with a graph $\Gamma\in \BVGraphs_{\Sigma_g}(r+1)$, and say we take the $\BVc$-coaction at the first two vertices (w.l.o.g.).
		Denote by $\Delta_{12}$ this coaction followed by the projection to the coproduct cogenerator in $\BVc(2)$.
		On the graphical side, we then have
		\[
		\Delta_{12}(\Gamma) = \Gamma_{12},
		\]
		where $\Gamma_{12}$ is obtained from $\Gamma$ by fusing external vertices 1 and 2 into one vertex.
		In particular, if $\Gamma$ had an edge connecting vertices 1 and 2, then the graph $\Gamma$ has a tadpole at the fusion vertex.
		Let us first assume that $\Gamma$ has no edge between vertices 1 and 2.
		
		On the other side, $\Delta_{12}F(\Gamma)$ is obtained by applying the respective coaction to the various factors $\BVc_{(g,g)}$ or $\BVc_C$ in the totalization. 
		This means, in the geometric language above, that one considers only strata in which points 1 and 2 are ``infinitesimally close'', i.e., in the same copy of $\BVc_{(g,g)}$ or $\BVc_C$ in the tensor product.
		But the fiber integral above then sees 1 and 2 as the same point, so that our cocomposition (and projection) indeed is given by $F(\Gamma_{12})$ as desired.
		If $\Gamma$ does have an edge between vertices 1 and 2 not much changes. Neither the edge nor the tadpole contribute in the fiber integration, since they are not incident to internal vertices.
		We merely have to check that $\Delta_{12}\omega=\eta$, which follows from \eqref{equ:delta_omega_eta}.
		
		Finally we look at the ``forgetful'' coaction, and consider diagram \eqref{equ:coactioncd} for $s=0$.
		The corresponding cooperation raises the arity by 1. On the graphical side, it acts on $\Gamma\in \BVGraphs_{\Sigma_g}(r-1)$ by adding an $r$-th external vertex of valency 0 to $\Gamma$. On $\alpha\in \HH_g(r-1)$ the corresponding action is notationally awkward to define, but there is a natural way to extend the ``form'' $\alpha$ to an element of $\HH_g(r-1)$ that just ``does not depend'' on the $r$-th point. It is also clear that the map $F$ respects these operations, since the fiber integrals do not interact with the additional $r$-th vertex.
		
		Overall, we have shown that $F$ is a well-defined map of dg Hopf $\BVc$-comodules.
	\end{proof}

If we look at the outline of the proof of Theorem \ref{thm:main}, then to actually get a proof for $g\geq 1$ the only thing left to check is the following.

\begin{proposition}\label{prop:Fqiso}
	The map $F$ is a quasi-isomorphism.
\end{proposition}
\begin{proof}
We already know that $H(\BVGraphs_{\Sigma_g})=H(\HH_g)=H(\FFM_{\Sigma_{g}})$. We only need to check that the maps
 \[
 F: \BVGraphs_{\Sigma_g}(r) \to \HH_g(r)
 \] 
induce isomorphisms on cohomology.
We do this by induction on $r$.
For $r=1$ the cohomology of both sides is the cohomology of the frame bundle 
\[
H(\FFM_{\Sigma_g}(1))
=
\begin{cases}
\Q \oplus \Q^{2g}[-1] \oplus \Q[-1] \oplus \Q[-2] & \text{for $g=1$} \\
\Q \oplus \Q^{2g}[-1] & \text{otherwise.}
\end{cases}
\]
Suppose first that $g\geq 2$.
Then representatives of the above classes in $\HH_g(1)$ are given by the forms $1$, $a^j$ and $b^j$ from section \ref{sec:cohom reps}.
We have similar 1-vertex graph representatives in $\BVGraphs_{\Sigma_g}$ given by the diagrams
\begin{align*}
&\begin{tikzpicture}[baseline=-.65ex]
\node[ext] (v) {$\scriptstyle 1$};
\end{tikzpicture}
&
&\begin{tikzpicture}[baseline=-.65ex]
\node[ext, label=90:{$a^j$}] (v) {$\scriptstyle 1$};
\end{tikzpicture}
&
&\begin{tikzpicture}[baseline=-.65ex]
\node[ext, label=90:{$b^j$}] (v) {$\scriptstyle 1$};
\end{tikzpicture}\, .
\end{align*}
It is clear from the definition of $F$ that it sends the latter set of cohomology generators to the former.
Next consider the torus, $g=1$. 
Then the additional two cohomology classes are represented in $\HH_g'(1)$ by the forms $\eta$ and $\nu$ from section \ref{sec:cohom reps}.
Corresponding representatives in $\BVGraphs_{\Sigma_g}(1)$ are given by the cocycles
\begin{align*}
&\begin{tikzpicture}[baseline=-.65ex]
\node[ext] (v) {$\scriptstyle 1$};
\draw (v) edge[loop] (v);
\end{tikzpicture}
+
\begin{tikzpicture}[baseline=-.65ex]
\node[ext] (v) at (0,0) {$\scriptstyle 1$};
\node[int, label=90:{$\scriptstyle 2\nu-2a^1b^1$}] (i) at (0,.7) {};
\draw (v) edge (i);
\end{tikzpicture}
&
&\begin{tikzpicture}[baseline=-.65ex]
\node[ext, label=90:{$\nu$}] (v) {$\scriptstyle 1$};
\end{tikzpicture}
\end{align*}
Again, it is immediate that the map $F$ sends the generators onto each other. Note in particular that the second graph above is sent to zero by $F$ (in fact already by the map $A$) since we have that $2\nu-2a^1b^1=0$ in $\HH_1(1)$.

Next consider the induction step $r\to r+1$, for $r\geq 1$.
Geometrically, note that we have the fibration 
\[
X \to \FFM_{\Sigma_g}(r+1) \to \FFM_{\Sigma_g}(r),
\]
by forgetting the first point. The fiber $X$ can be identified with the compactified configuration space of 1 framed point in the surface with $r$ points removed, $\Sigma_g\setminus *^{\sqcup r}$.
Its cohomology is 
\begin{equation}\label{equ:HXfiber}
H(X) = \Q \oplus \Q[-1]^{2g+r}.
\end{equation}
By the previous Proposition we have the following commutative diagram of dgcas
\[
\begin{tikzcd}
\BVGraphs_{\Sigma_g}(r) \ar{r}{F}\ar{d}& \HH_g(r) \ar{d}\\
\BVGraphs_{\Sigma_g}(r+1)\ar{r}{F} & \HH_g(r+1).
\end{tikzcd}\,
\]
We assume by induction that the upper horizontal arrow is a quasi-isomorphism.
To check that the lower horizontal arrow is one as well it suffices to check that the induced map between the vertical homotopy cofibers 
\[
\BVGraphs_{\Sigma_g}(r+1)\hotimes_{\BVGraphs_{\Sigma_g}(r)} \Q
\to 
\HH_g(r+1)\hotimes_{\HH_g(r)} \Q
\]
is a quasi-isomorphism.
We realized the homotopy tensor products using the bar resolutions as usual.
Both sides have cohomology $H(X)$ as in \eqref{equ:HXfiber}, we just need to check that sets of representatives are mapped to each other.
For the right-hand side we may take the following cohomology generators:
\begin{itemize}
	\item The first $1+2g$ generators are 
	\[
	1,\pi_1^*a^j, \pi_1^*b^j\in \HH_g(r+1)\subset \HH_g(r+1)\hotimes_{ \HH_g(r)} \Q,
	\] 
	for $j=1,\dots,g$.
	\item Define the elements
	\[
	e_{j}
	=
	\pi_{1j}^*\omega
	-
	(1, \pi_j^*\nu)
	+
	\sum_{q=1}^g
	(\pi_1^*a^q, \pi_j^*b^q)
	+
	(\pi_1^*b^q, \pi_j^*a^q)
	\in \HH_g(r+1)\hotimes_{\HH_g(r)} \Q.
	\]
	Then $de_j=\pi_1^* \nu$.
	
	We can then take $e_k-e_2$ for $k=3,\dots,r+1$ as representatives for $r-1$ of the missing $r$ cohomology classes. (Geometrically, these are the classes from the punctures.)
	
	Finally, there is one class from the framing. For $g=1$ it is represented by $\pi_1^*\eta$.
	For $g\geq 2$ that latter class is not closed but satisfies
	\[
	\pi_1^*\eta =-2\sum_{j=2}^g a^jb^j
	\]
	We will define in section \ref{sec:integral_example} below an explicit degree $1$ element $\alpha\in \HH_g(1)$ so that (see \eqref{equ:dalpha_example})
	\[
	d\alpha = (2g-2)\nu-2\sum_{j=2}^g a^jb^j.
	\]
	Then the missing cohomology representative can be taken to be $\pi_1^*\eta -\alpha- (2-2g)e_2$.
\end{itemize}  
The corresponding representatives in $\BVGraphs_{\Sigma_g}(r+1)\hotimes_{\BVGraphs_{\Sigma_g}(r)} \Q$ are built analogously:
\begin{itemize}
	\item 
	Define the graphs 
	\[
	\gamma_\beta^{i,r}=
	\begin{tikzpicture}[baseline=-.65ex]
	\node[ext] (v1) at (0,0) {$\scriptstyle 1$};
	\node  at (.7,0) {$\cdots$};
	\node[ext,label=90:{$\beta$}] (vi) at (1.4,0) {$\scriptstyle i$};
	\node  at (2.1,0) {$\cdots$};
	\node[ext] (vr) at (2.8,0) {$\scriptstyle r$};
	\end{tikzpicture} \in \BVGraphs_{\Sigma_g}(r)
	\]
	For the degree 0 and 1 classes coming from the cohomology of $M$ we can just take
 $1,\gamma_{a_j}^{1,r+1},\gamma_{b_j}^{1,r+1}$ for $j=1,\dots,g$.
	\item Consider the graphs (with an edge between vertices $i$ and $j$)
	\[
	e_{ij}
	=
	\begin{tikzpicture}[baseline=-.65ex]
	\node[ext] (v1) at (0,0) {$\scriptstyle 1$};
	\node at (.7,0) {$\cdots$};
	\node[ext] (vi) at (1.4,0) {$\scriptstyle i$};
	\node at (2.1,0) {$\cdots$};
	\node[ext] (vj) at (2.8,0) {$\scriptstyle 1$};
	\node at (3.5,0) {$\cdots$};
	\node[ext] (vr) at (4.2,0) {$\scriptstyle r+1$};
	\draw (vi) edge[bend left] (vj);
	\end{tikzpicture}\,.
	\]
	We define the elements in the derived tensor product
	\[
	e_j' =e_{1j}
	-
	(1, \gamma_\nu^{j,r})
	-\sum_{q=1}^g
	(\gamma^{1,r+1}_{a_q}, \gamma^{j-1,r}_{b_q})
	-
    (\gamma^{1,r+1}_{b_q}, \gamma^{j-1,r}_{a_q}).
	\]
	They satisfy $de_j'=\gamma_{\nu}^{1,r+1}$.
	Then we can write down the cohomology generators $e_k'-e_2'$, for $k=3,\dots r+1$.
	
	Next define the sum of graphs
	\[
	\eta'=
	\begin{tikzpicture}[baseline=-.65ex]
	\node[ext] (v1) at (0,0) {$\scriptstyle 1$};
	\node[ext] (v2) at (.7,0) {$\scriptstyle 2$};
	\node at (1.4,0) {$\cdots$};
	\node[ext] (vr) at (2.1,0) {$\scriptstyle r+1$};
	\draw (v1) edge[loop] (v1);
	\end{tikzpicture}
	\, +
	\begin{tikzpicture}[baseline=-.65ex]
	\node[ext] (v) at (0,0) {$\scriptstyle 1$};
	\node[int, label=90:{$\scriptstyle 2\nu-2a^1b^1$}] (i) at (0,.7) {};
	\draw (v) edge (i);
	\node[ext] (v1) at (0,0) {$\scriptstyle 1$};
	\node[ext] (v2) at (.7,0) {$\scriptstyle 2$};
	\node at (1.4,0) {$\cdots$};
	\node[ext] (vr) at (2.1,0) {$\scriptstyle r+1$};
	\end{tikzpicture}
	\,.
	\]
	For the last cohomology generator we can then take $\eta'-(2-2g)e_2'$.
\end{itemize}
By similar arguments as in the induction base before, we then check that $F$ indeed sends the above sets of generators onto each other.
In particular, for the ``framing'' class this uses the computation of section \ref{sec:integral_example}.
\end{proof}

\begin{remark}\label{rem:Fvanishing}
	We note that the integral $F(\Gamma)$ is zero for all but a very special subset of graphs. 
	Concretely, note that in order for the integral \eqref{equ:int_def_sector} to have a chance to be nonzero, every one of the $k$ point integrated out must come with a corresponding form piece in the longitudinal coordinates $dt_j^{(i)}$. However, looking at the ``Feynman rules'' \eqref{equ:feynman_rules} we see that the only terms that can contribute such a form are in fact the decorations $a^j$, or $\nu$ at the internal vertex. Hence $F(\Gamma)=0$ unless every internal vertex is decorated by exactly one $a^j$ or $\nu$.
	
	Furthermore, if a graph $\Gamma$ has an internally connected component of loop order $s\geq 2$, then also $F(\Gamma)=0$.
	The reason is, geometrically speaking, that the say $k$ internal vertices correspond to a $2k$-dimensional space to be integrated over. However, if we count the form degrees that depend only on those vertices we get at least
	\[
	k+k+s-1>2k,
	\]
	with the first $k$ being from the decorations, and $k+s-1$ the number of edges internal to that component.
	By the same argument line, we also see that in an internally connected component, at most one vertex can have decorations of degree $2$ to have $F(\Gamma)\neq 0$, and if so, the component must be an internal tree.
	
	Furthermore, if two internal vertices that are connected are decorated by $a_i$ and $a_j$ for different $i,j\geq 2$, then also $F(\Gamma)=0$. This is because the form associated to $a_j$ is concentrated on the $j$-th handle and the form $\omega$ associated to the edge between our vertices vanishes when both involved vertices are on different handles with indices $\geq 2$ (strata $C_1C_2$ above). 
	
	There are also certain more obvious constraints, for example if a single vertex is decorated by $a_1a_2$ then $F$ vanishes since $a_1a_2=0$ in $\HH_g(1)$.
\end{remark}

\subsection{Evaluation of the partition function and vanishing results}
	The remaining step is the following.
	\begin{lemma}\label{lem:partition_function}
		Let $\Gamma\in G_{H_g}$ be a connected graph with $k$ vertices.
		Then 
		\[
		\int_k A(\Gamma) = Z_{triv}(\Gamma).
		\]
	\end{lemma}
	In words, the integral is zero unless $\Gamma$ consists of exactly one vertex, and in that case the integral is just the pairing of the decorations at that vertex.
	
	Before we show the result, let us show the following.
	\begin{lemma}\label{lem:vanishing}
		We have that $F(\Gamma) = 0$ for graphs $\Gamma$ that contain one of the following patterns:
		\begin{itemize}
			\item An internal vertex with exactly one incoming edge and no other decoration.
			
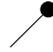
\[
			\begin{tikzcd}
			\node[int] (v) at (0,0) {};
			\draw (v) edge +(-.5,-.5); 
			\end{tikzcd}
			\]
			\item An internal vertex with exactly one incoming edge and one decoration
			
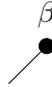
\[
			\begin{tikzpicture}
			\node[int, label=90:{$\beta$}] (v) at (0,0) {};
			\draw (v) edge +(-.5,-.5); 
			\end{tikzpicture}
			\]
			\item An internal vertex with exactly two edges and no decoration
			
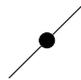
\[
			\begin{tikzpicture}
			\node[int] (v) at (0,0) {};
			\draw (v) edge +(-.5,-.5) edge +(.5,.5); 
			\end{tikzpicture}
			\]
		\end{itemize} 
	\end{lemma}
	\begin{proof}
		The first and third case immediately follow by Remark \ref{rem:Fvanishing} above that graphs with internal vertices without decorations by $a^j$ or $\nu$ have vanishing weight.
		
		For the second case, due to Lemma \ref{lem:fubini}, it suffices to check that  $F(\Gamma_{\beta})=0$ for the graph 
		\[
		\Gamma_{\beta}=
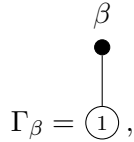
		\begin{tikzpicture}[baseline=-.65ex]
		\node[ext] (v) at (0,0) {$\scriptstyle 1$};
		\node[int,label=90:{$\beta$}] (w) at (0,1) {};
		\draw (v) edge (w);
		\end{tikzpicture}\, ,
		\]
		with $\beta=a^j$ or $\beta=\nu$.
		
		We need to consider $F(\Gamma_{\beta})_{\sS}$ for every stratum $\sS$.
		Let us denote by 1 the ``fixed'' point on the base, and by 2 the ``moving'' point whose position we integrate over, corresponding to the internal vertex in the graph.
		First, suppose point 1 is in the bulk. 
		Then, the only nontrivial contributions come from stratum $B_2$ if $\beta=a^1$ or $\beta=\nu$, or strata of type $C_2$ (with point 2 being on the $j$-th handle), with if $\beta=a^j$. 
		For the $B_2$ contributions we get, looking at the \eqref{equ:omega_B2} 
		\begin{align*}
		F(\Gamma_{a_1})_{\sS} &=  \int_0^1 dt_1^{(2)} (-\frac 1 2+t_1^{(2)}) =0
		&
		F(\Gamma_{\nu})_{\sS} &= (\pi_1^*f_1) \int_0^1 dt_1^{(2)} (\frac 1 2-t_1^{(2)}) =0.
		\end{align*}
		For the $C_2$ contribution we get $0$ since \eqref{equ:omega_C2} does not contribute a term $\omega_{*2}$ that would be necessary to produce a nonzero integral.
		
		Next suppose that point $1$ is on the first handle.
		In the cases $\beta=a^1$, $\beta=\nu$ we have contributions from point 2 on the first handle and they are (cf. \eqref{equ:omega_B12})
		\begin{align*}
		F(\Gamma_{a_1})_{\sS} 
		&=  \int_0^{t_1^{(1)})} dt_1^{(2)} (\frac 1 2 +t_1^{(2)} -t_1^{(1)}) 
			+\int_{t_1^{(1)}}^1dt_1^{(2)}(-\frac 1 2 +t_1^{(2)} -t_1^{(1)})
			=\frac 1 2 (t_1^{(1)}-(t_1^{(1)})^2 -t_1^{(1)}+(t_1^{(1)})^2)=0
		\\
		F(\Gamma_{\nu})_{\sS} &= (\pi_1^*f_1) \left(
		\int_0^{t_1^{(1)})} dt_1^{(2)} (-\frac 1 2 -t_1^{(2)} +t_1^{(1)})
		 +
		 \int_{t_1^{(1)}}^1dt_1^{(2)}(\frac 1 2 -t_1^{(2)} +t_1^{(1)})
		\right)=0.
		\end{align*}
		
		If $\beta=a^j$, $j\geq 2$, we do not get any contributions, because there is no term $\omega_{*2}$ in formula \eqref{equ:omega_B1C2}.
		
		Finally suppose that point $1$ is on the $j$-th handle, $j\geq 2$.
		If $\beta=a^1$ we get the contribution (cf. \eqref{equ:omega_B1C2})
		\[
		\int_0^1dt_1^{(2)} (\frac 1 2 -t_1^{(2)})=0.
		\]
		We do not get any contributions if $\beta=\nu$, by \eqref{equ:omega_B1C2}.
		Finally, for $\beta=a_j$ we have to compute (cf. \eqref{equ:omega_C12})
		\[
		\int_0^{t_j^{(1)}}dt_j^{(2)} (1-t_j^{(1)})
			-
			\int_{t_j^{(1)}}^1dt_j^{(2)} (-t_j^{(1)})
				=0.
		\]
	\end{proof}

	\begin{proof}[Proof of Lemma \ref{lem:partition_function}]
		First, if the graph $\Gamma$ has exactly one vertex it is an easy verification that the formula is correct. We just need to show that otherwise the integral is zero, so assume the connected graph $\Gamma$ has at least two vertices.
		
		The arguments of Remark \ref{rem:Fvanishing} immediately imply that $\int_k A(\Gamma) =0$ unless each internal vertices of $\Gamma$ is decorated by some $a^j$, or $\nu$. If $\Gamma$ has loop order (first Betti number) more than one, then $A(\Gamma)$ is of degree $\geq 2k+1$, and the integral zero by degree reasons.
		If $\Gamma$ is a tree, then by degree reasons at most one leaf can have decorations of degree 2. Hence at least one leaf can only have decoration of degree 1, and hence the integral is zero by Lemma \ref{lem:vanishing} above. 
		
		That leaves loop order $1$. 
		If $\Gamma$ is of loop order 1, it is expressible as a simple loop with (possibly) several trees connected to each vertex of the loop.
		Again by degree reasons and Lemma \ref{lem:vanishing} the only case that yields a potentially nonvanishing integral are 1-loop graphs of the following form
		\[
		\Gamma=
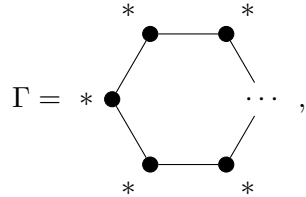
		\begin{tikzpicture}[baseline=-.65ex]
		\node (v1) at (0:1) {$\cdots$};
		\node[int, label=60:{$*$}] (v2) at (60:1) {};
		\node[int, label=120:{$*$}] (v3) at (120:1) {};
		\node[int, label=180:{$*$}] (v4) at (180:1) {};
		\node[int, label=240:{$*$}] (v5) at (240:1) {};
		\node[int, label=300:{$*$}] (v6) at (300:1) {};
		\draw (v1) edge (v2) edge (v6) (v3) edge (v2) edge (v4) (v5) edge (v6) edge (v4); 
		\end{tikzpicture}\, ,
		\]
		where $*$ is a placeholder for some $a^j$.
		Now we claim that the $j$ has to be constant on the graph, i.e., all $*$ are the same $a^j$.
		Suppose they are not. Then if two $a^i$- and $a^j$-decorated vertices are next to each other, with $i\neq j$, $i,j\geq 2$, then $A(\Gamma)=0$ as in Remark \ref{rem:Fvanishing}.
		
		If an $a^1$-decorated vertex is next to an $a^i$-decorated with $i\geq 2$ the edge between the vertices has to be replaced in the integral (essentially) by the propagator $\omega$ on stratum $B_1C_2$, see \eqref{equ:omega_B1C2}.
		In particular there is no dependence on the $a^i$-decorated vertex.
		That means the integral has a sub-integral of the second type of Lemma \ref{lem:vanishing}, and is hence zero.
		Hence we have reduced our considerations to $\Gamma$ of the following form, with all decorations the same.
		
		\[
		\Gamma=
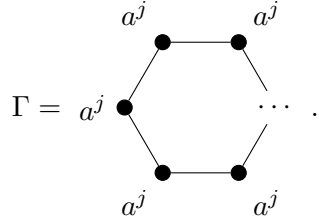
		\begin{tikzpicture}[baseline=-.65ex]
		\node (v1) at (0:1) {$\cdots$};
		\node[int, label=60:{$a^j$}] (v2) at (60:1) {};
		\node[int, label=120:{$a^j$}] (v3) at (120:1) {};
		\node[int, label=180:{$a^j$}] (v4) at (180:1) {};
		\node[int, label=240:{$a^j$}] (v5) at (240:1) {};
		\node[int, label=300:{$a^j$}] (v6) at (300:1) {};
		\draw (v1) edge (v2) edge (v6) (v3) edge (v2) edge (v4) (v5) edge (v6) edge (v4); 
		\end{tikzpicture}\, .
		\]
		By symmetry reasons, such graphs are zero for even number of vertices $k$.
		Otherwise the relevant integral is an integral over the cube $[0,1]^k$ representing longitudinal configurations of $k$ points on the $j$-th handle,
		\[
		\int_{[0,1]^k} \alpha ,
		\]
		where the precise form of the $k$-form $\alpha$ will not be relevant.
		We consider the automorphism of the cube
		\[
		I : (t_1,\dots,t_k) \mapsto (1-t_1,\dots,1-t_k).
		\]
		One verifies (cf. \eqref{equ:ajdef}) that the pullback $I^*$ sends the forms $a_j$ to $-a_j$ and 
		the propagator $\omega$ to $-\omega$ (cf.\eqref{equ:omega_B12} and  \eqref{equ:omega_C12}).
		The net effect on the integrand is hence $I^*\alpha=(-1)^{2k}\alpha=\alpha$.
		Furthermore, $I$ changes the orientation of the cube since $k$ is odd. Hence
		\[
		\int_{[0,1]^k} \alpha = (-1)^{2k}\int_{[0,1]^k} I^* \alpha
		=
		\int_{[0,1]^k} I^* \alpha = -\int_{[0,1]^k} \alpha =0
		\]
		and we are done.
	\end{proof}
	
    \subsection{An example computation}\label{sec:integral_example}
To provide an explicit example of the fiber integral, and to fill some details in the proof of Proposition \ref{prop:Fqiso}, we will compute explicitly the element 
\[
F(\Gamma_j)\in \HH_g(1),
\]
for 
\[
\Gamma_j =
\begin{tikzpicture}[baseline=-.65ex]
\node[ext] (v) at (0,0) {$\scriptstyle 1$};
\node[int, label=90:{$\scriptstyle a^jb^j$}] (i) at (0,.7) {};
\draw (v) edge (i);
\end{tikzpicture}
\in \BVGraphs_{\Sigma_g}(1).
\]

First we have $F(\Gamma_1)=0$, since $a^1b^1=\nu$ in $\HH_g(1)$, and by Lemma \ref{lem:vanishing}.
So we only consider $j=2,\dots, g$.
We have to consider the external vertex $1$ separately on each stratum. The internal vertex is always confined to the $j$-th handle, since $a^jb^j$ is supported there.
\begin{itemize}
\item On the bulk stratum the integral to evaluate becomes (cf. \eqref{equ:omega_C2})
\[
(\omega_{1\underline j})
\int_0^1 dt_j^{(2)}
(1-t_j^{(2)})
+(\omega_{1\overline j})
\int_0^1 dt_j^{(2)}
t_j^{(2)}
-\frac 1 2 \pi_1^*f_1
=
\frac 1 2 (\omega_{1\underline j}-\omega_{1\overline j}-\omega_{1\underline 1})
\]
\item On the first handle we get (cf. \eqref{equ:omega_B1C2})
\[
(-\frac 1 2 +t_1^{(1)}) \omega_{1*}.
\]
\item On the $j$-th handle we get (cf. \eqref{equ:omega_C12})
\[
\omega_{1*}\left(
\int_0^{t_j^{(1)}}dt_j^{(2)}(-t_j^{(2)})
+
\int{t_j^{(1)}}^1dt_j^{(2)}(1-t_j^{(2)})
\right)
= (\omega_{1*})(\frac 1 2 -t_j^{(1)}).
\]
\item On the other handles the integral is zero.
\end{itemize}
Overall, we see that we get a form that satisfies 
\[
dF(\Gamma_j)= \nu - a^jb^j = F(d\Gamma).
\]

We also note that the form 
\[
\alpha=2\sum_{j=2}^g F(\Gamma_j)
\]
hence satisfies 
\begin{equation}\label{equ:dalpha_example}
d\alpha = (2g-2)\nu - 2 \sum_{j=2}^g a^jb^j.
\end{equation}

\subsection{The case $g=0$}\label{sec:g0case}
The special case of a sphere, $g=0$, can be handled by a similar, but much simpler treatment.
Here we use as the model for the configuration space of points on the sphere
\[
(\FFM_2)' \times_{\FFM_{C}}\FFM_2,
\]
where the $(-)'$ on the left shall indicate the we use the involution $I:\FFM_C\to \FFM_C^{op}$ of \eqref{equ:IFMC} to transform the natural left $\FM_{C}$-module structure to a right $\FM_{C}$-module structure.
We shall think of the resulting model for the configuration space as two spheres, the left-hand sphere and the right-hand sphere, connecting by a thin ``handle'', in which points can ``move''.
Dually, our model for this configuration space is the fat totalization  
\[
\HH_0\coloneqq (\BVc)'\hotimes_{\BVc_C} \BVc. 
\]
Similar to the above, we will construct a zigzag of quasi-isomorphisms 
\[
\Mo_{0} \leftarrow \BVGraphs_{S^2}\to \HH_0.
\]
To this end we need the following ingredients.
First, our representatives for the cohomology classes of the sphere in $\HH_0(1)$ will be $1$ and $\nu$, the latter defined on the three strata as follows.
\begin{itemize}
	\item On the strata for which the point is on the left or right sphere, we set $\nu=0$.
	\item In the handle, we set $\nu=dt \omega_{1*}$, where $t$ is the coordinate along the handle.
\end{itemize}

Since we have framings, we can provide a primitive $\eta\in \HH_0(1)$ of $\nu$.
This is defined on the three strata as follows.
\begin{itemize}
	\item On the strata for which the point is on the left or right sphere, we set 
	\[
	\eta = 
	\theta_1 \in \BVc(1)
	\]
	\item On the handle, we set 
	\[
	\eta = \theta_1 - 2 (1-t) \omega_{1*} .
	\]
\end{itemize}
We note that the presence of the involution at the left end of the handle ensures that $\eta$ is indeed continuous.
Furthermore, it is clear that $d\eta=2\nu$, as desired.

Next, we define a ``propagator'' $\omega\in \HH_0(2)$.
There are 11 different strata to consider, which we discuss alongside the definition of $\omega$ as follows.
\begin{itemize}
	\item If both points are in one (the same) sphere, we set 
	\[
	\omega=\omega_{12}
	\in \BVc(1) \, .
	\]
	If both points are on spheres, but one on the left and the other on the right, we set $\omega=0$. (4 strata)
	\item If point 1 is on the right-hand sphere and point 2 is on the handle, we set 
	\[
	\omega= t^{(2)}\omega_{2*},
	\]
	where $t^{(2)}$ is the coordinate of the point 2 on the handle, with the end $t^{(2)}=0$ corresponding to the left-hand end, and $t^{(2)}=1$ to the right-hand end.
	If point 1 is on the left-hand sphere instead, we set 
	\[
	\omega= (t^{(2)}-1) \pi_2^*f.
	\]
	Note that in either case $\omega$ does not depend on the position of point $1$. (This essentially has to be the case since $\FM_2(1)=*$.
	Furthermore, note that at the left end of the handle we take the involution \eqref{equ:IBVc1} to create a right module structure from a left module structure. This involution changes the sign of $f$, and hence we get the somewhat odd sign in the formula.
	If the roles of points 1 and 2 are interchanged, we just define $\omega$ symmetrically.
	\item Suppose both points are on the handle, with coordinates along the handle $t^{(1)}<t^{(2)}$.
	Then we set 
	\[
	\omega=t^{(1)} \omega_{1*} -(1-t^{(2)}) \omega_{2*}
	\]
	Again we extend this symmetrically to the case $t^{(1)}>t^{(2)}$.
	Together, we hence have the formula
	\begin{equation}\label{equ:omega_12_g0}
	\omega=
	(t^{(1)}  - 1_{t^{(1)}>t^{(2)}}) \omega_{1*}
	+
	(t^{(2)}  - 1_{t^{(2)}>t^{(1)}}) \omega_{1*}.
	\end{equation}
	\item Finally, suppose both points are on the handle, and infinitesimally close.
	Then we define 
	\[
	\omega= \omega_{12} -(1-t)(\omega_{1*} + \omega_{2*}),
	\]
	
\end{itemize}

One then has.
\begin{lemma}\label{lem:prop0}
	The element $\omega\in \HH_0(2)$ is well-defined and satisfies
	\[
	d\omega = \pi_1^*\nu + \pi_2^*\nu. 
	\]
\end{lemma}
\begin{proof}
	This is shown by a straightforward computation, analogous to the proof is Lemma \ref{lem:prop} above.
\end{proof}

We can hence define a map $A:\BVGra_{H_0}\to \HH_0'$ as in Lemma \ref{lem:def-A}.
Then we define the map $\BVGraphs_{S^2} \to \HH_0$ just as in the case $g\geq 1$ above, see section \ref{sec:themapF}, by sending a graph $\Gamma\in \Graphs_{S^2}(r)$ with $k$ internal vertices to the fiber integral
\[
\int_k A(\Gamma).
\]
The resulting map is in fact relatively trivial due to the following observations:
\begin{itemize}
	\item As before, if $\Gamma$ has an internal vertex without $\nu$-decoration, then the fiber integral is zero.
	Hence we only need to consider graphs all of whose internal vertices are decorated by $\nu$'s.
	\item If such a graph has an edge between two internal vertices, then the resulting integral is also zero by degree reasons:,There are at least five from degrees depending on the 4 degrees of freedom of the two vertices the edge connects. 
\end{itemize}
By these observations it immediately follows that the partition function is trivial, i.e., given as in \eqref{equ:Ztriv}.
Overall, we can then proceed as in section \ref{sec:themapF} and complete the proof of Theorem \ref{thm:main} for $g=0$.

\begin{remark}
In this case it is possible (though not necessary) to fully evaluate the fiber integrals above.
In fact, by the remarks above the most general (internally connected) diagram which is not sent to zero is of the form
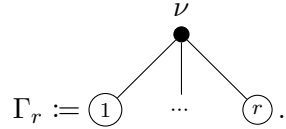
\[
\Gamma_r\coloneqq 
\begin{tikzpicture}[baseline=-.65ex]
\node[int,label=90:{$\nu$}] (i) at (0,1) {};
\node[ext] (v1) at (-1,0) {$\scriptstyle 1$}; 
\node (v2) at (0,0) {$\scriptstyle \cdots$}; 
\node[ext] (v3) at (1,0) {$\scriptstyle r$}; 
\draw (i) edge (v1) edge (v2) edge (v3);
\end{tikzpicture}\, .
\]
This is also easily evaluated given the explicit formula \eqref{equ:omega_12_g0} for the propagator.
The resulting form is zero when not all $r$ points are on the handle. Say (w.l.o.g.) that they have coordinates $t^{(1)}\leq \cdots\leq  t^{(r)}$.
Then the graph $\Gamma$ is sent to a form 
\[
f(t^{(1)}, \cdots, t^{(r)}) \pi_1^*f\pi_2^*f \cdots \pi_r^*f,
\]
where (cf. \eqref{equ:omega_12_g0})
\begin{align*}
f(t^{(1)}, \cdots, t^{(r)}) 
&=
\int_0^1
\prod_{j=1}^r
(t^{(j)}-1_{t<t^{(j)}})
dt
\\
&=
\sum_{j=0}^r 
(t^{(j+1)}-t^{(j)})
\left(\prod_{k=1}^{j}t^{(k)}\right)\left( \prod_{k=j+1}^r (t^{(k)}-1) \right),
\end{align*}
where we set $t^{(0)}\coloneqq 0$ and $t^{(r+1)}\coloneqq 1$.
In particular, for $r=1$, we get 
\[
f(t^{(1)}) = t^{(1)}(t^{(1)}-1) + (1-t^{(1)})t^{(1)}=0.
\]
It hence follows that Lemma \ref{lem:vanishing} continues to hold in the $g=0$-setting.
\end{remark}

	\section{Alternative interpretation of the results -- informal remarks}
	\subsection{Physical interpretation and the partition function of the Poisson Sigma Model on surfaces}
	Recall that the definition of $\Graphs_{\Sigma_g}$ depends on a choice of partition function, i.e., a Maurer--Cartan element $Z$, which enters into the differential. Let us make this dependence explicit and write 
	$\Graphs_{\Sigma_g}^Z$.
	Note that in \cite{CamposWillwacher2016} a specific such $Z=Z_{CW}$ is constructed so that $\Graphs_{\Sigma_g}^Z$ is a dg Hopf comodule model for $\FM_{\Sigma_g}$.
	In this paper we have shown that $\Graphs_{\Sigma_g}^{Z_{triv}}$, with $Z_{triv}$ the trivial Maurer--Cartan element \eqref{equ:Ztriv} is also a model for $\FM_{\Sigma_g}$. Hence we conclude in particular that we have a quasi-isomorphism 
	\[
	\BVGraphs_{\Sigma_g}^{Z_{triv}} \simeq \BVGraphs_{\Sigma_g}^{Z_{CW}},
	\]
	given by some zigzag.
	A priori we cannot conclude from this that $Z_{triv}$ is gauge equivalent to $Z_{CW}$. However, the deformation theory of $\BVGraphs_{\Sigma_g}$ is computed in \cite{Felder2019}, and it is shown that the graph complex dual to $G_{H_g}$ acts on $\BVGraphs_{\Sigma_g}$ in a faithful manner, up to homotopy.
	From this combined with our result we are then able to conclude that in fact $Z_{CW}$ and $Z_{triv}$ are gauge equivalent.
	
	We remark that our computation can hence physically be interpreted as a computation of the partition function of the Poisson $\sigma$-model (PSM) on the surfaces $\Sigma_g$. More informally, our specific method of computation can be seen as an algebro-topological version of working in a singular gauge akin to \cite{BCM2012}.

	\subsection{Higher genus associators}
	Drinfeld associators are algebraic objects defined by Drinfeld \cite{Drinfeld90} in connection with the study of quantum groups, and the absolute Galois group $\mathrm{Gal}(\bar \Q/\Q)$.
	Topologically speaking, we can understand a Drinfeld associator as a homotopy class of formality morphisms of the (framed) little disks operad.
	Similarly, higher genus versions of associators have been defined in the literature.	
	Topologically, we may define a genus $g$ associator to be a homotopy class of quasi-isomorphisms of $(\Omega(\FFM_{2}),\Omega(\FFM_{\Sigma_g}))$ and $(\BVc,\Mo_g)$.
	
	The main result of this paper can then be restated as saying that we provide a procedure to extend a Drinfeld associator to a genus $g$ associator.
	
	The set of Drinfeld associators is a torsor for the Grothendieck-Teichmüller group $\GRT$. The latter group may be defined topologically as the group of connected components of the homotopy automorphism group of a model of the framed little disks operad.
	To be concrete, $\GRT$ acts on the completed parenthesized chord diagrams model, cf. \cite{Tamarkin2002, FresseBook}.
	One can check similarly to Lemma \ref{lem:barnatan_cyc} that this action actually preserves the cyclic structure. 
		
	Similarly, on may define the higher genus Grothendieck-Teichm\"uller group $\GRT_g$ as the homotopy automorphisms of the pair $(\BVc,\Mo_g)$.
	One then has a forgetful map $\GRT_g\to \GRT$.
	It has shown by B. Enriquez \cite[Propositions 2.20, 4.7]{Enriquez2014} that this map has a one sided inverse for $g=1$.\footnote{Enriquez in fact has a slightly different definition of $\GRT_1$, but we shall ignore the difference, see also \cite{Felder2019}.}
	
	Now note that our construction of $\Mo_g$ out of the cyclic Hopf cooperad $\BVc$ was almost functorial in $\BVc$. The only non-functorial piece is that we fixed an augmentation dual to the choice of basepoint over which to take the fiber in \eqref{equ:FFM21qdef}.
	Unfortunately, the augmentation is generally not preserved by the $\GRT$ action -- except in the case $g=1$, where the basepoint can be taken to be the unit in $\FFM_2(1)$.
	Hence in that case we do obtain a $\GRT$-action by functoriality and hence a map $\GRT\to \GRT_{g=1}$, reproducing Enriquez' result.

	
	\printbibliography
	
	\listoftodos
	\todototoc
	
\end{document}

%% file: preamble.tex
\usepackage[T1]{fontenc}
\usepackage[utf8]{inputenc}
\usepackage[english]{babel}
\usepackage{lmodern}
\usepackage{exscale}
\usepackage[babel]{microtype}
\usepackage{amsmath, amssymb, mathtools, mathrsfs}
\usepackage{xstring}
\usepackage{xparse}
\usepackage{graphicx}
\usepackage{xcolor}
\usepackage{array, booktabs}
\usepackage[inline]{enumitem}
\usepackage{subcaption}

\usepackage{csquotes}
\usepackage[style=alphabetic, sorting=nyt, maxnames=10, maxalphanames=5]{biblatex}

\usepackage{tikz}
\usepackage{tikz-cd}
\usetikzlibrary{3d,arrows,calc,decorations,decorations.pathmorphing,decorations.pathreplacing,matrix,positioning,shapes,through,fit}
\usepackage[colorinlistoftodos, disable]{todonotes}

\usepackage[numbered]{bookmark}
\usepackage{amsthm, thmtools}
\usepackage[xspace]{ellipsis}

\renewbibmacro*{publisher+location+date}{%
  \printlist{publisher}%
  \setunit*{\addcomma\space}%
  \printlist{location}%
  \setunit*{\addcomma\space}%
  \usebibmacro{date}%
  \newunit}

\numberwithin{figure}{section}
\numberwithin{equation}{section}
\numberwithin{table}{section}

\numberwithin{mycounter}{section}
\NewDocumentCommand{\declthm}{m m O{plain}}{
  \declaretheorem[
  name = #2,
  sibling = mycounter,
  style = #3
  ]
  {#1}
}
\declthm{theorem}{Theorem}
\declthm{proposition}{Proposition}
\declthm{corollary}{Corollary}
\declthm{lemma}{Lemma}
\declthm{definition}{Definition}[definition]
\declthm{example}{Example}[remark]
\declthm{remark}{Remark}[remark]
\declthm{conjecture}{Conjecture}
\newtheorem{thmA}{Theorem}



%% file: math.tex
\newcommand{\Q}{\mathbb{Q}}
\newcommand{\R}{\mathbb{R}}

\newcommand{\fiber}{\mathrm{fiber}}

\DeclareMathOperator{\MC}{MC}
\newcommand{\op}{\mathsf}

\newcommand{\p}{\partial}
\newcommand{\Conf}{\mathrm{Conf}}

\newcommand{\FM}{\mathsf{FM}}

\newcommand{\Gra}{\mathsf{Gra}}
\newcommand{\BVGra}{\mathsf{BVGra}}
\newcommand{\Graphs}{\mathsf{Graphs}}

\newcommand{\hotimes}{\mathbin{\hat{\otimes}}}

\NewDocumentCommand{\qiso}{s}{
  \IfBooleanTF #1
  {\xleftarrow{\sim}}
  {\xrightarrow{\sim}}%
}

\tikzset{
  ext/.style={circle, draw, fill=white, inner sep=2pt},
  int/.style={circle,draw,fill,inner sep=2pt},nil/.style={inner sep=2pt},
  unkv/.style={circle, fill = gray, draw, inner sep = 2pt}
}

\usepackage{accents}

\newcommand{\beq}[1]{\begin{equation}
\label{#1}
}
\newcommand{\eeq}{\end{equation}}

